\documentclass[11pt, reqno]{amsart}
\usepackage{amsmath, euscript}
\usepackage{mathtools}
\usepackage{}
\usepackage{graphicx}
\usepackage{xcolor}
\definecolor{mutedcyan}{RGB}{68,114,196}
\definecolor{othermutedcyan}{RGB}{70,103,150}
\usepackage{amsfonts}
\usepackage{amsthm}
\usepackage[colorlinks=true, allcolors=othermutedcyan]{hyperref}
\usepackage{cleveref}
\usepackage{newlfont}
\usepackage{amscd}
\usepackage{amsgen}
\usepackage{amssymb} 
\usepackage{mathrsfs}	
\usepackage{longtable}
\usepackage{extarrows}
\usepackage{tikz}
\usepackage{tikz-cd}
\usepackage{verbatim}
\numberwithin{equation}{section}
\usepackage[all]{xy}
\usepackage{color}
\usepackage{amssymb}
\usepackage[left=3.2cm,right=3.2cm, top=3.25cm, bottom=3.25cm]{geometry}
\usepackage{tikz}
\usepackage{tikz-cd}
\usepackage{mathtools}
\usepackage{bm} 
\usepackage{dynkin-diagrams}
\usetikzlibrary{matrix,shapes,arrows,decorations.pathmorphing}
\usepackage{calligra}
\usepackage{mathrsfs}
\usepackage{enumitem} 
\usepackage{tgpagella} 
\usepackage[parfill]{parskip} 
\usepackage{float}
\restylefloat{table}
\usepackage{ytableau}
\usepackage{adjustbox}
\usetikzlibrary{calc,matrix,positioning}
\usepackage{xargs}                     
\usepackage{todonotes}
\usepackage[dvipsnames]{xcolor}
\usepackage{xcolor}
\usetikzlibrary{arrows.meta}

\usepackage{listings}
\usepackage{tcolorbox} 
\definecolor{codegray}{gray}{0.9}
\definecolor{background}{RGB}{250,250,250}
\definecolor{string}{RGB}{214,157,133}
\definecolor{comment}{RGB}{87,166,74}
\definecolor{identifier}{RGB}{36,36,36}
\definecolor{keyword}{rgb}{0.0, 0.2, 0.8}        
\definecolor{type}{rgb}{0.5, 0.0, 0.5}
\definecolor{function}{rgb}{0.0, 0.5, 0.5}

\lstdefinelanguage{macaulay2}{
  sensitive=true,
  morecomment=[l]{--},
  morestring=[b]",
  keywords={if, then, else, for, from, to, do, while, return, break, continue, try, catch, when, new, in, is},
  keywordstyle=\color{keyword}\bfseries,
  morekeywords=[2]{ZZ, QQ, RR, CC, Matrix, Ideal, Ring, Module, List, Set, HashTable, MutableHashTable, String, Option, Symbol},
  keywordstyle=[2]\color{type}\bfseries,
  morekeywords=[3]{matrix, mutableMatrix, ideal, ker, dim, rank, target, append, entries, flatten, toList, set, product, sum, needsPackage, print, apply, select, toSequence, contract, dual, sub, mingens, source, map, transpose},
  keywordstyle=[3]\color{function},
}

\lstdefinestyle{macaulay2}{
  language=Macaulay2,
  basicstyle=\ttfamily\small,
  commentstyle=\color{comment}\itshape,
  stringstyle=\color{string},
  numbers=left,
  numberstyle=\tiny\color{gray},
  stepnumber=1,
  numbersep=8pt,
  showstringspaces=false,
  breaklines=true,
  breakatwhitespace=true,
  tabsize=4,
  frame=single,
  rulecolor=\color{gray!20},
  backgroundcolor=\color{background},
  captionpos=b,
  keepspaces=true
}

\lstdefinestyle{python}{
  language=Python,
  basicstyle=\ttfamily\small,
  commentstyle=\color{comment}\itshape,
  stringstyle=\color{string},
  numbers=left,
  numberstyle=\tiny\color{gray},
  stepnumber=1,
  numbersep=8pt,
  showstringspaces=false,
  breaklines=true,
  breakatwhitespace=true,
  tabsize=4,
  frame=single,
  rulecolor=\color{gray!20},
  backgroundcolor=\color{background},
  captionpos=b,
  keepspaces=true
}

\tikzset{ 
    table/.style={
        matrix of nodes,
        row sep=-\pgflinewidth,
        column sep=-\pgflinewidth,
        nodes={rectangle, text width=2.5em, text height = 1.5em, align=center},
        text depth=1.25ex,
        text height=2.5ex,
        nodes in empty cells
    },
}

\let\blb\mathbb

\def\CC{{\blb C}}

\def\GG{{\blb G}}

\def\LL{{\blb L}}

\def\PP{{\blb P}}

\def\WW{{\blb W}}

\def\ZZ{{\blb Z}}

\let\cal\mathcal

\def\Cc{{\cal C}}

\def\Ec{{\cal E}}

\def\Hc{{\cal H}}

\def\Oc{{\cal O}}
\def\Pc{{\cal P}}
\def\Qc{{\cal Q}}

\def\Uc{{\cal U}}

\def\Bl{\operatorname{Bl}}
\def\wt{\widetilde}

\def\arw{\longrightarrow}
\def\Hom{\operatorname{Hom}}

\def\Aut{\operatorname{Aut}}

\def\Ext{\operatorname{Ext}}
\def\rk{\operatorname{rk}}
\def\Crit{\operatorname{Crit}}

\def\Pf{\operatorname{Pf}}

\def\Tot{\operatorname{Tot}}

\newcommand{\proj}{\PP}

\DeclareMathOperator{\sHom}{\mathscr{H}\text{\kern -3pt {\calligra\large om}}\,}

\newcommand\quotient[2]{
        \mathchoice
            {
                \text{\raise1ex\hbox{$#1$}\Big/\lower1ex\hbox{$#2$}}%
            }
            {
                #1\,/\,#2
            }
            {
                #1\,/\,#2
            }
            {
                #1\,/\,#2
            }
    }

\makeatletter
\def\namedlabel#1#2{\begingroup
    #2%
    \def\@currentlabel{#2}
    \phantomsection\label{#1}\endgroup
}
\makeatother

\newcommand{\Y}{\mathcal{Y}}
\newcommand{\V}{\mathcal{V}}
\newcommand{\U}{\mathcal{U}}
\newcommand{\N}{\mathcal{N}}
\newcommand{\LG}{(\mathfrak{X},\omega)}
\newcommand{\rchar}{\C_R^*}
\newcommand{\W}{\mathcal{W}}
\newcommand{\Q}{\mathcal{Q}}
\newcommand{\til}{\tilde}
\newcommand{\wtil}{\widetilde}
\newcommand{\lam}{\lambda}
\newcommand{\I}{\mathcal{I}}
 \newcommand{\X}{\mathcal{X}}
\renewcommand{\P}{\mathbb{P}} 
\newcommand{\Z}{\mathbb{Z}}
\renewcommand{\O}{\mathcal{O}} 
\newcommand{\C}{\mathbb{C}} 
\newcommand{\blp}{\widetilde{\P}^8} 
\newcommand{\bly}{\widetilde{Y}} 
\newcommand{\kb}{U/\Gam} 
\newcommand{\bunNeg}{\mathcal{U}^\vee} 
\newcommand{\bunPos}{\mathcal{S}^\vee} 
\renewcommand{\S}{\mathcal{S}}
\newcommand{\GamEl}{\begin{pmatrix}
\alpha & \beta \\
0 & \delta
\end{pmatrix}}
\newcommand{\LGpos}{(U/\Gam, \omega)}
\newcommand{\coordsOfV}{(B,\underline v,\underline w)}
\newcommand{\coordsOfU}{(B, v,\ul w)}
\newcommand{\G}{\mathcal G}
\newcommand{\chk}{\check}
\newcommand{\ul}{\underline} 
\renewcommand{\sp}{\hspace{1 mm}}
\newcommand{\eq}{\equiv}
\newcommand{\Gam}{\Gamma}
\newcommand{\al}{\alpha}
\newcommand{\del}{\delta}
\newcommand{\Sym}{\mathrm{Sym}}
\newcommand{\F}{\mathcal{F}}
\newcommand{\E}{\mathcal{E}}
\newcommand{\potneg}{\chk s_-}
\newcommand{\forallsp}{\sp\forall\sp}
\renewcommand{\H}{\mathcal{H}}
\newcommand{\derTensor}{\stackrel{L}{\otimes}}
\newcommand{\st}{\sp|\sp}
\newcommand{\matTwo}[4]{\begin{pmatrix} #1 & #2 \\ #3 & #4 \end{pmatrix}}
\newcommand{\colTwo}[2]{\begin{pmatrix} #1 \\ #2 \end{pmatrix}}
\newcommand{\Srep}{\mathrm{S}}

\newcommandx{\prajwal}[2][1=]
{\todo[linecolor=blue!20,backgroundcolor=blue!20,bordercolor=blue!20, size=\scriptsize, size=\scriptsize, #1]{#2}}
\newcommandx{\marco}[2][1=]
{\todo[linecolor=red!20,backgroundcolor=red!20,bordercolor=red!20, size=\scriptsize, #1]{#2}}
\newcommandx{\michal}[2][1=] {\todo[linecolor=green!20,backgroundcolor=green!20,bordercolor=green!20, size=\scriptsize, #1]{#2}}

\theoremstyle{definition}

\newtheorem{theorem}{Theorem}[section]
\newtheorem{proposition}[theorem]{Proposition}
\newtheorem{lemma}[theorem]{Lemma}
\newtheorem{corollary}[theorem]{Corollary}
\newtheorem{definition}[theorem]{Definition}
\newtheorem{remark}[theorem]{Remark}
\newtheorem{observation}[theorem]{Observation}
\newtheorem{example}[theorem]{Example}
\newtheorem{conjecture}[theorem]{Conjecture}

\newtheorem{mainThm}{Theorem}
\newtheorem{mainConj}[mainThm]{Conjecture}

\title{The Last Picard Rank 1 Double--Mirror Calabi--Yau Pair?}

\author[M. Kapustka]{Micha\l\ Kapustka}
\address{M. Kapustka:
IMPAN \\ 
Św. Tomasza 30 \\
30-056\\
Krakow, Poland}
\email{mkapustka@impan.pl}

\author[M. Rampazzo]{Marco Rampazzo}
\address{M. Rampazzo: 
Yau Mathematical Sciences Center \\ 
Tsinghua University \\ 
Shuangqing Complex Bldg.\\ 
Haidian district, Beijing, China.}
\email{marcorampazzo@tsinghua.edu.cn, marco.rampazzo.90@icloud.com}

\author[P. Samal]{Prajwal Samal}
\address{P. Samal: 
IMPAN \\ 
Św. Tomasza 30 \\
30-056\\
Krakow, Poland}
\email{psamal@impan.pl, prajwal.samal@gmail.com}


\begin{document}

\begin{abstract}
We consider a certain pair of families of Picard rank 1 Calabi--Yau threefolds, that have appeared earlier in mathematical literature in unrelated contexts: the family $\mathcal{X}$ of degree 33 threefolds in $G(2,6)$ (constructed by Miura) and the family $\mathcal{Y}$ of arithmetically Gorenstein degree 21 threefolds in $\mathbb{P}^8$ (constructed by Schenck--Stillman--Yuan). After establishing a natural geometric correspondence between their general members, we go on to show that any pair of corresponding threefolds 
in these families satisfy certain classical dualities. We moreover discover that their geometries may be related by a mathematical gauged linear sigma model, using which we prove that they are derived equivalent. This settles a conjecture of Miura, who predicted the existence of non-trivial Fourier--Mukai partners to members of $\mathcal{X}$, based on a study of its mirror moduli. This conjecture was also formulated later by Gerhardus--Jockers, in a physical context. In fact, starting from the other family $\mathcal{Y}$,
we conjecturally arrive at the same mirror moduli. Thus, such a pair is a new, and quite possibly the last, addition to the small list of deformation families of non-birational, double-mirror Calabi--Yau threefolds having Picard rank 1.

\end{abstract}

\maketitle

\setcounter{tocdepth}{1}
\tableofcontents

\section{Introduction}

While the derived category of coherent sheaves has proven to be an invariant up to isomorphism for smooth Fano and canonically polarized varieties \cite{bondalorlovreconstruction}, it has been known for many years that this property does not extend to Calabi--Yau varieties. Due to this, a systematic study of the interplay between several different invariants, such as Hodge structure, derived category and class in the Grothendieck ring, has been carried out, leading to the failure of categorical Torelli theorems for Calabi--Yau varieties \cite{ottemrennemo, borisovcaldararuperry} and examples of \emph{non-trivial Fourier--Mukai partners}. The latter consist in pairs of non-birational Calabi--Yau varieties which are derived equivalent\footnote{Here, the terminology "non-trivial" is justified by the fact that in dimension three or lower, birational Calabi--Yau varieties are derived equivalent (see \cite{bridgeland_flops} for the threefold case).}.

The construction and study of such pairs is of considerable interest, as they represent testing grounds to investigate the role of the derived category as an invariant for Calabi-Yau threefolds, especially in contrast to the birational class treated as an invariant. Another reason for their consideration is their surprising relevance in the context of mirror symmetry, where one often finds that they have the same Hodge theoretic mirror. We refer to such pairs as \textit{double mirror}. In fact, usually one deduces a given Calabi-Yau pair to be double mirror prior to proving the subsequent mirror symmetric expectation of them being derived equivalent. That said, we should point out that except for the case of Calabi-Yau complete intersections in toric varieties, most double mirror computations are conjectural, in that, one deduces relevant properties of the mirror without actually constructing it nor proving its existence. We will use this phrase in the same spirit. Such computations are somewhat tractable for a Calabi-Yau variety that has Picard rank 1, as its mirror moduli is expected to be one dimensional. Our focus, in particular, will be on this case.

Up to date, pairs of non-trivial Fourier-Mukai partner Calabi--Yau threefolds having Picard rank 1 are quite rare. To the authors' knowledge, only three examples (up to deformation) are known.  

\subsubsection*{The Pfaffian--Grassmannian pair}
    Consider a vector space $V_7\simeq \C^7$ and a general linear subspace $L\subset \wedge^2V_7^\vee$ of dimension 7. Then, the Grassmannian section $X\equiv G(2,V_7)\cap\P(L^\perp)$ and the Pfaffian section $Y\equiv \Pf(4,V_7^\vee)\cap\P(L)$ are non-birational Calabi-Yau threefolds having Picard rank 1. This pair was conjectured to be derived equivalent by Rødland \cite{rodland}, based on computations that indicated that they are double mirror. Subsequently, a proof was established by Borisov--Caldararu by explicitly constructing a Fourier--Mukai kernel \cite{borisovcaldararu}. A new proof, within the framework of window categories and variation of GIT, was later provided by Addington, Donovan and Segal \cite{Seg2014}. A degeneration of this family of pairs can be realized as follows: given the complete flag variety $G_2/B$ of type $G_2$, call $p_i$ the two projections to the $G_2$-Grassmannians $G_2/P_i$. In this notation, $G_2/P_1$ is a five-dimensional quadric and $G_2/P_2$ is the so-called ``$G_2$-Grassmannian''. Then, for a general section $\sigma$ of $p_1^*\Oc(1)\otimes p_2^*\Oc(1)$, the vanishing loci $Z(p_{i*}\sigma)$ are Calabi--Yau threefolds. This picture was studied in \cite{imou} in the context of the so-called ``$\LL$-equivalence'' and a proof of derived equivalence was given in \cite{kuznetsov_g2_cy3}.
    
\subsubsection*{The Hosono--Takagi pair}
        Given a vector space $V$, let $\eta: \Sym^2\PP(V)\to \PP(\Sym^2 V)$ be the natural morphism which acts as $\eta([v_1], [v_2]) = [v_1\otimes v_2 + v_2\otimes v_1]$. Also let $\zeta:\til Z\to H$ be a double cover of a determinantal quintic hypersurface $H\subset\P(\Sym^2V^\vee)$ (called the \emph{quintic symmetroid}). Then, for every subspace $L\subset \Sym^2 V^\vee$, one can consider the stack $X \coloneqq \eta^{-1}(\PP(L^\perp))$ and the variety $Y\coloneqq\zeta^{-1}(\P(L^\vee))$. In \cite{hosono-takagi}, Hosono and Takagi focus on the case of $\dim V = \dim L = 5$, where they show that the resulting spaces $X$ and $Y$ are derived equivalent, non-birational Calabi--Yau threefolds having Picard rank 1. This derived equivalence was later re-proved by Rennemo \cite{rennemo_hpd_sym2V} in the context of homological projective duality for $\Sym^2\PP(V)$, for general $V$.

\subsubsection*{The intersections of general translates of $G(2, 5)$}
    Consider the Grassmannian $G(2,5)$ under its Pl\"ucker embedding $G(2,5)\subset \P^{9}$. For a general automorphism $g\in\Aut(\PP^{9})$, the variety $X_g\equiv G(2, 5)\cap g(G(2, 5))$ is a Calabi--Yau threefold having Picard rank 1. A general pair $(X_g$, $X_{g^T})$ of such threefolds were shown to be double mirror in \cite{MS_Kap}. Subsequently, in \cite{ottemrennemo, borisovcaldararuperry}, using a generalized version of homological projective duality \cite{categorical_plucker_formula}, it was shown that such a pair are Fourier-Mukai partners, but not birational. A degeneration of this construction was considered in \cite{KapRam17, KapRam2025}. 

\subsection*{Relationship with physics} In mathematical physics, one expects double mirror Calabi--Yau varieties to appear as target spaces of some supersymmetric quantum field theories called \emph{gauged Landau-Ginzburg models}, which in turn appear as different limits of a \emph{gauged linear sigma model} (GLSM), connected by phase transitions \cite{hori}. In the original approach of the foundational paper \cite{witten_glsm}, a phase transition connects the quintic Calabi--Yau phase to an exotic, stacky geometry supported on a point. However, one can engineer more sophisticated examples, where geometric phases appear on both sides of the phase transitions. The physical properties of D-branes then translate to the mathematical expectation of the associated Calabi--Yau varieties being derived equivalent. This has been proven to be true in many examples.

\subsection{Results}
We consider two families of Picard rank 1 Calabi-Yau threefolds $\X$ and $\Y$ whose general members are non-birational, and relate them by constructing a natural correspondence that takes a general member $X$ of $\X$ and sends it to a general member $Y$ of $\Y$. The first family $\X$ consists of sub-varieties of $G(2,6)$, and was first constructed and studied in \cite{miuraMirrorSymmetry}. It also later appeared in \cite[Table 1, no. 7]{ionue_ito_miura_CY3s}, in the context of classification of Calabi-Yau threefolds given as zero loci of sections of homogeneous vector bundles on Grassmannians.  The second family $\Y$ first appeared in \cite{SSY}, in a strictly algebraic context. Its members are arithmetically Gorenstein Calabi-Yau threefolds of codimension 5 in $\P^8$. The central thesis of this paper is that these Calabi-Yau families are dual in the sense that they have the same mirror family. We showcase evidence for this double mirror duality both in the intersection-theoretic and homological senses.

For the intersection-theoretic side, we prove certain classical equivalences (Hodge equivalence and L-equivalence) between these families and moreover show that the mirror moduli spaces of both these families have the same Picard-Fuchs operator. The latter is done using a conjectural method of Batyrev using "small" toric degenerations of certain Fano varieties that are ambient to our Calabi-Yau varieties.

\begin{mainThm}[\Cref{sec:geometry}]\label{thm:thmA}
Let $\mathcal X$ be the family of Calabi-Yau threefolds appearing in \cite[Table 1, no. 7]{ionue_ito_miura_CY3s}. And $\mathcal Y$ be the family of arithmetically Gorenstein Calabi-Yau threefolds of codimension 5 in $\PP^8$ appearing in \cite[Section 3.1]{SSY}. Then the following holds: for a general member $X\in \mathcal X$ there exists a member $Y\in\mathcal Y$ which is Hodge-equivalent, $\LL$-equivalent to $X$, but not birational to $X$.
\end{mainThm}

For the homological side, we will construct a Fourier-Mukai functor between corresponding general members ($X$ and $Y$) of our families, by setting up a mathematical GLSM. This is an extensively developing technique, pioneered in \cite{Seg2011} and \cite{Seg2014}, motivated by the physics of \cite{hori}. Here, one replaces the geometries of the Calabi-Yau threefolds by certain birational Landau-Ginzburg models and then one constructs a functor at the level of B-brane categories of these models. Using this approach we arrive at the following result.

\begin{mainThm}[\Cref{thm:derivedEquivalence}]\label{thm:thmB}
    Let $X$ and $Y$ be threefolds as in \Cref{thm:thmA}. Then, there is an equivalence of triangulated categories $D^b(\operatorname{coh}(X))\simeq D^b(\operatorname{coh}(Y))$.
\end{mainThm}
In particular, our construction amounts to a new family of pairs of non-birational derived-equivalent Calabi-Yau threefolds having Picard rank 1. Moreover, based on the current status of enumeration of Calabi-Yau type operators \cite{database}, it seems that our pair is the last one of its kind. More precisely, to our knowledge, there is no additional Calabi-Yau type operator in the database having two MUM points, for which monodromy calculations conjecturally predict the existence of a new double mirror pair of Calabi-Yau threefolds having Picard rank one. 

A posteriori, this result also serves as a further validation of Batyrev's conjectural mirror construction \cite{batyrevToricDegen}, which has traditionally been used to predict that certain Calabi-Yau threefold pairs are double mirror. More concretely, it settles \cite[Conjecture 5.8]{miuraMirrorSymmetry} (see also \cite[Section 4]{Gerhardus_Jockers} and \cite{galkin_talk}), where Miura predicts the existence of a non-birational Fourier-Mukai partner to $X$ having appropriate numerical properties. His prediction was based on a (conjectural) computation of the Picard-Fuchs operator of the mirror moduli of $X$ and the discovery of the fact that this operator has two maximal unipotent monodromy (MUM) points. We verify these numerical properties of $Y$ as part of the proof of \Cref{thm:thmA}.

 A new and unexpected feature of our GLSM-based proof of \Cref{thm:thmB} is that we were forced to resolve the GIT quotient in the ``difficult" phase of the GLSM by performing a stacky blow-up, to obtain a critical locus having the desired geometry. In fact, for this reason, our GLSM relates $X$ to the blow-up of $Y$ along an elliptic curve (call this blow-up $\wtil Y$).
 The mysterious role of the elliptic curve is further indicated by the fact that the GIT quotient is ``stacky" along the fiber over the linear span of the elliptic curve (this fiber is exactly the center of the stacky blow-up).

Finally, in the light of above results and the discussion in the introduction, we are motivated to make the following conjecture. We state and discuss it in a more precise language in \Cref{sec:mirrorEquivalence}. 

\begin{mainConj}[\Cref{conj:doubleMirror}]\label[conjecture]{conj:conjectureC}
    There exists a family $\mathcal{Z}$ of Calabi-Yau threefolds having two MUM points, with 
    $\X$ and $\Y$ being Hodge theoretic mirrors to $\mathcal{Z}$ at these points, respectively. 
\end{mainConj}

\subsection{Plan of the Paper} Here we review the section-wise plan of the paper, while also highlighting the strategy of proof of the derived equivalence.

In \Cref{sec:geometry}, we will recall the family $\X$, construct the family $\Y$ as well as a geometric correspondence relating a general member of one family  with a suitable general member of the other family  by using the generalized roof $F(1,2,6)$. In particular, we will prove \Cref{thm:thmA}. 

In \Cref{sec:mirrorEquivalence}, we use the Batyrev's conjectural method of using "small" toric degenerations to find the Picard-Fuchs operator of the mirror moduli of $\Y$. We then compare it with that of $\X$ and as a result motivate \Cref{conj:conjectureC}. We also provide a candidate birational model for the common mirror family $\mathcal{Z}$. 

In \Cref{sec:LG}, we will recall the notions of Landau-Ginzburg models and B-branes on them. We will also recall Kn\"orrer periodicity, which serves as the "bridge" that we use to relate derived categories and B-brane categories.  

In \Cref{sec:GLSMsetup}, we will set up the GLSM that relates the geometry of $X$ to $\wtil Y$ (the blow-up of $Y$ along an elliptic curve) by realizing them as critical loci of two birational Kn\"orrer models (\Cref{def:knorrerModel}). Our starting point is the description of $X$ as the vanishing of a section of a homogeneous vector bundle. After a variation of GIT of the dual bundle, we arrive at a stacky GIT quotient, which we resolve by certain birational modifications. These modifications amount to a blow-up followed by an open sub-stack restriction. 

In \Cref{sec:constructionOfFunctor}, working within our GLSM setup, we will construct a natural fully-faithful functor from $D^b(\operatorname{coh}(X))$ to $D^b(\operatorname{coh}(\til{Y}))$. This is done in several steps, which all amount to the construction of a functor at the level of derived categories of the Kn\"orrer models (using a window sub-category of the central quotient stack) followed by the construction of its lifting to a functor between the B-brane categories of those models (also called "turning on the superpotential" ).

In \Cref{sec:proofOfEquiv}, we will prove that our functor in fact induces an equivalence between $D^b(\operatorname{coh}(X))$ to $D^b(\operatorname{coh}(Y))$, thus proving \Cref{thm:thmB}. We do this by showing that its image is semi-orthogonal to a copy of the category of the elliptic curve in $D^b(\operatorname{coh}(\til{Y}))$

\subsection*{Notations \& Conventions}
We work over $\C$. For a finite-dimensional vector space $V$, the projective space $\P(V)$ parameterizes one-dimensional subspaces of $V$ and the Grassmannian $G(k,V)$ parameterizes $k$-dimensional subspaces of $V$.

When we refer to a vector bundle, we usually mean its locally free sheaf of sections on the base scheme. For a vector bundle $\E$ over a base $B$, we denote its projectivization as $\P(\E)\equiv\operatorname{Proj}_B(\Sym^{\bullet}\E^\vee)$, parameterizing lines in the fibers of $\E$. So, if $\pi:\P(\E)\rightarrow B$ is its natural projection, then $\O_{\pi}(-1)\subset\pi^*\E$ is the tautological line bundle and $\pi_*\O_{\pi}(1)\simeq\E^\vee$. We denote the total space of $\E$ as $\Tot(\E)\equiv\operatorname{Spec}_B(\Sym^{\bullet}\E^\vee)$. For a section $s$ of $\E$, we use $Z(s)$ to denote its zero locus, with its natural scheme structure.

For any scheme or algebraic stack $\mathfrak{X}$, we use $D(\mathfrak{X})$ to denote $D^b(\operatorname{coh}(\mathfrak{X}))$, the bounded derived category of coherent sheaves on $\mathfrak{X}$. And for a Landau-Ginzburg model $(\mathfrak{X},\omega)$, we use $D(\mathfrak{X},\omega)$ to denote its category of B-branes, as defined in \Cref{sec:LG}.

\subsection*{Acknowledgments}
The authors gratefully acknowledge the contribution of Enrico Fatighenti, Luigi Martinelli and Giovanni Mongardi to this project. They also wish to thank Will Donovan, Mauricio Romo, and Ying Xie for useful and inspiring discussions. PS would like to thank Jørgen Rennemo for pedagogical and technical assistance, on the tools that went into the proof of \Cref{thm:thmB} (especially for pointing out \Cref{obs:cartesianSquare}), during the author's research visit at the University of Oslo. MR was supported by the Tsinghua-YMSC--Imperial College joint postdoctoral program in algebraic geometry. MK and MR were supported by the Polish National Sciences center project number 2024/55/B/ST1/02409. PS was supported by Polish National Sciences center project number
2021/43/O/ST1/02676.

\section{A Tale of two Calabi-Yau Threefolds}\label{sec:geometry}
The first titular character in our tale is the Calabi-Yau threefold \cite[Table 1, no. 7]{ionue_ito_miura_CY3s}. Let us recall its construction. 

\subsection{The Grassmannian Calabi-Yau Family $\X$}\label{sec:X}
Fix two vector spaces $V_6\simeq\CC^6$ and $L\simeq\CC^3$. Consider the Grassmannian $G(2, V_6)$ parameterizing two-dimensional subspaces of $V_6$. Denote by $\Uc$ and $\Q $, the universal and quotient bundle over it, respectively. They appear in the tautological sequence, written below.
\begin{equation*}
    0\arw \Uc\arw V_6\otimes\Oc \arw\Qc\arw 0.
\end{equation*}
Similarly, consider the projective space $\P(V_6)$ paramaterizing lines in $V_6$. Denote by $\wt\Q$, its quotient bundle, constructed under the identification $\P(V_6)\simeq G(1,V_6)$. Here, the tautological sequence is just the twisted Euler sequence, as written below. 
\begin{equation}
    0\arw \O(-h)\arw V_6\otimes\O \arw\wt\Q\arw 0.
\end{equation}
This shows, in particular, that $\wt\Q$ is just the tangent bundle twisted by $-h$, where $h$ denotes the hyperplane class of $\PP(V_6)$.

Denote by $X$, the vanishing locus of a general regular section $s\in H^0(\Uc^\vee(\xi)\oplus L\otimes\Oc(\xi))$. Then, $X$ is a smooth Calabi-Yau threefold of degree 33 having the following Hodge diamond \cite[Table 1, no. 7]{ionue_ito_miura_CY3s}.  As $s$ varies, $X$ varies in a family of threefolds, which we will denote by $\X$.

\begin{equation*}
    \begin{array}{ccccccc}
            &&& 1 &&& \\
            && 0 && 0 && \\
            & 0 && 1 && 0 & \\
            1 && 52 && 52 && 1 \\
            & 0 && 1 && 0 & \\
            && 0 && 0 && \\
            &&& 1 &&&
        \end{array}
\end{equation*}

Now consider the the flag variety $F(1, 2, V_6)$ parameterizing pairs consisting of a line and a two-space in $V_6$ such that the line is contained in the two-space. One has the following standard projections out of this flag variety, which endow it with two different projective bundle structures. 
\begin{equation}
    G(2,V_6)\xleftarrow{q}\P(\U(-\xi))\simeq F(1,2,V_6)\simeq\P(\wt\Q(-2h))\xrightarrow{p}\P(V_6)
\end{equation}
Here, $\xi$ denotes the hyperplane class of the Grassmannian (under its Plucker embedding) and $h$ denotes the hyperplane class of the projective space. By abuse of notation, we will denote their pullbacks to the flag variety by the same symbols. The choices of twists, appearing in these projectivizations, are such that both give rise to the same relative $\O(1)$ on the the flag variety. It is
\begin{equation}
    \O_p(1)=\O(h+\xi)=\O_q(1)
\end{equation}
Therefore, we have the following pushforward formulas
\begin{equation}
    \begin{aligned}
        q_*(\O(h))&=\U^\vee \quad &q_*(\O(\xi))&=\O(\xi)\\
        p_*(\O(h))&=\O(h)\quad &p_*(\O(\xi))&=\wt\Q^\vee(h)
    \end{aligned}
\end{equation}
Using these formulas and the natural isomorphisms between section spaces and Hom spaces of vector bundles, one deduces the following natural chain of isomorphisms.
\begin{equation}\label{eq:chainOfIsos}
\begin{aligned}
    H^0(G(2, V_6), \Uc^\vee(\xi)\oplus L\otimes\Oc(\xi))&\simeq H^0(F(1, 2, V_6), \Oc(h+\xi)\oplus L\otimes\Oc(\xi))\\
    &\simeq H^0(\PP(V_6), \wt\Qc^\vee(2h)\oplus L\otimes\wt\Qc^\vee(h))\\
    & \simeq \Hom_{\PP(V_6)}(\Oc(-h)\oplus L^\vee\otimes\Oc, \wt\Qc^\vee(h)).
\end{aligned}
\end{equation}

In this way, the section $s$ naturally induces a morphism $f_s:\Oc(-h)\oplus L^\vee\otimes\Oc \rightarrow \wt\Qc^\vee(h)$ of bundles on the projective space. Consider its rank 3 degeneracy locus:
\begin{equation}
    \wt Y\equiv  D_3(f_s)\equiv\{[\ul v]\in\PP(V_6) \ |\  f_s|_{[\ul v]} \text{ has rank } \leq 3 \}
\end{equation}

It is well-known \cite[Section 7.2]{lazarsfeld_positivity_II} that, given a general morphism of vector bundles $f\in\Hom_Z(\E, \F)$ on a smooth projective variety $Z$, such that $\E^\vee\otimes \F$ is globally generated, the dimension and singular locus of each of the degeneracy loci are given as follows.
\begin{equation}
\begin{aligned}
    \dim D_k(f) &= \dim Z - ( \rk E - k) (\rk F - k)\\
    \operatorname{Sing}D_k(f) &= D_{k-1}(f)
\end{aligned}
\end{equation}
Moreover, assuming $\rk \E \leq \rk \F$, a standard construction (see \cite[Section 7.2]{lazarsfeld_positivity_II} or \cite[Section 2.2]{tanturri_thesis} for a more recent perspective) allows one to resolve the first degeneracy locus $D_{\rk E - 1}(f)$ as the zero locus of a section of $\pi^*\F\otimes \O(H)$, where $\pi$ is the projective bundle map $\PP(\E) \arw Z$ and $H$ is its relative hyperplane class. Note that, under this notation we have: $\pi_*\Oc(H) \simeq \E^\vee$.

Therefore, one finds that $\wt Y$ is a threefold. Moreover, as $D_2(f_s)$ is empty, we know that $\wt Y$ is smooth and hence is isomorphic to its resolution, which sits inside $\PP(\Oc(-h)\oplus L^\vee\otimes\Oc)\xrightarrow{\pi}\PP(V_6)$ as the vanishing locus of a section of $\pi^*\wt\Qc^\vee(h+H)$. We also call such a resolution $\wt Y$. Now, the natural embedding of bundles $\O(-h)\oplus L^\vee\otimes\O\hookrightarrow (V_6\oplus L^\vee)\otimes \O$ on $\P(V_6)$, induces the embedding $\PP(\Oc(-h)\oplus L^\vee\otimes\Oc)\subset\P(V_6\oplus L^\vee)\times \P(V_6)$. Projecting to the first factor, we get exactly the blow-up morphism $\beta$ for the blow-up of $\P(V_6\oplus L^\vee)$ along $\P(L^\vee)$. Essentially, the two maps $\pi$ and $\beta$ may be seen as the resolution of the projection of $\PP(V_6\oplus L^\vee)$ from $\PP(L^\vee)$, as we depict in the following diagram:
\begin{equation*}
    \begin{tikzcd}[row sep = huge]
        & \Bl_{\PP(L^\vee)}\PP(V_6\oplus L^\vee) \simeq \PP(\Oc(-h)\oplus L^\vee\otimes\Oc) \ar[swap]{dl}{\beta} \ar{dr}{\pi} & \\
        \PP(V_6\oplus L^\vee) \ar[dashed]{rr} & & \PP(V_6)
    \end{tikzcd}
\end{equation*}
Under our notation, the relative hyperplane class of $\pi$ is the pullback of the hyperplane class $H$ of $\P(V_6\oplus L^\vee)$. By abuse of notation, we will refer to this pullback as $H$ again. Finally, we denote by $Y$, the image of $\wt Y$ through the blow-up morphism $\beta$.

\subsection{The Arithmetically Gorenstein Calabi-Yau Family $\Y$} \label{sec:Y}
It turns out, as we will see, that the variety $Y$ obtained above is a Calabi-Yau threefold and moreover is the second titular threefold in our tale. We will sometimes refer to it as the \textit{dual Calabi-Yau threefold} to $X$.

Let us first show that $Y$, embedded in $\P^8$, is described by  a Huneke-Ulrich ideal, as in \cite[Section 3.1]{SSY}.
\begin{proposition}\label[proposition]{prop:explicitDescriptionOfY}There exist a $6\times 6$ skewsymmetric matrix $M$ of linear forms in $\mathbb P(V_6\oplus L^\vee)$ and a $6\times 1$ matrix of linear forms $v$ cutting out $\mathbb P(L^\vee)\subset \mathbb P(V_6\oplus L^\vee)$, such that $Y\subset \mathbb P(V_6\oplus L^\vee)$ has the following description 
$$Y=\{[x]\in \mathbb P(V_6\oplus L^\vee) | M(x)v(x)=0, \Pf(M(x))=0 \}.$$
Moreover the exceptional divisor of the map $\beta_Y$ is described in $\mathbb P(L^\vee)\times \mathbb P(V_6)$ as the locus 
$$\{([l],[w])\in \mathbb P(L^{\vee})\times \mathbb P(V_6) | M(l,0) w=0\}.$$

\end{proposition}
\begin{proof}
Let us first consider the description of $Y$ in $\mathbb P(V_6\oplus L^\vee)\setminus \mathbb P(L^{\vee})$. Recall that $\wt Y$ is described on $\mathbb W:=\PP_{\PP(V_6)}(\Oc(-h)\oplus L^{\vee}\otimes\Oc)$ as the zero locus of a section of $ \pi^*\wt \Q^{\vee}(h+H)$. Also, recall the following short exact sequence on $\mathbb P(V_6)$, where the right map is given by coordinates on $\mathbb P(V_6)$.
\begin{equation}
    0\to \wt \Q^{\vee}\to V_6\otimes \O _{\mathbb P(V_6)}\to \O _{\mathbb P(V_6)}(h)\to 0 
\end{equation} 
It further induces the following exact sequence.
\begin{equation}
    0\to \pi^*\wt \Q^{\vee}(h+H)\to V_6\otimes \O _{\proj (V_6)}(2H-E)\to \O _{\mathbb P(V_6)}(2h+H)\to 0
\end{equation}
It follows that a section of $\pi^*\wt \Q^{\vee}(h+H)$ is a section of $V_6\otimes \O _{\proj(V_6)}(2H-E)$ satisfying the property that it is mapped to zero after multiplication with the vector of coordinates $\mathbb P(V_6)$. In particular,   $Y\setminus \mathbb P(L^{\vee})$ is defined in $\mathbb{P}^{\circ}:=\mathbb P(V_6\oplus L^\vee)\setminus \mathbb P(L^{\vee})$ by a system of quadrics $q_1\dots q_{6}$ vanishing on $\mathbb P(L^{\vee})$, such that 
\begin{equation}
    [q_1\dots q_{6}]* [x_1\dots x_6]^T=0
\end{equation}
where $x_1\dots x_6$ are the coordinates of $\mathbb P(V_6\oplus L^\vee)$ vanishing on $\PP(L^{\vee}) $ i.e. the ones which correspond to coordinates of $\mathbb P(V_6)$.
It follows from the Koszul sequence:
$$\wedge^2 V_6\otimes \O _{\mathbb{P}^{\circ}}(1)\to V_6\otimes \O _{\mathbb{P}^{\circ}}(2)\to \O _{\mathbb{P}^{\circ}}(3)\to 0 $$
that 
$[q_1\dots q_{6}]^T=M*[x_1\dots x_6]^T$ with $M$ a skew symmetric $6\times 6$ matrix of linear forms. Clearly, a general matrix $M$ can appear in this construction.

Now $Y$ in $\mathbb P(V_6\oplus L^{\vee})$ is the closure of 
$Y^{\circ}\subset \mathbb P^{\circ}$. For a description of the closure, observe that the vanishing of the quadrics at a point $p$ outside $\PP(L^\vee)$ implies that if we evaluate $M$ in $p$ it has a kernel and hence the Pfaffian of the matrix $M$ also vanishes at $p$. It follows that the Pfaffian of $M$ is also in the ideal of $Y$. Hence $Y$ is contained in a scheme defined by a Huneke-Ulrich ideal, which is known to be saturated, smooth and irreducible in general (see \cite[Section 3.1]{SSY}). It follows that the ideal of $Y$ is generated by the quadrics obtained from multiplication $M*[x_1\dots x_6]^T$ and the Pfaffian $\Pf(M)$. The equations of the exceptional divisor follow straight from the fact that $\beta$ is the blow-up of $\mathbb P(V_6\oplus L^\vee)$ in the locus $\{v(x)=0\}$ and $\tilde Y$ the proper transform of $Y$ via this blow-up. 
\end{proof}

\begin{proposition}
    $Y$ is an arithmetically Gorenstein threefold in $\P(V_6\oplus L^\vee)=\P^8$ of regularity 4.
\end{proposition}
\begin{proof}
\label[remark]{rem:universalHunekeUlrich}
One can implement the universal Huneke-Ulrich ideal in Macaulay2 by considering 
general matrices of linear forms $\til M$ and $\til v$ (of the shape specified in \Cref{prop:explicitDescriptionOfY}) in $\P^{20}$. Then one can verify that the minimal free resolution of the vanishing of this ideal, that is $V\equiv V(\til M.\til v, \Pf(\til M))\subset\P^{20}$, has the following form.
    \begin{multline}
        0 \leftarrow \mathcal{O}_V \leftarrow \mathcal{O} \leftarrow \mathcal{O}(-2)^{\oplus 6} \oplus \mathcal{O}(-3) \leftarrow \mathcal{O}(-3) \oplus \mathcal{O}(-4)^{\oplus 21}\leftarrow\\
        \leftarrow \mathcal{O}(-5)^{\oplus 21} \oplus \mathcal{O}(-6) \leftarrow \mathcal{O}(-6) \oplus \mathcal{O}(-7)^{\oplus 6} \leftarrow \mathcal{O}(-9) \leftarrow 0
    \end{multline}
Taking any general linear section of $V$ yields the same minimal resolution in the linear space containing the section. In particular, $Y$, due to its description in \Cref{prop:explicitDescriptionOfY}, has a minimal resolution of the same form, but now in a $\P^8$. So, due to the self-duality of this resolution, one concludes that $Y$ is arithmetically Gorenstein. Then, from the length of the resolution and the minimum twist appearing in the resolution one deduces that the regularity is 4.
\end{proof}

\begin{remark}\label{rem:23}
Using the above resolution, one can then verify the middle cohomology vanishings of $Y$ and compute the Hodge numbers and degree of $Y$. Moreover, using Macaulay2, one can also verify that the singular locus of $V$ has codimension 7 and thus deduce the smoothness and irreducibility of its general $\P^8$ section, $Y$. Then, from the fact that $Y$ is an arithmetically Gorenstein threefold having regularity 4, one concludes that $Y$ is Calabi-Yau. Such computations were performed in \cite[Section 3.1  and Example 4.1]{SSY}. Regardless, in what follows we will provide purely conceptual computations and proofs for the above stated invariants and properties of $Y$, independent of such computer algebra computations.
\end{remark}

\begin{proposition}\label[proposition]{prop:gorCY3} The following hold for the varieties $Y$ and $\wt Y$.
    \begin{itemize}
        \item $Y$ is a smooth Calabi--Yau threefold of degree $21$ in $\PP(V_6\oplus L^\vee)$. 
        \item $\wt Y$ is the blowup of $Y$ in a smooth elliptic curve $C = Y\cap \PP(L^\vee)$.
    \end{itemize}
\end{proposition}
\begin{proof}
    The degree of $Y$ in $\P(V_6\oplus L^\vee)$ is same as  $[\til Y].H^3$ computed in $\Bl_{\P(L^\vee)}\PP(V_6\oplus L^\vee)$, which in turn can be evaluated by Grothendieck relation for the relative hyperplane class of $\Bl_{\PP(L^\vee)}\PP(V_6\oplus L^\vee))$ (seen as a projective bundle).

    By the description of $\wt Y$ as the zero locus of a section, one finds that its normal bundle in the blown-up projective space is $\pi^*\wt\Qc^\vee(h+H)|_{\wt Y}$. Hence, by adjunction, the canonical class is $K_{\wt Y} = H - h = E_Y$ where $E_Y$ is the restriction of the exceptional divisor $E$ of $\beta$. Hence, $\beta|_{\wt Y}$ contracts an exceptional divisor $E$ to the intersection $C\coloneqq \PP(L^\vee)\cap Y$. Note that $Y\setminus C$ is smooth because $\wt Y\setminus E_Y$ is. And $C$ and $E_Y$ themselves are smooth by the proof of \Cref{prop:explicitDescriptionOfY}. So, by the Fujiki--Nakano contraction criterion \cite{fujiki-nakano-contraction-criterion}, to prove that $Y$ is smooth in a neighborhood of $C$, we just need to show that the contracted divisor $E_Y$ is a $\PP^1$ -bundle over $C$, and that the restriction of $\Oc(E_Y)$ to a fiber of $E_Y$ is $\Oc(-1)$. The latter claim follows by observing that such divisor is a zero locus of a general section of the restriction of $\wt \Qc^\vee(h+H)$ to $\PP(L^\vee)\times \PP(V_6)$. In particular, the pre-image of a point $[l]\in\PP(L^\vee)$ is cut by linear equations (a section of $\wt \Q^\vee(H)$ on $\PP(V_6)$). Fibers over $\PP(L^\vee)$, therefore, must be linear spaces. Since $E_Y$ is a smooth irreducible surface, it cannot contain any $\PP^2$. Hence, since the contraction is surjective onto $C$, it must be a scroll. This proves that $\beta_Y\colon E_Y\arw C$ is a $\PP^1$-fibration. Now, let us consider the restriction of $\Oc_{\wt Y}(E_Y)$ to a fiber of $\beta_Y$: since $\Oc_{\wt Y}(E_Y)\simeq \Oc_{\wt Y}(H-h)$, such a restriction is just $\Oc_{\PP^1}(-1)$. This shows that $Y$ is smooth.

    Finally, we known by \Cref{prop:explicitDescriptionOfY} that $C$ is a plane cubic hence an elliptic curve and by the canonical bundle formula for a blow-up,  that $Y$ has trivial canonical bundle. Moreover, the intermediate cohomologies of $\Oc_Y$ vanish, by the Hodge number computations in the proof of \Cref{prop:hodge_numbers}. Therefore, $Y$ is Calabi-Yau.

\end{proof}

In the light of above results, we can regard $Y$ as a general member of the family $\Y$ of arithmetically Gorenstein Calabi-Yau threefolds in $\P(V_6\oplus L^\vee)=\P^8$, that we define below. 

\begin{equation}\label{def:Y}
\Y := \left\{ V(M.v, \Pf(M)) \subset \P^8 \ \middle|\ 
\begin{aligned}
    & M \text{ is a general } 6 \times 6 \text{ skew-symmetric matrix of linear forms in }\P^8\\
    & v \text{ is a general } 6 \times 1 \text{ column of linear forms in }\P^8
\end{aligned}
\right\}
\end{equation}

\subsection{Hodge Equivalence, L-equivalence and Non-Birationality }

\subsubsection{Hodge Equivalence} In the following, we compute the Hodge numbers of $Y$ by first computing those of $\wt Y$. Comparing these numbers to those of $X$, we find: 
\begin{proposition}\label[proposition]{prop:hodge_numbers}
        $X$ and $Y$ have the same Hodge numbers.
    \end{proposition}
    \begin{proof}
        In the light of \Cref{prop:gorCY3} and \cite[Theorem 7.31]{voisin}, we just need to show that $\wt Y$ has the following Hodge numbers:
        \begin{equation*}
            \begin{array}{ccccccc}
                    &&& 1 &&& \\
                    && 0 && 0 && \\
                    & 0 && 2 && 0 & \\
                    1 && 53 && 53 && 1 \\
                    & 0 && 2 && 0 & \\
                    && 0 && 0 && \\
                    &&& 1 &&&
                \end{array}
        \end{equation*}
        The Hodge numbers of $\wt Y$ can be computed by its description as a zero locus of a general section of $\Ec \coloneqq \wt\Qc^\vee(H+h)$ in $\WW$, as follows. First, recall that to compute the Hodge numbers of a threefold one just needs to know the cohomology of its structure sheaf and its cotangent bundle. The latter can be resolved by the dual of the normal bundle sequence:
        \begin{equation*}
            0 \arw\Ec^\vee\big|_{\wt Y} \arw \Omega_\Pc^1\big|_{\wt Y} \arw \Omega_{\wt Y}^1 \arw 0.
        \end{equation*}
        Both terms of the resolution can be further resolved by vector bundles on $\Pc$ by means of the Koszul resolution of $\wt Y$ as a zero locus of a section of $\Ec$. Similarly, one can resolve $\Oc_{\wt Y}$. In this way, the calculation reduces to computing cohomology of vector bundles of the following shape:
        \begin{equation*}
            \pi^*\Sigma^\lambda\Qc(lh)\otimes\Oc(mH) \quad \text{for } l, m\in\ZZ, \quad \lambda \text{  partition}
        \end{equation*}
        where $\pi^*\Sigma^\lambda\Qc$ denotes the Schur power of $\Qc$ with respect to $\lambda$.\\
        \medskip
        The cohomology of such terms can be easily computed by pushing forward along $\pi$. In fact, one has:
        \begin{equation*}
            \begin{split}
                R\pi_*\bigl(\pi^*\Sigma^\lambda\Qc(lh)\otimes\Oc(mH)\bigr) & \simeq \Sigma^\lambda\Qc(lh)\otimes R\pi_*\Oc(mH) \\
                & \simeq \left\{
            \begin{array}{lc}
                \Sigma^\lambda\Qc(lh)\otimes\Sym^m(\Oc(h)\oplus L\otimes\Oc) & m\geq 0 \\
                0 & -3\leq m \leq -1 \\
                \Sigma^\lambda\Qc((l-1)h)\otimes\bigl(\Sym^{-4-m}(\Oc(h)\oplus L\otimes \Oc)\bigr)^\vee[-3] & m\leq -4
            \end{array}
            \right.
            \end{split}
        \end{equation*}
        The first two cases are straightforward. The case $m\leq -4$ can be deduced by the following computation:
        \begin{equation}
            \begin{split}
                R\pi_*\Oc(mH) & \simeq R\pi_*\Hc om\bigl(\Oc, \Oc(mH)\bigr) \\
                & \simeq R\pi_* \Hc om \bigl(\Oc((-4-m)H + h), \Oc(-4H+h)\bigr) \\
                & \simeq R\pi_*\Hc om\bigl(\Oc((-4-m)H+h), \pi^!\Oc\bigl)[-3] \\ 
                & \simeq \Hc om\bigl(R\pi_*\Oc((-4-m)H + h), \Oc\bigr)[-3] \\
                & \simeq \Oc(-h)\otimes\bigl(\Sym^{-4-m}(\Oc(h)\oplus L\otimes \Oc)\bigr)^\vee[-3],
            \end{split}
        \end{equation}
        where we used the Grothendieck duality and the fact that $\pi^!\Oc \simeq \omega_\Pc\otimes\pi^*\omega_{\PP(V_6)}^\vee[3]$.
        This allows to compute the cohomology of each bundle appearing in the resolutions by means of the Borel--Weil--Bott theorem. After a simple but tedious computation, one obtains the Hodge diamond above. The details of the computation can be found in Appendix \ref{app:hodgeCode}
    \end{proof}

The correspondence between $X$ and $Y$ can be made even more geometrically manifest by noting that there is a birational map between $\PP(V_6)$ and a codimension three linear section of $G(2, V_6)$, which is resolved on each side by blowing up $\wt Y$ and $X$, respectively. We elaborate on this construction below.

Recall the section $s$ that defines $X$ as its vanishing locus and the chain of isomorphisms (\ref{eq:chainOfIsos}). Denote by $Z$, the zero locus of the section $\sigma\in H^0(F(1,2,V_6),\Oc(h+\xi)\oplus L\otimes\Oc(\xi))$ such that $q_*\sigma = s$. Then, the restriction $q|_Z$ is a relative hyperplane section of a $\PP^1$-bundle. This can be seen by noting that $F(1, 2, V_6)$ is isomorphic to the projectivization of $\Uc(-\xi)$, and that intersecting $F(1, 2, V_6)$ with a section of $L\otimes\Oc(\xi)$ still yields a $\PP^1$-bundle over a codimension three section of $G(2, V_6)$. We choose this section to be the projection of $s$ to the $H^0(G(2,V_6),L\otimes\O(\xi))$ factor and denote the resulting linear section of $G(2, V_6)$ by $\mathbb{G}_L$. Therefore, $q|_Z$ is an isomorphism outside the pre-image of $X$, which is the projectivization of the normal bundle of $X$.

On the other hand, the contraction $F(1, 2, V_6)\arw \P(V_6)$ is a $\P^4$ -bundle, and the vector bundle $\Oc(h+\xi)\oplus L\otimes\Oc(\xi)$, on each fiber, restricts to four copies of $\Oc(1)$. This implies that the general fiber of $p:Z\arw \P(V_6)$ is a point. Higher-dimensional fibers occur precisely over the points of $\P(V_6)$ where the corresponding four linear sections are linearly dependent. This is the degeneracy locus $\wt Y$ of corank one: by the analysis above, no higher-corank loci appear, and therefore all fibers over $\wt Y$ are isomorphic to $\P^1$. Let us call $E_Y$ the divisor associated to such a $\P^1$-fibration over $\wt Y$, and introduce the notation $\bar p: E_Y\arw \wt Y$. For a fiber $\bar p^{-1}(x)$, by adjunction, one computes

\begin{equation}
    \omega_Z|_{\bar p^{-1}(x)}\simeq\Oc_{\bar p^{-1}(x)}(-5+4)=\Oc_{\bar p^{-1}(x)}(-1).
\end{equation}

Since $p$ is a birational morphism between smooth varieties with exceptional divisor $E$, by \cite[Section 2.3]{kollar_mori} we have
\begin{equation}
    \omega_Z=p^*\omega_{\PP(V_6)}(aE)
\end{equation}

for some positive integer $a$ (note that $Z$ has Picard rank two). Restricting to the fiber gives $a = 1$. The Fujiki--Nakano contraction criterion \cite{fujiki-nakano-contraction-criterion} therefore shows that $p$ is the blowup of $\P(V_6)$ with center $\wt Y$.

We summarize the above construction by the following diagram.

\begin{equation}\label{eq:blowup_diagram}
    \begin{tikzcd}[row sep = huge, column sep = huge]
        & E_X \ar[hook]{r}{\beta} \ar{dl}{\bar q} & Z \ar[hookleftarrow]{r}{\alpha} \ar{dl}{q} \ar[swap]{dr}{p} & E_Y \ar[swap]{dr}{\bar p} & \\
        X \ar[hook]{r} & \mathbb{G}_L & & \PP(V_6) \ar[dashed, <->]{ll} \ar[hookleftarrow]{r} & \wt Y
    \end{tikzcd}
\end{equation}

Here, by abuse of notation, we still use $p$ and $q$ to denote their restrictions to $Z$, and by $\bar p$ and $\bar q$ we denote their further restrictions to the exceptional divisors (which are $\PP^1$-bundles). This construction yields, in particular, the following isomorphism, that we will use going further.
\begin{proposition}
    There exists an isomorphism $\Bl_{X}\mathbb{G}_L\simeq \Bl_{\wt Y}\P(V_6)$
\end{proposition}

\subsubsection{L-Equivalence}
    In the following, we consider classes of varieties (denoted by [\_]) in the Grothendieck ring. We will denote by $\LL$ the class of the affine line, by $\mathbb{G}$ the Grassmannian $G(2,6)$. Let us begin by recalling some standard facts.
    \begin{lemma}\label{lem:class_projective_space_grothendieck_ring}
        We have:
        \begin{equation}
            [\PP^n] = 1+\LL+\LL^2+\cdots+\LL^n
        \end{equation}
    \end{lemma}
    \begin{lemma}\label{lem:class_blowup_grothendieck_ring}
        Consider a variety $Y$ and a subvariety $C\subset Y$ of codimension $c\geq 2$. Denote by $\Bl_C Y$, the blowup of $Y$ along $C$. We have:
        \begin{equation}
            [\Bl_C Y] = [Y] + [C](\LL+\cdots + \LL^{c-1}).
        \end{equation}
    \end{lemma}

    \begin{theorem}\label{thm:L_equivalence}
        The following relation holds in the Grothendieck ring of varieties. In particular, $X$ and $Y$ are $\mathbb{L}$-equivalent.
        \begin{equation*}
            ([X] - [Y])\LL^3 = 0
        \end{equation*}
    \end{theorem}
        \begin{proof}
            In particular, in light of Diagram \ref{eq:blowup_diagram}, by Lemma \ref{lem:class_blowup_grothendieck_ring} one has:
            \begin{equation}
                \begin{split}
                    [Z] & = [\mathbb{G}_L] + [X]\LL \\
                    & = [\PP^5] + [\wt Y]\LL.
                \end{split}
            \end{equation}
            However, since $\wt Y$ is also a blowup, we have $[\wt Y] = [Y] + [C]\LL$, and therefore:
            \begin{equation}\label{eq:main_relation_L_equivalence}
                0 = \bigl( [X]-[Y] \bigr)\LL + [\mathbb{G}_L] - [\PP^5] -[C]\LL^2.
            \end{equation}
            By \cite[Theorem 3.3]{laterveer} one has the following equality:
            \begin{equation}\label{eq:laterveers_equation}
                [\mathbb{G}_L]\LL^2 + [\PP^1][\mathbb{G}] = [C]\LL^4 + [\PP^2][\mathbb{G}\cap H]
            \end{equation}
            where $[\mathbb{G}\cap H]$ denotes the class of a hyperplane section in $G(2, 6)$.\\
            \\
            Here we use the fact that the two elliptic curves appearing in Equation \ref{eq:main_relation_L_equivalence} and Equation \ref{eq:laterveers_equation} are in fact the same, because they are defined by the same Pfaffian equations restricted to $\PP(L^\vee)$.
            Applying a result of \cite{martin_grothendieck_ring} we have
            \begin{equation}
                [\mathbb{G}] = [\PP^4](1+\LL^2+\LL^4)
            \end{equation}
            while \cite[Lemma 3.2]{laterveer} gives
            \begin{equation}
                [\mathbb{G}\cap H] = [\PP^3](1+\LL^2+\LL^4).
            \end{equation}
            plugging the last two relations into Equation \ref{eq:laterveers_equation} and rearranging the powers of $\LL$ by means of Lemma \ref{lem:class_projective_space_grothendieck_ring} yields:
            \begin{equation}
                [\mathbb{G}_L]\LL^2 = [C]\LL^4 + [\PP^5]\LL^2.
            \end{equation}
            Then, the claim follows by multiplying both sides of Equation \ref{eq:main_relation_L_equivalence} by $\LL^2$ and using the last equation to rewrite $[\mathbb{G}_L]\LL^2$.
        \end{proof}

\subsubsection{Non-Birationality}
The above stated geometric relationships are non--trivial, in the light of the fact that $X$ and $Y$ are non-birational. We prove this below.
\begin{proposition}
    $X$ and $Y$ are not birational.
\end{proposition}
\begin{proof}
    Recall that birational Calabi--Yau threefolds having Picard rank one are isomorphic. Now, if $X$ and $Y$ were isomorphic, then the ratio of the cube of any ample class on $X$ by the cube of any ample class on $Y$ must be a rational cube. This is clearly not the case, by the degree computations of $X$ and $Y$. Therefore, we conclude that $X$ and $Y$ cannot be birational.
\end{proof}

\subsection{Semi-Orthogonal Decompositions}
    One can use Orlov's blowup formula \cite{orlovblowup} with respect to the two blowup structures of Diagram \ref{eq:blowup_diagram}, obtaining the following semi-orthogonal decompositions:
    \begin{equation}\label{eq:short_sods}
        \begin{split}
            D(Z) & = \langle R\alpha_*\bar p^*D(\wt Y) \otimes\Oc(E_Y), Lp^*D(\PP(V_6^\vee)) \rangle \\
            & = \langle R\beta_*\bar q^*D(X) \otimes\Oc(E_X), Lq^*D(\mathbb{G}_L) \rangle
        \end{split}
    \end{equation}
    These decompositions can be further refined. First, recall the Beilinson decomposition \cite{beilinson}:
    \begin{equation}\label{eq:sod_P5}
        D(\PP(V_6)) = \langle \Oc, \Oc(h), \dots, \Oc(4h), \Oc(5h)\rangle.
    \end{equation}
    Moreover, a semi-orthogonal decomposition for $\mathbb{G}_L$ can be found by means of homological projective duality (see \cite[Corollary 10.9]{kuznetsov_grassmannians_of_lines} for details):
    \begin{equation}\label{eq:sod_G_L}
        D(\mathbb{G}_L) = \langle \Cc, \Oc, \Uc^\vee, \Oc(\xi), \Uc^\vee(\xi), \Oc(2\xi), \Uc^\vee(2\xi)\rangle.
    \end{equation}
    Here $\Cc$ denotes a subcategory which is equivalent to the derived category of an elliptic curve. By plugging \eqref{eq:sod_P5} and \eqref{eq:sod_G_L} into \eqref{eq:short_sods} we find two different decompositions of $D(Z)$, each consisting on six vector bundles, a Calabi--Yau type category of dimension three, and a Calabi--Yau type category of dimension one. The latter, on the $\PP(V_6)$-side, can be achieved by means of the structure of $\wt Y$ as a blowup of $Y$ in an elliptic curve. Hence, it is reasonable to expect that the subcategory $\langle R\beta_*\bar q^*D(X), \Cc \rangle$ can be proven to be equivalent to $D(\wt Y)$ via a sequence of mutations.

    This approach, in principle, could be generalized and refined in several directions. In fact, proving that $\Phi(\Cc)$ is semi-orthogonal to $D(Y)$ would produce a new proof of Theorem \ref{thm:thmB} which does not rely on matrix factorization categories. However, this proved to be technically involved.

    On the other hand, a similar pattern of mutations can be realized in greater generality, by letting the dimensions of $L$ and $V_6$ vary, thus describing the derived categories of certain sub-varieties of general Grassmannians of lines. This problem is related to the construction of the homological projective dual of a  categorical crepant resolution of singularities of a twelve-dimensional, singular subvariety of $\PP^{20}$, and it is the subject of ongoing work \cite{our_hpd_project}.
    
\section{The Mirror Moduli}\label{sec:mirrorEquivalence}

Here, we will use Batyrev's conjectural mirror construction, following \cite{batyrevGKZ} and \cite{batyrevToricDegen}, to understand and compare the mirror moduli of the Calabi-Yau threefold families $\X$ and $\Y$ constructed in \Cref{sec:geometry}. Given a Calabi-Yau variety $V$ that is a complete intersection of Cartier divisors in a smooth Fano variety $W$ and given that there exists a "small toric degeneration" \cite[Definition 3.1]{batyrevToricDegen} of $W$, Batyrev constructs a certain power series, which is expected to serve as the principal period of the complex structure moduli space of the mirror manifold. Now by the mirror map, this moduli space may be re-described as a certain space associated to $V$, called the \textit{stringy Kahler moduli space} (SKMS). Roughly speaking, this space is expected to be some "reasonable extension" of the complexified Kahler moduli space of $V$, by including certain physically motivated, but non-birational models into the moduli space. From this perspective, Batyrev's method can be interpreted as a conjectural way to understand  the quantum variation of Hodge structures on the SKMS of a Calabi-Yau variety (which, by definition, encode the Gromov-Witten invariants of the variety), independent of the specificity of the actual construction of its mirror. We re-write this conjecture below.

\begin{remark}
    In the above paragraph, we intend not to precisely define the SKMS, but rather to give a heuristic overview of it. As of now, there is no general intrinsic (i.e. mirror independent) definition of the SKMS, but only various situation-dependent proposals floating around. See for example, \cite{skmsHLeistnerSam}, \cite{skmsDonovanWemyss}.
\end{remark}
\begin{conjecture}[Batyrev \cite{batyrevToricDegen}]\label[conjecture]{conj:Batyrev}
Consider a Calabi-Yau variety $\til V$ that is a complete intersection of Cartier divisors $\{\til V_i\}_i$ in a smooth Fano variety $\til W$. Let 
\begin{itemize}
    \item $W$ be a small toric degeneration of $\til W$ with fan $\Sigma$.
    \item $V\subset W$ be the simultaneous degeneration of $\til V\subset \til W$
    \item $\{V_i \subset W\}_i$ be the simultaneous degenerations of $\{\til V_i\subset \til W\}_i$. Then, $V\subset W$ is the complete intersection of the Cartier divisors $\{V_i\}_i$ in $W$.
    \item $\{W_j\}_j$ be the torus-invariant Weil divisors corresponding to the rays of the fan describing $W$ as a toric variety.
\end{itemize} 

Denote by $L(\Sigma)$, the monoid in $\Z_{\geq0}^{|\Sigma(1)|}$ of non-negative integral relations among the primitive ray generators of $\Sigma$. 
And consider the following pairing between elements of $L(\Sigma)$ and torus invariant Weil divisors in $W$.
\begin{equation}
    \langle \ul l, \sum_j a_j W_j\rangle \equiv \sum_j l_j a_j
\end{equation}

Then, the power series $\Phi_0(\ul t)$ defined below is the principal period of the Hodge-theoretic B-side mirror moduli of $S$. We will call it the \textit{mirror principal period}.

\begin{equation}\label{eqn:mirrorPrincipalPeriod}
\Phi_0(\ul t):=\sum_{\ul l\in L(\Sigma)}\frac{\prod_i \langle \ul l, V_i\rangle !}{\prod_j \langle \ul l, W_j\rangle  !}\ul t^{\ul l}
\end{equation}
where $\ul t\in A(\Sigma)$ with $A(\Sigma)$ being a certain sub-variety of $\C^{|\Sigma(1)|}$, defined in \cite[Definition 4.1]{batyrevToricDegen}.
\end{conjecture}

\begin{remark}
    By the definition of small toric degenerations, it follows that $W$ must be a Gorenstein toric Fano variety with at worst terminal singularities. Hence, it must be singular in codimension at least three and moreover, must admit at worst compound $A_1$-singularities in codimension three. So, it follows that $V$ must have at worst nodal singularities, in which case it is called a \textit{Calabi-Yau conifold}.
\end{remark}

Let us focus on the case where $V$ is a threefold of Picard rank 1, which is our case of interest. This case is particularly simple, as here we expect the mirror principal period to be a power series in a single complex variable. To get such a description for $\Phi_0(\ul t)$, we first fix a minimal generating set $\{f_1,...,f_k\} \subset L(\Sigma)$, that allows us to express any $\ul l\in L(\Sigma)$ as a non-negative integral combination of vectors $f_i$ and hence any monomial $\ul t^{\ul l}$ as a product of monomials $\ul t^{f_i}$. Suppose it turns out that $\ul t^{f_i}|_{A(\Sigma)}=z^{a_i}$ for some function $z\in \O(A(\Sigma))$ and $a_i\in\Z_{>0}$. Then, we may express $\Phi_0(\ul t)$ as a power series $\Phi_0(z)$, in the single variable $z$ which we will view as living in $\C\P^1$.

Now consider the ring $\C[z,\Theta]$ of "log-differential operators" on the power series ring $\C[\![z]\!]$, where $\Theta:=z\frac{\partial}{\partial z}$. This ring is non-commutative as we have a non-trivial commutator $[\Theta, z]=z$. This commutation relation, allows one to write any $D\in \C[z,\Theta]$ in the form $D=z^m P_0(\Theta)+z^{m-1}P_1(\Theta)+...+P_m(\Theta)$, where $P_k$ are some polynomials. It is well known (see for example \cite{batyrevGKZ}) that if the power series $\Phi_0(z)$ is indeed the principal period for a pencil of threefolds having $h^{1,2}=1$, then it must uniquely determine (up to a $\C[z]$ polynomial factor) such a operator $D$ satisfying the following conditions.
\begin{itemize}

    \item $D$ has $\Theta$-degree 4. 
    \item $D$ has the unique indicial root 0. That is, $P_m(\Theta)$ is proportional to $\Theta^4$
    
    \item $D$ annihilates $\Phi_0(z)$. That is, $D\Phi_0(z)=0$. 
    
\end{itemize}
So, \Cref{conj:Batyrev} implies that if such an operator $D$ exists, then it must be the Picard-Fuchs operator that describes variation of middle Hodge structure on the mirror moduli. Whenever it exists, we will call it the \textit{mirror Picard-Fuchs operator}. In this point of view, the mirror moduli (or intrinsically speaking, the SKMS) of $V$ must look like $\C\P^1_z$ minus the singularities of $D$ (that is, points around which $D$ has non-trivial monodromy). Next, we aim to compute and compare the mirror Picard-Fuchs operators for the Calabi-Yau pair discussed in \Cref{sec:geometry}.

\subsection{Computations for $\X$} A general member $X$ of the family $\X$ (see \Cref{sec:X} or \cite[Table 1, no. 7]{ionue_ito_miura_CY3s}) may be described as a linear section of a certain Schubert variety $\Sigma$ in the Cayley plane $O\P^2$, by 9 general hyperplanes. Later, in \cite{miuraMirrorSymmetry}, Miura constructed a small toric degeneration of $\Sigma$. Then, following the procedure laid out in \Cref{conj:Batyrev}, he computed the mirror principal period and then the mirror Picard-Fuchs operator of $X$, which we re-write below\footnote{At the time of writing this paper, this operator is the one labeled as "AESZ no. 198" in the database \cite{database} of Calabi-Yau type operators}. 
\begin{equation}
    \begin{aligned}
        D_X &= 11^{2} \Theta^4+77 z\left(130\Theta^4+266\Theta^3+210\Theta^2+77\Theta+11\right)\\
        &-z^{2}\left(32126\Theta^4 +89990\Theta^3+103725\Theta^2+55253\Theta+11198\right)\\ &+z^{3}\left(28723\Theta^4+74184\Theta^3+63474\Theta^2+20625\Theta+1716\right)\\
        &-7z^{4}\left(1135\Theta^4+2336\Theta^3+1881\Theta^2+713\Theta+110\right)+7^{2}z^{5}\left((\Theta+1)^4\right)
    \end{aligned}
\end{equation}

\subsection{Computations for $\Y$}\label{sec:toricComputationsForY}
A general member $Y$ of the family $\Y$ (see \eqref{def:Y}) may be described as a linear section of the following six dimensional smooth variety (see \Cref{rem:23} for smoothness) $\til W$ in $\P^{11}$, by three general hyperplanes. 
\begin{equation}
    \til W:=V(\til M \cdot \til v, \Pf(\til M)) \subset \P^{11}
\end{equation}
where $\til M$ is a general $6 \times 6$ skew-symmetric matrix of linear forms and $\til v$ is a general $6 \times 1$ column of linear forms. Note that by Grothendieck-Lefchetz's theorem, $\til W$ must have Picard rank 1, as the same is true for $Y$. Now consider the following specializations of the general matrices $\til M$ and $\til v$.
\begin{equation}
 M := \begin{bmatrix}
0 & x_0 & 0 & 0 & 0 & -x_5 \\
-x_0 & 0 & x_1 & 0 & 0 & 0 \\
0 & -x_1 & 0 & x_2 & 0 & 0 \\
0 & 0 & -x_2 & 0 & x_3 & 0 \\
0 & 0 & 0 & -x_3 & 0 & x_4 \\
x_5 & 0 & 0 & 0 & -x_4 & 0
\end{bmatrix}, \quad
 v := \begin{bmatrix}x_6 \\x_7 \\x_8 \\x_9 \\x_{10} \\x_{11}
\end{bmatrix}
\end{equation}
Using them we may construct the following specialization of $\til W$, which turns out to be a toric variety as its homogenous ideal is prime and generated by binomials.
\begin{equation}
\begin{aligned}
     W &:= V(M \cdot v, \Pf(M)) \\
      &  \phantom{||}
          \begin{aligned}
              =V(& x_3 x_9 - x_4 x_{11}, \\
              & x_2 x_8 - x_3 x_{10}, \\
              & x_1 x_7 - x_2 x_9, \\
              & x_0 x_7 - x_5 x_{11}, \\
              & x_5 x_6 - x_4 x_{10}, \\
              & x_0 x_6 - x_1 x_8, \\
              & x_0 x_2 x_4 - x_1 x_3 x_5)
          \end{aligned}
\end{aligned}
\end{equation}  
In \Cref{app:toricCode}, aided by a Macaulay2 script, we perform several computations on the above toric degeneration to obtain the mirror Picard-Fuchs equation of $Y$. We will summarize our findings below. First, we verify that $ W$ is indeed a flat degeneration of $\til W$ by showing that its Hilbert polynomial matches with that of $\til W$. Then, we find the following collection of torus characters (of $(\C^*)^7$) that describe a monomial parametrization of $ W$.
\begin{equation}
\begin{array}{cc}
     \begin{aligned}
        v_0 &= (1,0,-1,0,0,0,1)\\
        v_1 &= (1,1,-1,-1,0,0,1)\\
        v_2 &= (1,1,0,-1,-1,0,1)\\
        v_3 &= (1,1,0,0,-1,-1,1)\\
        v_4 &= (1,1,0,0,0,-1,0)\\
        v_5 &= (1,0,0,0,0,0,0)
    \end{aligned} &
    \begin{aligned}
        v_6 &= (0,1,0,0,0,0,0)\\
        v_7 &= (0,0,1,0,0,0,0)\\
        v_8 &= (0,0,0,1,0,0,0)\\
        v_9 &= (0,0,0,0,1,0,0)\\
        v_{10} &= (0,0,0,0,0,1,0)\\
        v_{11} &= (0,0,0,0,0,0,1)
    \end{aligned}
\end{array}
\end{equation}

That is, $ W = $ closure of image of the following map
\begin{equation}\label{eqn:parametrization}
\begin{aligned}
    (\C^*)^7 &\longrightarrow \P^{11}\\
    \ul t &\longmapsto  [\ul t^{v_0}:\ul t^{v_1}:..:\ul t^{v_{11}}]
\end{aligned}
\end{equation}
Let $\Sigma$ be the normal fan of the polytope $P$ obtained by taking the convex hull of the above set of characters. Then, \eqref{eqn:parametrization} is equivalent to saying $ W \simeq X_\Sigma\equiv$ the normal toric variety associated to the fan $\Sigma$. Then, we inspect the properties of $\Sigma$ and conclude that $ W$ is a terminal Gorenstein toric Fano variety having Picard rank 1 and hence is a small toric degeneration of the smooth Fano variety $W$.

Our fan $\Sigma$ has 18 rays which correspond to 18 torus-invariant Weil divisors $ W_0,.., W_{17}\subset  W$. We know $-K_{ W} =  W_0+...+ W_{17}$. Let us partition this summation as follows
\begin{equation}
\begin{aligned}
    -K_{ W} &= H_1+H_2+H_3,\\
    \text{where } H_1 &\equiv  W_5+ W_6+ W_7+ W_8+ W_9,\\
    H_2 &\equiv  W_0+ W_3+ W_{11}+ W_{13}+ W_{16},\\
    H_3 &\equiv  W_1+ W_2+ W_4+ W_{10}+ W_{12}+ W_{14}+ W_{15}+ W_{17}
\end{aligned}
\end{equation}
Turns out, $H_1,H_2,H_3$ are divisors in the hyperplane class of $ W\subset \P^{11}$. So, seeing them as degenerations of three hyperplane sections of $W$, that cut-out a Calabi-Yau $Y\subset W$ from the family $\Y$, we can realize $ Y \equiv  W\cap H_1\cap H_2\cap H_3$ as a flat degeneration of that $Y$. 

 Also, $|\Sigma(1)|=18$ means $L(\Sigma)$ and $A(\Sigma)$ live in $\Z_{\geq0}^{18}$ and $\C^{18}$, respectively. We find that $L(\Sigma)$ has 50 minimal generators $f_1,...,f_{50}$ whose corresponding Laurent monomials, restricted to $A(\Sigma)$, look like $\ul t^{f_i}|_{A(\Sigma)}=z^{\langle f_i,H\rangle}$, where $z$ is a certain function on $A(\Sigma)$ and $H$ is any divisor in the hyperplane class. Therefore, the data $(\Sigma, \{H_1,H_2,H_3\})$ describe the following power series

\begin{equation}
\Phi_0(\ul t)=\sum_{\ul l\in L(\Sigma)}\frac{\langle \ul l, H_1\rangle ! \langle \ul l, H_2\rangle ! \langle \ul l, H_3\rangle !}{\prod_j \langle \ul l,  W_j\rangle  !}\ul t^{\ul l} = \sum_{n=0}^\infty\sum_{\substack{\ul l\in L(\Sigma)\\ \langle \ul l, H\rangle=n}}\frac{(n!)^3}{\prod_j l_j  !} z^n = \Phi_0(z)
\end{equation}

We find that the lattice spanned by $L(\Sigma)$ has rank 12, which allows us to re-write above summation in 12 variables only, which we call $\ul m\equiv(m_0,...,m_{11})$. In these variables, $L(\Sigma)$ is given by inequalities $m_0,...,m_{11},f_1(\ul m),...,f_6(\ul m)\geq 0$, where $f_i(\ul m)$ are certain integral linear forms in $\ul m$, defined below. And the condition $\langle \ul l, H\rangle=n$ becomes $m_4+...+m_{11}=n$. We finally get
\begin{equation}\label{eqn:mirrorPrincipalPeriodForY}
\begin{aligned}
\Phi_0(z) &=\sum_{n=0}^\infty\sum_{\substack{\ul m \in \Z^{12}_{\geq0}\\f_i(\ul m)\geq 0 \ \forall i\\ m_4+...+m_{11} =n}}\frac{(n!)^{3}}{m_0!...m_{11}!f_1(\ul m)!...f_6(\ul m)!} z^n \\
\text{where }
f_1(\ul m) &\equiv -m_1-m_2+m_4+m_5+m_6,\\
f_2(\ul m) &\equiv m_0-m_1-m_2-m_3+m_4+m_5+2m_6+m_8+m_9+m_{10}+2m_{11},\\
f_3(\ul m) &\equiv -m_1-m_3+m_5+m_6+m_7+m_8+m_{10}+m_{11},\\
f_4(\ul m) &\equiv -m_3+m_7+m_8+m_9+m_{10}+m_{11},\\
f_5(\ul m) &\equiv -m_0+m_1+m_3-m_6-m_8-m_{11},\\
f_6(\ul m) &\equiv -m_0+m_1+m_2+m_3-m_5-m_6-m_{10}-m_{11}
\end{aligned}
\end{equation}

\begin{proposition}\label[proposition]{prop:PicardFuchs}
    The following differential operator\footnote{At the time of writing this paper, this operator is the one labeled as "AESZ no. 193" in the database \cite{database} of Calabi-Yau type operators} annihilates $\Phi_0(z)$ and hence is the mirror Picard-Fuchs operator for $\Y$. 
\begin{equation}
\begin{aligned}
D_Y &= 7^{2} \Theta^4-7 z\left(1135\Theta^4+2204\Theta^3+1683\Theta^2+581\Theta+77\right)\\
&+z^{2}\left(28723\Theta^4+40708\Theta^3+13260\Theta^2-1337\Theta-896\right)\\
&-z^{3}\left(32126\Theta^4+38514\Theta^3+26511\Theta^2+10731\Theta+1806\right)\\
&+77z^{4}\left(130\Theta^4+254\Theta^3+192\Theta^2+65\Theta+8\right)+11^{2}z^{5}\left((\Theta+1)^4\right)
\end{aligned}
\end{equation}
\end{proposition}
\begin{proof}
The proof is somewhat technical and requires the aid of Macaulay2 code. So, will simply summarize it here and leave the details to \Cref{app:proofOfPicardFuchs}. We find that $\Phi_0(z)$ can be written as the following
\begin{equation}
    \Phi_0(z)=\sum_{n=0}^\infty \text{CT}((G_1G_2G_3)^n)z^n
\end{equation}
where CT stands for "constant term" and $G_i$ are the following Laurent polynomials in $\C[u_1^{\pm1},...,u_4^{\pm1}]$.
\begin{equation}
\begin{aligned}
G_1 &= \frac{u_1 u_2 u_3}{u_4} + u_1 u_2 + \frac{u_2 u_3 }{u_4 } + u_1 + 1, \\
G_2 &= \frac{u_4 }{u_2} + u_2 + u_3 + u_4 + 1, \\
G_3 &= \frac{1}{u_1 u_2} + \frac{1}{u_1 u_2 u_3} + \frac{u_4 }{u_1 u_2^2 u_3} + \frac{1}{u_3} + \frac{u_4}{u_2 u_3} + \frac{1}{u_2} + \frac{1}{u_2 u_3} + \frac{u_4 }{u_2^2 u_3}.
\end{aligned}
\end{equation}
So, by \cite[Proposition 7]{LGFano}, we have that $\Phi_0(z)$ is the principal period of the pencil $\{V(G_1G_2G_3-t)\subset(\C^*)^4)\}_{t\in \C}$ of threefolds. In other words, the Picard-Fuchs operator $D\in\C[z,\Theta]$ of this pencil annihilates $\Phi_0(z)$. Then, by an explicit computation, we verify that $D$ is in fact equal to the given operator $D_Y$. 
\end{proof}
\begin{remark}
    By the above proof, it is reasonable to expect that the mirror family to $\Y$ is (fiber-wise) birational to the pencil $\{V(G_1G_2G_3-t)\subset(\C^*)^4)\}_{t\in \C}$.
\end{remark}

\subsection{Comparing the computations} 

First we note that, $D_X$ and $D_Y$ have the same discriminant locus $Z$, seen up to the rational coordinate change $z\mapsto 1/z$. This can be verified by direct computation or by referring to the discriminant locus computations on the entries of $D_X$ and $D_Y$ in the database \cite{database}. Explicitly, it is given as follows.
\begin{equation}
Z=V((z^3-159z^2+84z+1)(-11+7z)^2)
\end{equation}

Also, under this coordinate change, we have $\Theta\mapsto-\Theta$ and hence $D_Y\mapsto D_Y(1/z,-\Theta)$. To this transformed operator, we left compose $z^4$ and right compose $z$ to get a polynomial operator having indicial root zero. Finally, we arrive at $z^4\circ D_Y(1/z,-\Theta)\circ z$, which is equal to $z^5D_Y(1/z,-\Theta-1)\in\C[z,\Theta]$, due to the commutation relation $[\Theta,z]=z$. Turns out, this series of manipulations to $D_Y$ exactly retrieves the operator $D_X$. That is, the following identity holds (it is straightforward to verify by direct computation).
\begin{equation}
    D_X(z,\Theta)=z^5D_Y(1/z,-\Theta-1)
\end{equation}
Above computations indicate that $D_X$ and $D_Y$ induce isomorphic local systems on $\C\P^1\setminus Z$. In other words, they can be regarded as Picard-Fuchs operators for the same one parameter moduli space, considered locally around $0$ and $\infty$ respectively. In the light of this realization and the monodromy computations for either of these operators (see \cite[Section 5.2.2]{miuraMirrorSymmetry} or their entries in database \cite{database}), it is reasonable to make the following conjecture. Here, our definition of a Hodge-theoretic mirror is as per \cite[Conjecture 17.1]{calabiYauBook}.
\begin{conjecture}[Double Mirror]\label[conjecture]{conj:doubleMirror}
    There exists a family $\mathcal{Z}$ of Calabi-Yau threefolds having $h^{1,1}=52$ and $h^{1,2}=1$ such that 
    \begin{itemize}

        \item $\mathcal{Z}$ is fiber-wise birational to the pencil $\{V(G_1G_2G_3-t)\subset(\C^*)^4)\}_{t\in \C}$ of threefolds, where $G_i$ are the Laurent polynomials appearing in the proof of \Cref{prop:PicardFuchs}.
        \item The variation of middle Hodge structure on $\mathcal{Z}$ has two maximal unipotent monodromy (MUM) points and three conifold points.
        
        \item $\X$ and $\Y$ are the Hodge-theoretic A-side mirrors to $\mathcal{Z}$ at those two MUM points, respectively.
    \end{itemize}
    As a consequence, general members $X\in \X$ and $Y\in \Y$ must have the same quantum cohomology and the same SKMS (for any "reasonable" definition of this space). Furthermore, the latter must look like \Cref{fig:SKMS}\footnote{This figure is a slightly modified version of \cite[Figure 1]{KapRam2025}. The topology of the SKMS for the case considered there is identical to that of ours.}.
\end{conjecture}

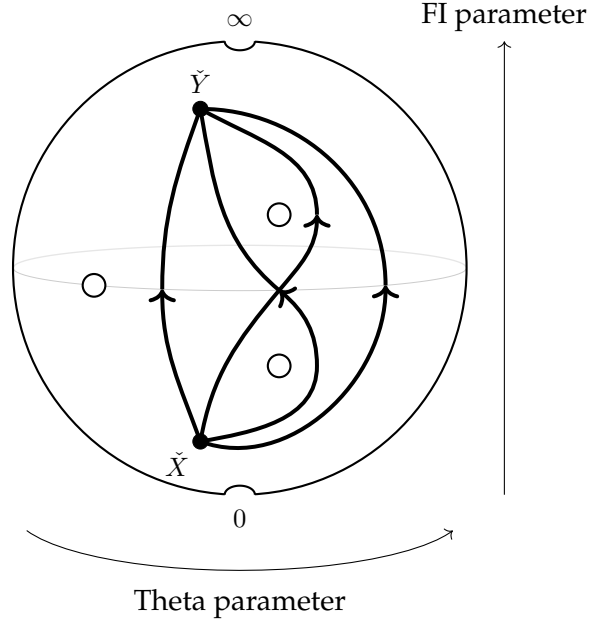
\begin{figure} 
\begin{tikzpicture}
  \def\c{3}
  \def\r{0.2}
  \def\p{0.1}
  \def\m{6}

  \draw[black!10, line width=0.5pt]
    plot[domain=0:180, smooth, variable=\t]
      ({\c*cos(\t)}, {\c*\p*sin(\t)});
  \draw[black!20, line width=0.3pt]
    plot[domain=180:360, smooth, variable=\t]
      ({\c*cos(\t)}, {\c*\p*sin(\t)});

  \coordinate (P) at ({3*cos(230)}, {0.3*sin(230)});
  \draw[fill=white, thick] (P) circle [radius=0.15];
  
  \draw[thick] (0,0) circle (\c);
  \fill[white] (0,\c) circle (\r); 
  \fill[white] (0,-\c) circle (\r);
  \draw[thick]
    plot[domain=180:360, smooth, variable=\t] ({\r*cos(\t)}, {\c+\m*\r*\p*sin(\t)});
  \draw[thick]
    plot[domain=180:360, smooth, variable=\t]
      ({\r*cos(\t)}, {-\c-\m*\r*\p*sin(\t)});  

  \coordinate (Q) at ({3*cos(280)}, {0.3*sin(280)});
  \coordinate (Qup) at ($(Q)+(0,1)$);
  \coordinate (Qdown) at ($(Q)+(0,-1)$);
  \draw[fill=white, thick] (Qup) circle [radius=0.15];
  \draw[fill=white, thick] (Qdown) circle [radius=0.15];
  
  \coordinate (M) at ({3*cos(260)}, {0.3*sin(260)});
  \coordinate (M+) at ($(M)+(0,2.4)$);
  \coordinate (M-) at ($(M)+(0,-2)$);
  \draw[fill=black] (M+) circle [radius=0.1] node[above=2pt] {\small $\chk Y$};
  \draw[fill=black] (M-) circle [radius=0.1] node[below left=0.5pt] {\small $\chk X$};
    
  \coordinate (b) at ({3*cos(250)}, {0.3*sin(250)});
  \coordinate (a) at ($(Qdown)+(0.5,0)$);
  \coordinate (c) at ($(Qup)+(0.5,0)$);
  \coordinate (d) at ({3*cos(310)}, {0.3*sin(310)});
  
  \draw[line width=1.5pt] (M+) to[out=250,in=90,relative=false] node[auto=right, pos=0.6]{} (b);
  \draw[line width=1.5pt, <-]     (b) to[out=270,in=110,relative=false] (M-);

  \draw[line width=1.5pt] (M+) to[out=280,in=140,relative=false] node[auto=right, pos=0.5]{} (Q);
  \draw[line width=1.5pt, <-]     (Q) to[out=320,in=90,relative=false] (a) to[out=270,in=10,relative=false] (M-);

  \draw[line width=1.5pt] (M+) to[out=330,in=90,relative=false] (c);
  \draw[line width=1.5pt, <-]     (c) to[out=270,in=50,relative=false] node[auto=left, pos=0.15]{} (Q) to[out=230,in=80,relative=false] (M-);

  \draw[line width=1.5pt] (M+) to[out=0,in=90,relative=false] node[auto=left, pos=0.7]{} (d);
  \draw[line width=1.5pt, <-]     (d) to[out=270,in=340,relative=false] (M-);
  
  \draw[->] (3.5,-3) to (3.5,3) node[above]{FI parameter};
  \draw[->] ({3*cos(200)}, {-3.2 + 0.8*sin(200)}) 
    arc[start angle=200, end angle=340, x radius=3, y radius=0.8] 
    node[midway, below=3pt] {Theta parameter};

\node[above=2pt] at (0,\c) {\small $\infty$};
\node[below=2pt] at (0,-\c) {\small $0$};

\end{tikzpicture}
\caption{SKMS with representatives of four homotopy classes of paths from $\chk X$ to $\chk Y$ (disregarding loops around the singularities, these are all classes of such paths).}
\label{fig:SKMS}
\end{figure}

Let us look at the homological implications of this conjecture. Keep in mind that these implications are themselves conjectural and should be treated heuristically. Denote by $\chk X$ and $\chk Y$, the homological A-side mirrors of general members $X\in \X$ and $Y\in \Y$, respectively. That is, $\chk X$ and $\chk Y$ are Calabi-Yau threefolds such that $D(X)\simeq Fuk(\chk X)$ and $D(Y)\simeq Fuk(\chk Y)$. Now, if we believe that $\chk X$ and $\chk Y$ are also Hodge theoretic mirrors, then above conjecture suggests that they must be deformation equivalent and in particular, live near the MUM points $0$ and $\infty$ of the SKMS (figure \ref{fig:SKMS}), respectively. So, their Fukaya categories must be equivalent by symplectic parallel transport along a path connecting them in the SKMS. Therefore, we expect there to exist a derived equivalence between $X$ and $Y$ via a composition of the form 
\begin{equation}
    D(X)\simeq Fuk(\chk X)\simeq Fuk(\chk Y)\simeq D(Y)
\end{equation}
The rest of this paper is devoted to the proof of existence of such a derived equivalence (without commenting on the existence of the homological mirrors $\chk X$ and $\chk Y$). Following that, in \Cref{rem:windowShiftAutoEquivalences} and \Cref{rem:4Windows}, we discuss the implications of having multiple homotopy classes of paths between $\chk X$ and $\chk Y$. A posteriori, this result may also be treated as another piece of evidence for \Cref{conj:doubleMirror}. 

\begin{remark}
Physically, the SKMS may be interpreted as a FI-theta parameter space (FI stands for "Fayet–Iliopoulos"),  parameterizing a family of a certain class of two dimensional (2,2) superconformal field theories, called \textit{gauged Landau-Ginzburg models}. There is a general conjecture in this setting which states that the "categories of B-branes", must be equivalent for the "limiting phases" of such a moduli, which in our case are $X$ and $Y$. Mathematically, this has been proven to be true in a quite a few cases (some non-trivial examples include \cite{Seg2014}, \cite{rennemo_hpd_sym2V}, \cite{KapRam2025}). Based on this physical motivation and prior mathematical successes of this approach, we will adopt the same for the purpose of our proof.
\end{remark}

\section{Landau-Ginzburg Models and B-Branes}\label{sec:LG}
We will give a brief overview of Landau-Ginzburg models and categories of B-branes over them, focusing only on notions that we will use for the proof. For a more rigorous and general treatment of these topics we refer to \cite{Ship2010} and \cite{Seg2011}.

\subsection{Landau-Ginzburg Models}
\begin{definition}
    Consider a smooth algebraic stack $\mathfrak{X}$ with a $\C^*$-action and a regular function $\omega$ on it. We say that the pair $(\mathfrak{X}, \omega)$ is a Landau-Ginzburg (LG) model if:
    \begin{itemize}
        \item $-1\in \C^*$ acts trivially. We will denote this $\C^*$ by $\C^*_R$ and refer to degree of various objects under this action by ``R-charge".
        \item $\omega$ has R-charge 2. i.e. $\omega(\lambda.x)=\lambda^2\omega(x)$
    \end{itemize}
\end{definition}
\begin{definition}
A morphism $(\mathfrak{X}_1, \omega_1)\rightarrow(\mathfrak{X}_2,\omega_2)$ between LG-models is a $\C^*_R$-equivariant morphism $\mathfrak{X}_1\rightarrow\mathfrak{X}_2$ under which $\omega_2$ pulls back to $\omega_1$.
\end{definition}
We will mainly be concerned with the case where $\mathfrak{X}$ is the quotient of a smooth quasi-affine scheme $S$ by an algebraic group $G$. In that case, the R-charge action is a suitable action by $\rchar$ on $S$ that commutes with the $G$-action and the function $\omega$ is a suitable $G$-invariant function on $S$. In this kind of model, $G$ is referred to as the \textit{gauge group}\footnote{While common in mathematics, this notation does not exactly coincide with the physical notion of a gauge group, which is usually taken to be a \emph{compact} subgroup of $G$.} and the $G$-invariance or equivariance property of associated objects is termed as \textit{gauge symmetry}. We will assume that our LG-models are of this kind from now on. If moreover $S$ is an affine space, we say that our model is a \textit{gauged linear sigma model (GLSM)}.

Following is an important object associated to an LG model.
\begin{definition}
    Given a LG model $(\mathfrak{X}, \omega)$, we define its critical locus as 
    \begin{equation}
        \Crit(\mathfrak{X}, \omega):=V(\omega,d\omega)\subset \mathfrak{X}
    \end{equation}
\end{definition}
Physically, one says that the critical locus of a LG model is the ``vacuum manifold" associated to that model, and is interpreted as the parameterizing space of all ground states in the theory.

\subsection{Category of B-branes}
\begin{definition}
    A graded B-brane on a LG model $\LG$ is a pair $(\E, d_\E)$ where, 
    \begin{itemize}
        \item $\E$ is a $\rchar$-equivariant vector bundle on $\mathfrak{X}$
        \item $d_\E:\E\rightarrow \E$ is a morphism of R-charge 1 such that ${d_\E}^2=\omega$.
    \end{itemize}
\end{definition}
We will usually refer to such objects simply as ``branes", in the rest of the paper. Given two branes $(\E, d_\E)$ and $(\F, d_\F)$ on $\LG$, we can construct the $\rchar$-equivariant bundle $\H om (\E,\F)=\E^\vee\otimes\F$ and an endomorphism on it:
\begin{equation}
    d_{\E,\F}(\varphi):= d_\F \circ \varphi - (-1)^{deg(\varphi)}\varphi\circ d_\E
\end{equation}
$d_{\E,\F}$ has $R$-charge 1 and satisfies ${d_{\E,\F}}^2=0$. So, $(\H om (\E,\F), d_{\E,\F})$ is a brane on $(\mathfrak{X}, 0)$. Moreover, note that taking its sheaf cohomology (of any degree $i$), we get a dg-vector space $(H^{i}(\mathfrak{X},\H om(\E,\F)),d_{\E,\F})$, graded by R-charge. The grading can be explicitly seen by working with the \v{C}ech model of the cohomology space with respect to a $\rchar$-equivariant open affine cover.

\begin{definition}
    The dg-category of B-branes on $\LG$ is denoted by $Br\LG$ and is defined as the category whose objects are graded $B$-branes on $\LG$ and morphisms are given as:
    \begin{equation}
        \Hom_{Br\LG}((\E,d_\E), (\F,d_\F)):= R\Gam(\H om (\E,\F), d_{\E,\F})
    \end{equation}
    where, $R\Gam$ is a ``suitable" monoidal functor $D(\mathfrak{X},0)\rightarrow D(Vect_\C)$ that computes derived global sections of branes.
\end{definition}
The functor $R\Gam$ has been defined in \cite{Seg2011} using Dolbeault cohomology and in \cite{Ship2010}  using Čech cohomology. The choice of its definition doesn't matter as finally we are interested in the homotopy category of $Br\LG$ which turns out to be the same for either choice.

\begin{definition}
    The category of B-branes on $\LG$ is denoted by $D\LG$ and is defined as the homotopy category of $Br\LG$, endowed with a natural triangulated structure defined as follows:
\begin{itemize}
    \item \textit{Shift functor}: Given any brane $(\E,d_\E)$, define $(\E,d_\E)[1]:=(\E(1),-d_\E)$, where $\E(1)$ is just $\E$ with its $\rchar$-equivariance structure shifted by 1 (which amounts to shifting its local $\Z$-grading by 1).
    \item \textit{Exact triangles}: Given any morphism of branes $\varphi:(\E,d_\E)\xrightarrow{}(\F,d_\F)$, define 
    \begin{equation}
        Cone(\varphi):=\left( \E(1)\oplus\F, \matTwo{-d_\E}{0}{\varphi}{d_\F}\right)
    \end{equation}
    Then, define exact triangles to be sequences of morphisms of branes that are isomorphic to the following natural sequence
    \begin{equation}
        (\E,d_\E)\xrightarrow{\varphi}(\F,d_\F)\xrightarrow{}Cone(\varphi)\xrightarrow{} (\E(1),-d_\E)
    \end{equation}

\end{itemize}
\end{definition}
Just like with derived categories, we will usually be interested in computing $RHom$ between objects in $D\LG$, which means considering all homologies of the morphism spaces of $Br\LG$ (instead of just the zeroth homology, as in the above definition). They may be computed using the following spectral sequence (see \cite[Remark 2.14]{Seg2011} or \cite[Section 2.1]{Ship2010}):
\begin{equation}\label{eqn:morsOfBranes}
    E_1^{p,q}\equiv (H^{p}(\mathfrak{X},\H om(\E,\F))_q,d_{\E,\F})\Longrightarrow H_{p+q} (R\Gam(\H om (\E,\F), d_{\E,\F}))
\end{equation}
In particular, we see that morphisms in $D\LG$ are independent of the precise choice of the functor $R\Gam$.

\begin{example}\label[example]{ex:LG}
    Consider $\mathfrak{X}$ with trivial $\C^*_R$-action and $\omega=0$. Then, any brane on $\mathfrak{X}$ is a graded vector bundle equipped with a degree 1 differential that squares to zero. In other words, it is simply a complex of vector bundles. Then looking at the morphisms between branes, it is easy to deduce that $D(\mathfrak{X},0)=\operatorname{Perf}(\mathfrak{X})$ (see \cite[Remark 2.8]{Seg2014}). But as $\mathfrak{X}$ is smooth, this means $D(\mathfrak{X},0)=D(\mathfrak{X})$.
\end{example}

\subsection{Kn\"orrer Periodicity}\label{sec:knorrer}
This is the main tool that we will use to relate the derived categories of our interest to certain B-brane categories. We will review the statement and the construction of the LG model involved, mainly following the exposition in \cite[Section 2]{Ship2010}. Consider, throughout this section, a smooth quasi-projective scheme $\mathfrak{X}$ and a globally generated vector bundle $\V$ over it. 

\begin{lemma}\label[lemma]{lem:section}
    There exists a natural identification 
    \begin{align*}
    \Gam(\mathfrak{X}, \V) &\xrightarrow{\sim} \{f \in \O(\Tot(\V^\vee))\ |\ f \text{ is linear on the fibers of the projection } \Tot(\V^\vee)\xrightarrow{p}\mathfrak{X}\}\\
    \psi &\mapsto \chk\psi
    \end{align*}
\end{lemma}
\begin{proof}
    Any $\psi\in \Gam(\mathfrak{X}, \V)$ maybe seen as a morphism $\psi:\mathfrak{X}\rightarrow \Tot(\V)$ that serves as a section for the projection $\Tot(\V)\rightarrow \mathfrak{X}$. Then locally, over a trivializing neighborhood $U\subset\mathfrak{X}$, it looks like $\psi = \sum_i\psi^ie_i$, where $\psi^i$ are some locally defined regular functions on $\mathfrak{X}$ and $(e_i)_i$ is a local frame of $\Tot(\V)$. This description induces following regular function on $\Tot(\V^\vee|_U)$
\begin{equation}\label{eqn:localViewOfFunc}
    \chk\psi := \sum_i\psi^iy_i
\end{equation}
  where $y_i$ are the coordinates on the fibers of $\Tot(\V^\vee|_U)$ corresponding to the chosen local frame for $\Tot(\V)$. These local functions patch up to give a global regular function $\chk\psi\in \O(\Tot(\V^\vee))$ that, by construction, is linear on the fibers. 
  
  Conversely, if we begin with such a function, then locally it must look like \eqref{eqn:localViewOfFunc}. Then one can run above construction in reverse to arrive at the unique section of $\V$ that induces this function.
\end{proof}
Define $\C^*_R$-action on $\Tot(\V^\vee)$ by fiber-wise dilation as $\lam.(x,\ul y):=(x, \lam^2\ul y)$. Then, given any section $\psi\in \Gam(\mathfrak{X}, \V) $, the corresponding function $\chk\psi \in \O(\Tot(\V^\vee))$ has R-charge 2 and hence induces a LG-model $(\Tot(\V^\vee),\chk\psi)$. 
\begin{definition}\label[definition]{def:knorrerModel}
    We will call an LG-model of the form $(\Tot(\V^\vee),\chk\psi)$ to be a \textit{Kn\"orrer model}.
\end{definition}

Now consider the pullback diagram
\begin{equation}
\begin{tikzcd}
    \Tot(\V^\vee|_{Z(\psi)})  \arrow[r, "i", hook]\arrow[d, "\bar p"] & \Tot(\V^\vee) \arrow[d, "p"] \\
    Z(\psi) \arrow[r, "\bar i", hook] &\mathfrak{X}
\end{tikzcd}
\end{equation}
By definition of $\chk \psi$, it vanishes on $\Tot(\V^\vee|_{Z(\psi)})\subset \Tot(\V^\vee)$. So, the maps $i$ and $\bar p$  lift to maps between LG models
\begin{equation}
    (Z(\psi),0)\xleftarrow{\bar p}(\Tot(\V^\vee|_{Z(\psi)}),0)\xrightarrow{i}(\Tot(\V^\vee),\chk\psi)
\end{equation}
where $Z(\psi)$ and $\Tot(\V^\vee|_{Z(\psi)})$ are seen with $\C^*_R$-actions induced by their embeddings in $\Tot(\V^\vee)$. Note that $Z(\psi)$ embeds in $\Tot(\V^\vee)$ via its embedding in $\mathfrak{X}$ followed by the zero section embedding of $\mathfrak{X}$ in $\Tot(\V^\vee)$. So in particular, the $\C^*_R$-action on $Z(\psi)$ is trivial and hence (by \Cref{ex:LG}) we have $D(Z(\psi),0)=D(Z(\psi))$. Therefore, we have the following induced sequence of functors
\begin{equation}
    D(Z(\psi))\xrightarrow{\bar p^*}D(\Tot(\V^\vee|_{Z(\psi)}),0)\xrightarrow{Ri_*}D(\Tot(\V^\vee),\chk\psi)
\end{equation}
\begin{theorem}[Kn\"orrer Periodicity \cite{Ship2010}]\label[theorem]{thm:knorrer}
    If $\psi$ is a regular section (i.e. $Z(\psi)\subset \mathfrak{X}$ has codimension equal to rank of $\V$), then the following functor is an equivalence of triangulated categories
    \begin{equation}
        Ri_*\circ \bar p^*:D(Z(\psi))\rightarrow D(\Tot(\V^\vee),\chk\psi)
    \end{equation}
\end{theorem}
\begin{remark}\label[remark]{rem:vansihingLocusVsCritLocus}
Using the local description of $\chk\psi$ \eqref{eqn:localViewOfFunc}, one can easily show that for a regular section $\psi$, such the $Z(\psi)$ is smooth and of expected codimension, we have $Z(\psi)=\Crit(\Tot(\V^\vee),\chk\psi)$. So, \Cref{thm:knorrer} gives a mathematical equivalence between the category of a (certain kind of) LG-model and the category of its critical locus, which is something that is expected to happen physically. 
\end{remark}

\section{Setting up the GLSM}\label{sec:GLSMsetup}
We start with a general member $X$ of the Calabi-Yau threefold family $\X$, discussed in \Cref{sec:X}. Recall that $X$ is given as the vanishing locus of a regular section $s\in \Gam(G(2,V_6),\mathcal{U^\vee}(\xi)\oplus \mathcal{O}(\xi)^{\oplus 3})$. By \Cref{rem:vansihingLocusVsCritLocus}, it can alternately be described as the critical locus of a regular function $\chk s_-$ on $\Tot(\mathcal{U}(-\xi)\oplus \mathcal{O}(-\xi)^{\oplus 3})$. Here, $\chk s_-$ is naturally induced by $s$ (by \Cref{lem:section}) and is linear on the fibers. Turns out, this total space can be realized as a GIT quotient of an affine space $V$ by a $G$-action, as described below. This is the starting point of the so-called "GLSM framework" as whenever we have a GIT quotient we may compute the quotient with respect to some other character of the group, to obtain something new. By the end of this section, we will have "retrieved" the dual Calabi-Yau threefold $Y$ (as defined in the beginning of \Cref{sec:Y}) by working entirely within the GLSM framework, thus confirming our physically motivated expectation laid out earlier, at the end of \Cref{sec:mirrorEquivalence}.

\subsection{Variation of GIT}
Denote by $E$, the $G$-representation $(S \otimes \det(S)) \oplus \det(S)^{\oplus 3}$, where $S \simeq \C^2 $ is the fundamental $G$-representation. And denote by $V_6$, a vector space of dimension 6. We then define $V$ to be the affine space whose underlying vector space is the following $G$-representation:
\begin{equation}
    V \eq \Hom(S, V_6) \oplus E
\end{equation}
So, the $G\curvearrowright V$ is explicitly given as $g.(B\sp,\sp\ul{v}\sp,\sp\ul{w}) := (B.g^{-1}\sp, \sp \det(g).g.\ul{v} \sp, \sp \det(g).\ul{w})$. The action of a general linear group on an affine space has two possible semi-stable loci and hence two GIT quotients, corresponding to the choice of $G$-characters $\det^{\pm 1}$. Denote the semi-stable locus of $V$ with respect to character $\det^{\pm1}$ by $V_\pm$. Also, denote by $\Hom(S, V_6)^o$, the subset of full rank matrices in $\Hom(S, V_6)$

\begin{proposition}\label[proposition]{prop:VGIT}
The semi-stable loci $V_\pm$ are as follows:
\begin{itemize}
    \item $V_-=\{\coordsOfV\in V \sp| \sp rank(B) = 2\}=\Hom(S,V_6)^o \times E$ 
    \item $V_+=\{\coordsOfV \in V \sp | \sp (B\ul v, \ul w) \neq 0\}$
\end{itemize}    
\end{proposition}
\begin{proof}
     A linearized line bundle on an affine space $V$ with a $G$-action corresponds to choice of a character $\mu\in \hat G$. Given such a choice, the Hilbert-Mumford criterion tells us that:
\begin{equation*}
u \in V \text{ is semi-stable} \iff 
\begin{aligned}[t]
    &\forall \text{ 1-PS subgroup } \lambda: \mathbb{C}^* \to G \text{ such that } \lim_{t \to \infty} \mu^{-1}(\lambda(t)) = 0, \\
    &\text{we have that } \lim_{t \to \infty} \lambda(t).u \text{ does not exist.}
\end{aligned}
\end{equation*}

Let us apply this criterion to our situation.\\
\\
\noindent \textit{Computing $V_-$}: Consider an arbitrary 1-parameter subgroup $\lam$ such that $\lim\limits_{t \to \infty} \det(\lambda(t))=0$ and an arbitrary element $\coordsOfV\in V$.

First, suppose $B$ has full-rank. It is equivalent to saying that the induced element $\det(B)\in \Hom(\det (S), \wedge ^2V_6)$ is non-zero. The induced action on $\det(B)$ is $\lam(t).\det(B) = \det(\lam(t^{-1})) \sp \det(B)$. This doesn't have a limit as $t\to \infty$. So, $\coordsOfV$ is semi-stable. Hence, $\lim\limits_{t\to \infty}B.\lam (t^{-1})$ doesn't exist and therefore $\lim\limits_{t\to \infty}(B.\lam (t^{-1})\sp, \sp \det(\lam (t)).\lam (t).\ul{v} \sp, \sp \det(\lam (t)).\ul{w})$ doesn't exist.

Now suppose $B$ doesn't have full rank. Then, there exists a basis of $S$, where $B$ (as a $6\times2$ matrix) looks like $[* \sp\sp0]$. Then, for $$\lam(t):=\matTwo{1}{0}{0}{t^{-1}}$$we have
$\lim\limits_{t\to \infty}(B.\lam (t^{-1})\sp, \sp \det(\lam (t)).\lam (t).\ul{v} \sp, \sp \det(\lam (t)).\ul{w})=(B,0,0)$. In particular, this limit exists. Therefore, $V_-=\{\coordsOfV\in V \sp| \sp rank(B) = 2\}$\\\\

\noindent \textit{Computing $V_+$}: Consider an arbitrary 1-parameter subgroup $\lam$ such that $\lim\limits_{t \to \infty} \det(\lambda(t))^{-1}=0$ and an arbitrary element $\coordsOfV\in V$. First suppose $(B\ul v,\sp\ul w) \neq0$. The induced action on it is $\lam(t).(B\ul v, \sp \ul w)=\det(\lam(t))\sp(B\ul v, \sp\ul w)$. This doesn't have a limit as $t\to \infty$. Hence, $\lim\limits_{t\to \infty}(B.\lam (t^{-1})\sp, \sp \det(\lam (t)).\lam (t).\ul{v} \sp, \sp \det(\lam (t)).\ul{w})$ doesn't exist. 

Now suppose $B\ul v =\ul w = 0$. Suppose moreover, $\ul v =0$. Then, for $$\lam(t):=\matTwo{1}{0}{0}{t}$$ we have $\lim\limits_{t\to \infty}(B.\lam (t^{-1})\sp, \sp \det(\lam (t)).\lam (t).\ul{v} \sp, \sp \det(\lam (t)).\ul{w})=([* \sp\sp0],0,0)$. In particular this limit exists. Assuming the other possibility $\ul v\neq 0$, there exists a basis of $S$, where $\ul v=(0,v)$, for some $v\neq0$. Then, $B\ul v=0\Rightarrow B=[* \sp\sp 0]$. Then, for $$\lam(t):=\matTwo{t^2}{0}{0}{t^{-1}}$$ we have $\lim\limits_{t\to \infty}(B.\lam (t^{-1})\sp, \sp \det(\lam (t)).\lam (t).\ul{v} \sp, \sp \det(\lam (t)).\ul{w})=(0,(0,v),0)$. In particular, again this limit exists. Therefore, $V_+=\{\coordsOfV \in V \sp | \sp (B\ul v, \ul w) \neq 0\}$.
 
\end{proof}

\subsection{The ``negative" phase }
By the description of $V_-$ in \Cref{prop:VGIT}, the GIT quotient $V_-/G$ is just the total space of the vector bundle  associated to the $G$-representation $E$, over the space $\Hom(S,V_6)^o/G=G(2,V_6)$. In other words, we have:
\begin{equation}
    V_-/G = \Tot(\mathcal{U}(-\xi)\oplus \mathcal{O}(-\xi)^{\oplus 3})\xrightarrow{} G(2,V_6)
\end{equation}
Equipping $V_-/G$ with the function $\check{s}_-$ and a $\C^*_R$-action by fiber-wise dilation as $\lam.(B,\ul v, \ul w):=(B,\lam^2\ul v, \lam^2\ul w)$ turns it into a LG model, which is an example of a Kn\"orrer model (\Cref{def:knorrerModel} ). Now as the closed complement of  $V_-\subset V$ has codimension $\geq$ 2, the $G$-invariant function $\chk s_-$ on $V_-$ extends uniquely to a $G$-invariant function $\chk s : V\xrightarrow{} \C$, due to Hartog's lemma. Also, the $\C^*_R$-action on $V_-/G$ has a natural extension to an action on $V/G$. So, we get a GLSM $(V/G,\chk s)$ and we say that the LG model $(V_-/G,\chk s_-)$ is the ``negative" phase of this GLSM. We have
\begin{lemma} \label[lemma]{critNeg}
    $\Crit(V_-/G, \sp \chk s_-) = X$. Moreover, $X$ embeds in $V_-/G$ through its natural embedding in $G(2,V_6)$, followed by the embedding of  $G(2,V_6)$ in $V_-/G$ as the zero section of the vector bundle projection.
\end{lemma}
\begin{proof}
    To compute the critical locus, we first need to describe $\chk s_-$ more explicitly. Writing $G(2,V_6)$ as the quotient $\Hom(S,V_6)^o/G$, we can realize $\Tot(\mathcal{U}^\vee (\xi)\oplus \mathcal{O}(\xi)^{\oplus 3})$ as the quotient $\Hom(S,V_6)^o\times_G E^\vee$. In this view, the section $s_-$ is simply a $G$-equivariant function $s_-:\Hom(S,V_6)^o\xrightarrow{}E^\vee$. This further corresponds to a $G$-invariant function $\chk s_-: \Hom(S,V_6)^o\times E = V_-\xrightarrow{} \C$ which is linear in the coordinates of $E$. Equivalently, it is a regular function on $V_-/G=\Tot(\mathcal{U}(-\xi)\oplus \mathcal{O}(-\xi)^{\oplus 3})$, which is linear on the vector space fibers. By the above chain of identifications, it is clear that $\Crit(V_-/G, \chk s_-)=Z(s_-)\eq X$. 
\end{proof}
The simplicity of the negative phase, lies in the fact that the GIT quotient $V_-/G$ is a vector bundle and the superpotential $\chk s_-$ is linear on the fibers of this bundle. Hence, we can apply Kn\"orrer periodicity (\Cref{thm:knorrer}) to relate the derived category of the critical locus $X=Z(s_-)$ with the B-brane category of the phase $(V_-/G, \chk s_-)$. That is, we have
\begin{corollary}
    There is an equivalence of categories $D(X)\simeq D(V_-/G,\chk s_-)$
\end{corollary}

\subsection{The ``positive" phase}
We may restrict the function $\chk s:V\rightarrow \C$ to $V_+$ to obtain a $G$-invariant function $\chk s_+$ on it. Let us also restrict the $\rchar$-action on $V/G$ to $V_+/G$ to get a $\rchar$-action on the latter. Equipping $V_+/G$ with $\chk s_+$ and the above $\C^*_R$-action turns it into another LG-model, which we will refer to as the ``positive" phase of our GLSM. Putting the GLSM and its phases together, we have inclusions of LG models

\begin{equation}
    \begin{tikzcd}
        (V_-/G, \chk s_-)\arrow[r, hook, "i_-"]  & (V/G, \chk s) & (V_+/G, \chk s_+) \arrow[l, hook', "i_+"']
    \end{tikzcd}
\end{equation}

which induce the following maps between the corresponding B-brane categories.
\begin{equation}
    \begin{tikzcd}
        D(V_-/G, \chk s_-)& D(V/G, \chk s) \arrow[l, "Li^*_-"']  \arrow[r, "Li^*_+"] & D(V_+/G, \chk s_+) 
    \end{tikzcd}
\end{equation}

\subsubsection{Birational Modifications}
The positive phase is more complicated than the negative phase. The quotient stack $V_+/G$ appearing here is not as nice, as some points of $V_+$ have non-trivial stabilizers with respect to the $G$-action. For example, a point like $(0, 0, \ul w)$ is stabilized by $SL(2)$. So, $V_+/G$ is ``stacky" along such loci. And although $V_+/G$ has a natural projection onto $\P^8$ given as $\coordsOfV \mapsto (B\ul v, \ul w)$, it is not a vector bundle over it. So, we cannot apply Kn\"orrer periodicity to relate the B-brane category of this phase with the derived category of the critical locus. Moreover, it is not clear how to retrieve the geometry of the second Calabi-Yau threefold from the critical locus in this quotient. For these reasons, we will birationally modify the quotient until we get rid of such problems. Before we do that, we will set up some notations below, which also sets up the coordinate naming conventions we will stick to throughout this paper. Here, $L^\vee\simeq \C^3$.

\begin{itemize}
\item For a general $B\in \Hom(S, V_6)$, denote $\ul b \equiv$ the first column of $B$ and $\ul d \equiv $ the second column of $B$. Then, identifying $\Hom(S, V_6)=V_6\oplus V_6$ (as vector spaces) we may write $B = (\ul b, \ul d)$
    \item $W \equiv \{(\ul b, \ul d, v,\ul w)\in V_6\oplus V_6 \oplus \mathbb{C} \oplus L^\vee \ | \ (v \ul b, \ul w) \neq 0\}$
    \item $U\equiv \{(\ul b, \ul d,v,\ul w)\in V_6\oplus V_6 \oplus \mathbb{C} \oplus L^\vee \ | \ (v, \ul w) \neq 0, \sp \ul b \neq 0 \}$
    \item $\Gamma \eq \left\{\matTwo{\alpha}{\beta}{0}{\delta}\in G \ |\ \alpha\delta\neq0\right\}$
\item $T \eq \left\{\matTwo{\alpha}{0}{0}{\delta}\in G \ |\ \alpha\delta\neq0\right\}$
\end{itemize}
It turns out that the correct birational modification is given by the following sequence of maps of quotient stacks, which we describe below.
\begin{equation}\label{eqn:seqOfMaps}
    \kb \xrightarrow{} W/\Gamma \xrightarrow{} V_+/G
\end{equation}
$W$ may be seen as a closed sub-scheme of $V_+$ given by vanishing of $v_2$ (the second coordinate of $\ul v$). The stabilizer of $W$ in $V_+$ is exactly $\Gamma$. Hence, we have a natural surjective map of quotient stacks $W/\Gamma \xrightarrow{} V_+/G$, which is given at the level of points as $[(\ul b, \ul d, v,\ul w)/\Gamma] \mapsto [(\ul b, \ul d, (v,0),\ul w)/G]$. The other map $\kb \xrightarrow{} W/\Gamma$ is an open inclusion of quotient stacks and is induced by the open inclusion of the $\Gamma$-invariant open subspace $U$ in $W$.

\subsubsection{Applying Kn\"orrer Periodicity}
Pulling back the function $\chk s_+$ through the maps in \eqref{eqn:seqOfMaps}, we get a function $\omega$ on $\kb$. We will now focus on the pair $(\kb, \omega)$. Note that the $\Gamma$-action on $U$ is free and the points of $U$ are stable under this action (the latter can be checked easily, using the Hilbert-Mumford criterion). So, $\kb$ is a scheme.

It will be useful to write down the $\Gamma$-action on $U$ explicitly:
\begin{equation} \label{eqn:gamSym}
    {\begin{pmatrix}
\alpha & \beta \\
0 & \delta
\end{pmatrix}}.(\ul b, \ul d,v,\ul w) := (\alpha^{-1}\ul b, \sp \delta^{-1}\ul d - \beta \alpha^{-1}\delta^{-1} \ul b, \sp\alpha^2\delta v, \sp \alpha \delta \ul w)
\end{equation}
Denote:
\begin{itemize}
    \item $U^{'}  \eq \{(\ul b, v,\ul w)\in V_6 \oplus \mathbb{C} \oplus L^\vee \ | \ (v, \ul w) \neq 0, \sp \ul b \neq 0 \}$
    \item $U^{''}  \eq \{(\ul x,\ul w)\in V_6 \oplus L^\vee \ | \ (\ul x, \ul w) \neq 0 \}$
\end{itemize}
We have natural actions:
\begin{itemize}
    \item $T \curvearrowright U'$ : $(\alpha, \delta). (\ul b, v,\ul w) := (\alpha^{-1}\ul b, \sp\alpha^2\delta v, \sp \alpha \delta \ul w)$
    \item $\mathbb{C}^* \curvearrowright U^{''}$ : $\lam. (\ul x,\ul w) := (\lam \ul x, \lam \ul w)$

\end{itemize}
Now consider the following sequence of projections.
\[
\begin{array}{ccccc}
U & \xrightarrow{\til p} & U^{'} & \xrightarrow{\til\pi'} & U^{''} \\
(\ul b, \ul d, v, \ul w) & \longmapsto & (\ul b, v,\ul w) & \longmapsto & (v\ul b,\ul w)
\end{array}
\]
This sequence is equivariant with respect to the following sequence of projections of algebraic groups.
\[
\begin{array}{ccccc}
\Gamma & \xrightarrow{} & T & \xrightarrow{} & \C^* \\
{\begin{pmatrix}
\alpha & \beta \\
0 & \delta
\end{pmatrix}} & \longmapsto & (\alpha, \delta)& \longmapsto & \al\del
\end{array}
\]
Hence, it induces a sequence of maps of quotient stacks. Now, the quotient $U^{''}/\C^*$ is just $\P(V_6\oplus L^\vee)=\P^8$ and the map $\pi':U^{'}/T\xrightarrow{} U^{''}/\C^*$ can be identified with the blow-up $\til{\P^8}$, of this $\P^8$ along the plane $\P(L^\vee)=\P^2$. This identification is done by writing $\blp = \{([\ul b],[\ul x,\ul w])\in \P(V_6)\times \P(V_6\oplus L^\vee) \sp | \sp \ul b \wedge \ul x = 0 \}$ and mapping: 
\begin{equation}
    U^{'}/T \owns(\ul b, v, \ul w)\mapsto ([\ul b],[v\ul b,\ul w])\in \blp
\end{equation}
So, finally we get an induced sequence of maps of quotient stacks:
\begin{equation}
    \kb\xrightarrow{p} \blp\xrightarrow{\pi'}\P^8
\end{equation}
$\blp$ also has a projection onto $\P(V_6)=\P^5$, by which it can be realized as the projective bundle $\P(\O_{\P^5}(-h) \oplus {\O_{\P^5}}^{\oplus 3})$. Denote this projection by $\pi:\blp \xrightarrow{} \P^5$. Recall the notations $h$ for the hyperplane class of the $\P^5$ and $H$ for the hyperplane class of the $\P^8$, introduced in \Cref{sec:geometry}. Whenever twists by these appear on sheaves defined on other varieties, we mean twists by their pullback under the obvious maps available in the situation. We will also record, for future use, the representation theoretic view of these twists.

\begin{observation}  On $\kb$, under the identification $Coh(\kb)\simeq Coh_\Gam(U)$, 

\begin{itemize}
    \item $\O(h)\simeq \C_{\alpha^{-1}}\otimes \O_U$, where $\C_{\alpha^{-1}}$ is the 1-dimensional $\Gam$-representation: $\GamEl.u := \alpha^{-1}u$
    \item $\O(H)\simeq\C_{\alpha \delta}\otimes \O_U$, where $\C_{\alpha\delta}$ is the 1-dimensional $\Gam$-representation: $\GamEl.u := \alpha \delta u$
\end{itemize}    
\end{observation}

\begin{lemma}
$\kb\xrightarrow{p} \blp$ can be realized as the total space of the bundle $\pi^*\wt \Q (-h-H)$, where $\wt \Q $ is the quotient bundle on $\P^5$, defined under the identification $\P^5=G(1,V_6)$. 
\end{lemma}
\begin{proof}
    We will view $\blp$ under the identification $\blp\simeq U^{'}/T$ and explicitly describe the fibers over points in it. An arbitrary point in it represents an orbit $T.(\ul b, v, \ul w)$ and the fiber over it is just the pullback of the corresponding map:
    \begin{equation}
        \{(\ul b, v, \ul w)\}\rightarrow U^{'}\rightarrow U^{'}/T
    \end{equation}
    We have the following fiber product diagram:   
    \[
\begin{tikzcd}
F \arrow[r] \arrow[d] \arrow[dr, phantom, "\lrcorner", very near start] & \{(\ul b, v, \ul w)\} \arrow[d] \\
T \times_{\Gam} U \arrow[r] \arrow[d] \arrow[dr, phantom, "\lrcorner", very near start] & U' \arrow[d] \\
U/\Gam \arrow[r, "p"'] & U'/T
\end{tikzcd}
\]
The bottom square in this diagram is standard, described by the following prescriptions:
\begin{itemize}
    \item $\Gam \curvearrowright T\times U$ is given as
    \begin{equation}
        \matTwo{\alpha'}{\beta'}{0}{\delta'}.\left(\matTwo{\alpha}{0}{0}{\delta}, \ (\ul b, \ul d, v,\ul w)\right):=\left(\matTwo{\alpha\alpha'^{-1}}{0}{0}{\delta\delta'^{-1}},\  \matTwo{\alpha'}{\beta'}{0}{\delta'}.(\ul b, \ul d, v,\ul w)\right)
    \end{equation}
    \item $T\times_\Gam U\rightarrow U'$ is given as $(t,u)\mapsto t.u$

    \item $T\times_\Gam U\rightarrow \kb $ is induced by the projection $T\times U \rightarrow U$.
\end{itemize}
So, the fiber is 
\begin{equation}
    F = \left\{ \left( \matTwo{\alpha}{0}{0}{\delta}, \matTwo{\alpha^{-1}}{0}{0}{\delta^{-1}} \cdot (\ul b, \ul d, v, \ul w) \right) \ |\ \alpha\delta\neq0,\ \ul d\in V_6 \right\}/\Gam
\end{equation}

Notice that $T\times_\Gam U$ has a natural $T$-action on it by left multiplication with the $T$ factor. $F$ inherits this action and under it, we have the following natural isomorphism of $T$-representations:
\begin{equation}
    \begin{array}{ccc}
     F&\xrightarrow{\sim}  &\C_{\delta^{-1}}\otimes V_6/\langle\ul b\rangle\\
     \left( \matTwo{1}{0}{0}{1}, \  (\ul b, \ul d, v, \ul w) \right) & \longmapsto & 1\otimes  [\ul d]
\end{array}
\end{equation}
Clearly, $\C_{\delta^{-1}}\otimes V_6/\langle\ul b\rangle$ is just the fiber of $\Tot(\pi^*\wt \Q (-h-H))$ over the point we started with.
\end{proof}
In the following lemma we compute the critical locus of $(U/\Gam,\omega)$. The variety $\bly$ appearing there is described explicitly in the proof. It turns out, that it is the same variety $\bly$ of \Cref{prop:gorCY3}, given that we make the same choice of section $s$ (that defines $X$) as we have made in this section.

\begin{lemma}\label[lemma]{lem:critPos}
$\Crit(\kb, \omega) = \wtil Y$. Moreover, $\bly$ embeds in  $\kb$ through its natural embedding in $\blp$, followed by the embedding of  $\blp$ in $\kb$ as the zero section of the vector bundle projection $p$.
\end{lemma}
\begin{proof}
Continuing from the proof of \Cref{critNeg}, let us start by describing the induced superpotential $\chk s$ on the central quotient stack $V/G$ even more explicitly. Then, we will pull it back to $\kb$. For readability, we will adopt here the ``Einstein's convention" of summing up repeated indices. 

Recall that $\chk s$ can be seen as a $G$-invariant function $V \eq \Hom(S, V_6) \oplus E\rightarrow \C$ that is linear in coordinates of $E$. In this view, it has the following general form, written in coordinates $\coordsOfV$ of  $V$: 
\begin{equation}\label{eqn:superpot}
    \chk s_- = s_1(B)v^1 + s_2(B)v^2 + l_1(B)w^1 + l_2(B)w^2 + l_3(B)w^3
\end{equation}
where,
\begin{itemize}
    \item $s_j(B)=\eta_{ipq}b^pd^qB^i_j$ with $\eta_{ipq}$ being general scalars such that $\eta_{ipq}=-\eta_{iqp}$.
    \item $l_j(B)=l_{jpq}b^pd^q$ with $l_{jpq}$ being general scalars such that $l_{ipq}=-l_{iqp}$.
\end{itemize}
Note that $(s_1(B),s_2(B))$ corresponds to a general section of $\bunNeg (1)$ and each $l_j(B)$ corresponds to general section of $\O(1)$. Pulling it back to $U$ amounts to putting $v^2=0$ and replacing $v^1$ with $v$. So, finally the superpotential on $U/\Gam$ is given by the following  $\Gam$-invariant function $U\rightarrow \C$.
\begin{equation}
    \omega = (\eta_{ipq}vb^i+l_{ipq}w^i)b^p d^q
\end{equation}
Notice that it is linear in $\ul d$, which are the coordinates of the fibers of $\kb\xrightarrow{} U^{'}/T \simeq \blp$. So, the critical locus is the locus where the coefficients, considered as functions on the base, vanish. That is,
\begin{equation}
    \Crit(\kb, \omega)=V(\{(\eta_{ipq}vb^i+l_{ipq}w^i)b^p\}_q)\subset U^{'}/T
\end{equation}
Now let us re-describe it under the identification $U^{'}/T \simeq \blp$. Working in $v\neq 0$ is equivalent to working in $\ul x \neq 0$ in $\blp$ i.e. away from the exceptional locus. Here, the critical locus equations are equivalent to the equations $(\eta_{ipq}x^i+l_{ipq}w^i)b^p=0$ and $(\eta_{ipq}x^i+l_{ipq}w^i)x^p=0$. Either of these vanishings forces the skew-symmetric matrix $M_{pq}\equiv (\eta_{ipq}x^i+l_{ipq}w^i)_{pq}$ to be degenerate. So, in addition to the above two sets of equations, $\Pf(M)=0$ must also hold away from the exceptional locus. So, it must also hold in the closure. Therefore, we conclude that under the identification $U^{'}/T \simeq \blp$, we have:
\begin{equation}
    \Crit(\kb, \omega)=V(M\ul b, M\ul x, \Pf(M))\subset \blp
\end{equation}
Above is just the blow-up of $Y\eq V(M\ul x, \Pf(M))$ along the linear section $ C\eq V(\ul x)\cap Y$. By the discussion in \Cref{sec:geometry}, it is clear that $Y$ is a general Calabi-Yau threefold in the family described there and $C$ is the elliptic curve in it.
\end{proof}
Considering $(U/\Gam, \omega)$ with the induced $\C^*_R$-action (from $V_+/G$), turns it into a LG model. Let us inspect this induced $\rchar$-action closely. It looks like
\begin{equation}
    \rchar \curvearrowright U/\Gam \ \sim \ \lam.(\ul b,\ul d, v, \ul w) = (\ul b, \ul d, \lam^2 v, \lam^2 \ul w)
\end{equation}

To correctly "measure" the R-charge along the fiber coordinates $\ul d$ we must first "gauge-fix" the base coordinates $(\ul b, v, \ul w)$, as these coordinates are seen up to gauge symmetry. In more precise terms, we use the $\Gam$-symmetry \eqref{eqn:gamSym} to write 
\begin{equation}
\begin{aligned}
\lam.(\ul b,\ul d, v, \ul w) &= (\ul b, \ul d, \lam^2 v, \lam^2 \ul w)\\
&=(\alpha ^{-1}\ul b,\  \delta^{-1}\ul d-\beta\alpha^{-1}\delta^{-1} \ul b,\  \alpha^2\delta\lam^2 v,\  \alpha\delta\lam^2 \ul w)\\
&=(\ul b,\lam^2\ul d, v, \ul w)
\end{aligned}
\end{equation}
where the last equality is obtained by choosing $\alpha =1$, $\beta =0$ and $\delta = \lambda^{-2}$. So, we see that the fiber coordinates $\ul d$ have R-charge 2 and there is no R-charge along the base. Also, notice that the superpotential $\omega$ is linear along the fibers. Therefore, after our stacky birational modifications in the positive side, we have finally arrived at a Kn\"orrer model $(U/\Gam, \omega)$ and hence are now in a position to apply Kn\"orrer periodicity (\Cref{thm:knorrer}). We get

\begin{corollary}\label[corollary]{cor:applyingKnorrer}
    $D(\bly)\simeq D(U/\Gam, \omega)$
\end{corollary}

\subsection*{The Overall Picture}
Following diagram captures the picture developed so far.
\begin{equation*}
\begin{tikzcd}[row sep=large, column sep=large]
    & (V/G, \check{s}) & & & \\
    (V_-/G, \check{s}_-) 
        \arrow[ur, hook, "i_-"'] 
        \arrow[rr, dashed, <->, "\sim"] 
        \arrow[d, "\pi"', xshift=1.0em] 
    & & 
    (V_+/G, \check{s}_+) 
        \arrow[ul, hook', "i_+"] 
        \arrow[d] 
    & 
    (W/\Gamma, \omega) 
        \arrow[l, "f"] 
    & 
    (U/\Gamma, \omega) 
        \arrow[l, hook', "i"] 
        \arrow[d, "p"', xshift=-1.2em] 
    \\
    X \subset G(2,V_6) 
    & & 
    \P ^8 
    & & 
    \widetilde{\P }^8 \supset \widetilde{Y} 
        \arrow[ll] 
\end{tikzcd}
\end{equation*}

\section{Construction of the Functor}\label{sec:constructionOfFunctor}
In this section, we will construct a fully-faithful functor $D(V_-/G,\chk s_-) \xrightarrow{} D(U/\Gam, \omega)$. We will do it in five steps. In the first four steps, we will construct a fully-faithful functor between the corresponding derived categories. And in the last step, we will ``turn on the superpotential" to produce a fully-faithful functor between the B-brane categories.

Before we begin, let us review the notion of window categories which is central to our construction. As both $V_\pm/G$ are smooth, any object in $D(V_\pm/G)$ is isomorphic to a complex of vector bundles and hence can be extended to an object in $D(V/G)$ (after extending the vector bundles themselves, one extends the morphisms between them by using Hartog's lemma). In other words, we have essentially surjective restriction functors:
\begin{equation}
    Li^*_{\pm}:D(V/G)\rightarrow D(V_\pm/G)
\end{equation}
So, one might hope that there exist sub-categories $\W_\pm\subset D(V/G)$ such that $Li^*_{\pm}|_{\W_\pm}$ are equivalences or at least fully-faithful. 
\begin{definition}
We say that $\W_\pm$ are \textit{window} sub-categories if indeed $Li^*_{\pm}|_{\W_\pm}$ are equivalences. Instead, if we only get fully-faithful functors we refer to $\W_\pm$ as \textit{partial windows}.     
\end{definition}
In general, the status of a sub-category being a (partial) window depends on the side ($+$ or $-$) of the variation of GIT (VGIT) we are on. Halpern-Leistner's theory \cite{HL} gives a general construction of windows for the quotient of a smooth quasi--projective variety by a reductive group, using certain ``grade restriction rules" on the weight space of the group. His results, in particular, imply that windows do exist for both sides of our VGIT.

\subsection{Step 1: Window Functor for the Positive Side}
Here we will construct a partial window for the positive side of the VGIT using Halpern-Leistner's theory \cite{HL}. This will turn out (as we will see in step 2) to be a full window for the negative side of the VGIT. To apply this theory we first need to construct the Kempf-Ness stratification of the unstable locus $V\ \setminus\ V_+$.

\subsubsection{Kempf-Ness Stratification}
We will apply the general formalism of Kempf-Ness stratification for affine GIT quotients, as given in \cite{Hoskins}, to our situation.

Fix the chosen maximal torus in $G$ to be $T$, the subgroup of diagonal matrices. To obtain the indices of the Kempf-Ness stratification we will focus on the induced action $T\curvearrowright V$. We have the following weight space decomposition of $V$ with respect to the torus action:
\begin{equation}
    \begin{array}{ccccccccccc}
V & = &\C^6_{\ul b} &\oplus & \C^6_{\ul d} &\oplus & \C_{v_1} &\oplus & \C_{v_2} &\oplus & C^3_{\ul w}\\
weight&:&  (-1,0) &&(0,-1)&&(2,1)&&(1,2)&&(1,1)
\end{array}
\end{equation}

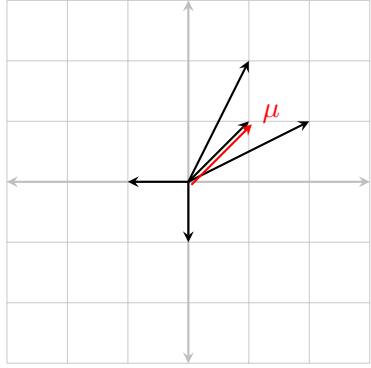
\begin{figure}[htbp]
    \centering
    \begin{tikzpicture}[>=stealth, scale=0.8]
        \draw[thin, gray!40] (-3,-3) grid (3,3);

        \draw[<->, gray!50, thick] (-3,0) -- (3,0);
        \draw[<->, gray!50, thick] (0,-3) -- (0,3);
        
        \draw[->, thick, black] (0,0) -- (-1,0);
        \draw[->, thick, black] (0,0) -- (0,-1);
        \draw[->, thick, black] (0,0) -- (1,2);
        \draw[->, thick, black] (0,0) -- (2,1);
        
        \draw[->, thick, black] (0,0) -- (1,1);
        \draw[->, thick, red] (0.05,-0.05) -- (1.05,0.95) node[right, yshift=4pt, red] {$\mu$};

    \end{tikzpicture}
    \caption{T-weights of $V$}
    \label{fig:vector_configuration}
\end{figure}

Denote the character and co-character lattice of $T$ by $\X^*(T)$ and $\X_*(T)$ respectively and do the same analogously for $G$. Note that a co-character of a group is just a 1-parameter subgroup (1-PS) in the group. We will identify both $\X^*(T)$ and $\X_*(T)$ as $\Z^2$ in the obvious way. Then the natural pairing between them is simply the dot product on $\Z^2$. The chosen character of this phase, $\det \in \X^*(G)$, restricts to the torus character $\mu=(1,1)\in \X^*(T)$. So, the normalized Hilbert-Mumford function for $T\curvearrowright V$ is

\begin{equation}
    HM_T(v) := \inf \frac{m+n}{\sqrt{m^2+n^2}}
\end{equation}
where $v\in V$ and the infimum is taken over all $(m,n)\in \X_*(T)$ such that $\lim_{t\rightarrow 0}\matTwo{t^m}{0}{0}{t^n}.v$ exists. This latter collection of 1-PSs is called the \textit{cone of allowable 1-PSs} for $v$ and is more explicitly given as:
\begin{equation}
    C_v\equiv\{(m,n)\in \X _*(T)\ |\ (m,n).(p,q)\geq 0 \text{ for all weights }(p,q) \text{ appearing in } v \}
\end{equation}
Note that a point $v$ is $T$-unstable if and only if $HM_T(v)<0$. For a $T$-unstable point $v$, we say that a 1-PS $(m_0,n_0)$ \textit{maximally destabilizes} $v$ if it achieves the above infimum, that is $ HM_T(v) = (m_0+n_0)/\sqrt{m_0^2+n_0^2}$. The $G$-conjugacy classes of such 1-PSs will form the indices of the desired stratification. 
For a 1-PS $\lam$ denote:
\begin{itemize}
    \item $Y_\lam\equiv$ the locus of points in $V$ that maximally destabilized by $\lam$
    \item $S_{[\lam]} \equiv G.Y_\lam=$ the locus of points in $V$ that are maximally destabilized by some $G$-conjugate of $\lam$. 
    \item  $Z_\lam\equiv$ the locus of points in $Y_\lam$ that are fixed by $\lam$.   
\end{itemize}
Observe that, for any $\lam$, we have a sequence of embeddings:
\begin{equation}
    Z_\lam\subseteq Y_\lam\subseteq S_{[\lam]}
\end{equation}
For any $T$-unstable point $v$ there exists a unique 1-PS that maximally destabilizes $v$ (see lemma 2.13 in \cite{Hoskins}). Explicitly, this is the 1-PS in $C_v$ having maximum angle with $\mu$.  Clearly, this it is determined just by the torus weights appearing in $v$, which is some subset of the finitely many torus weights appearing in the weight space decomposition of $V$. In particular, this shows that there are finitely many 1-PSs that are capable of maximally destabilizing points in $V$ and we can iterate through all of them by iterating through all subsets of weights that can appear in a general $T$-unstable point.  \Cref{tab:max1PSs}\footnote{In this table, the set of weights $\{(0,0)\}$ corresponds to the origin in $V$.} lists them all.

\begin{table}[htbp]
    \centering
    
\begin{tabular}{|c|c|c|}\hline
        \rule[-1.2ex]{0pt}{3.7ex} Weights of a T-unstable point& Maximally destabilizing 1-PS& HM function value\\ \hline
        \rule[-1.2ex]{0pt}{3.7ex}$\{(0,0)\}$ & $(-1,-1)$ & $-\sqrt{2}$\\\hline
        \rule[-1.2ex]{0pt}{3.7ex}$\{(-1,0)\}$ & $(-1,-1)$ & $-\sqrt{2}$\\\hline 
        \rule[-1.2ex]{0pt}{3.7ex}$\{(0,-1)\}$ & $(-1,-1)$ & $-\sqrt{2}$ \\\hline
        \rule[-1.2ex]{0pt}{3.7ex}$\{(2,1)\}$ & $(1,-2)$ & $-1/\sqrt{5}$\\\hline
        \rule[-1.2ex]{0pt}{3.7ex}$\{(1,2)\}$ & $(-2,1)$ & $-1/\sqrt{5}$\\\hline
        \rule[-1.2ex]{0pt}{3.7ex}$\{(0,-1), (-1,0)\}$ & $(-1,-1)$ & $-\sqrt{2}$ \\ \hline
        \rule[-1.2ex]{0pt}{3.7ex}$\{(-1,0),(1,2)\}$ & $(-2,1)$ & $-1/\sqrt{5}$\\\hline
        \rule[-1.2ex]{0pt}{3.7ex}$\{(0,-1),(2,1)\}$ & $(1,-2)$ & $-1/\sqrt{5}$\\\hline
    \end{tabular}
\caption{Maximally destabilizing 1-PSs}
\label{tab:max1PSs}
\end{table}

Considering the 1-PSs listed in \Cref{tab:max1PSs} up to conjugation we obtain just two classes: $[(-1,-1)]$ and $[(1,-2)]$. So, now we need to compute the associated loci $Z_\lam$, $Y_\lam$, $S_{[\lam]}$ for $\lam=(-1,-1),(1,-2)$. We do this in \Cref{tab:KempfNess} using the data in \Cref{tab:max1PSs}. The loci $S_{[\lam]}$ are exactly the strata in the Kempf-Ness stratification that we had set out to construct.
\begin{table}[htbp]
    \centering
\begin{tabular}{|c|c|c|c|}\hline
        \rule[-1.2ex]{0pt}{3.7ex} $\lam$& $Z_\lam$ &$Y_\lam$ & $S_{[\lam]}$\\ \hline
        \rule[-1.2ex]{0pt}{3.7ex} $(-1,-1)$ & $\{0\}$&$V(v_1,v_2,\ul w)$ & $V(v_1,v_2,\ul w)$\\\hline
        \rule[-1.2ex]{0pt}{3.7ex} $(1,-2)$ & $V(\ul b, \ul d, v_2, \ul w)\cap \{v_1\neq 0\}$ & $V(\ul b, v_2, \ul w)\cap \{v_1\neq 0\}$ &$V(\ul b v_1+\ul d v_2, \ \ul w)\cap \{\ul v\neq 0\}$ \\\hline
    \end{tabular}
\caption{Loci associated to Kempf-Ness Stratification of $V^{us}$ }
\label{tab:KempfNess}
\end{table}
\subsubsection{HL Window}
Computing the window using Halpern-Leistner's theory amounts to computing the numbers $\eta_\lam$ appearing in the so-called ``grade restriction rules" that define such windows (see  \cite[Lemma 2.9]{HL}). These numbers are defined as
\begin{equation}
    \eta_\lam:=wt_\lam(\det\ \N^\vee_{S_{[\lam]}/V}|_{Z_\lam})
\end{equation}
In our situation, the normal bundles involved are easy to compute using the explicit description of the stratification given in \Cref{tab:KempfNess}. They can be written in terms of (tensor products and direct sums of) $\O(1,0)$ and $\O(0,1)$, which denote the obvious line bundles wrt the $T$-action on either of the fixed loci $Z_\lam$. We get

\begin{itemize}
    \item $\eta_{(-1,-1)}= wt_{(-1,-1)}\ \det(\O(-2,-1)\oplus\O(-1,-2)\oplus\O(-1,-1)^{\oplus 3})$ \\
    \hspace*{3.63em}$=wt_{(-1,-1)}\O(-6,-6)=12$

    \item $\eta_{(1,-2)}= wt_{(1,-2)}\ \det(\O(-1,-1)^{\oplus 6 }\oplus\O(-1,-1)^{\oplus 3})=wt_{(1,-2)}\ \O(-9,-9) = 9$
\end{itemize}

For any choice of $(w_1,w_2)\in \Z^2$, \cite[Theorem 2.10]{HL} says that the following sub-category of $D(V/G)$ is a window for the positive side of the VGIT\footnote{Note that our definition actually retrieves the dual of Halpern-Leistner's window. But the cited theorem still holds.}.

\begin{equation*}
\W_+ := \left\{ \F^* \in D(V/G) \ \middle|\ 
\begin{aligned}
    &\text{denoting }\lam_1\equiv(-1,-1) \text{ and }\lam_2\equiv(1,-2) \text{, the following holds: }\\
    &\forall i,\text{any } \lam_i\text{-weight } \kappa_i \text{ appearing in } H_*(Li_{Z_{\lam_i}}^*(\F^*)^\vee) \text{ satisfies} \\
    &w_i\leq \kappa_i < w_i+\eta_{\lam_i}
\end{aligned}
\right\}
\end{equation*}

In particular, if we look for a $G$-equivariant vector bundle $\O_V\otimes E$ in the window, then the torus weight $(a,b)$ of $E^\vee$ must satisfy the following conditions.
\begin{equation}
    w_1 \leq -a-b < w_1+12 \quad \quad w_2 \leq a-2b < w_2+9
\end{equation}

Such weight conditions are called \textit{grade restriction rules}. For example, the grade restriction rules obtained for the choice $(w_1, w_2)=(-10,-6)$, cut out the shaded parallelogram drawn in \Cref{fig:wtWindow} which is contained in the weight lattice of $G$ (identified as a $\Z^2$).
\begin{figure}[htbp]
    \centering
    \begin{tikzpicture}[>=stealth, scale=0.8]
        \fill[gray!15] ({-8/3},{5/3}) -- (0,-1) -- ({22/3},{8/3}) -- ({14/3},{16/3}) -- cycle;

        \draw[thin, gray!40] (-4,-2) grid (9,6);

        \draw[<->, gray!70, very thick] (-4,0) -- (9,0) node[below left, black, scale=1.2] {$a$};
        \draw[<->, gray!70, very thick] (0,-2) -- (0,6) node[below left, black, scale=1.2] {$b$};
        
        \draw[thick, dashed, gray!70] (-2,-2) -- (6,6) node[below right, black, scale=1.1] {};

        \draw[black!60, very thick] ({-8/3},{5/3}) -- (0,-1) -- ({22/3},{8/3}) -- ({14/3},{16/3}) -- cycle;
        
        \tikzset{
            wt/.style={circle, fill=red, inner sep=1.5pt},
            domWt/.style={circle, draw=red, thick, inner sep=3.2pt}
        }
        \node[wt] at (-1,0) {}; \node[wt] at (0,-1) {}; 
        \node[wt] at (0,1) {}; \node[wt] at (1,0) {}; 
        \node[wt] at (0,2) {}; \node[wt] at (2,0) {}; 
        \node[wt] at (1,2) {}; \node[wt] at (2,1) {}; 
        \node[wt] at (1,3) {};\node[wt] at (3,1) {};
        \node[wt] at (2,3) {};\node[wt] at (3,2) {};
        \node[wt] at (2,4) {};\node[wt] at (4,2) {};
        \node[wt] at (3,4) {};\node[wt] at (4,3) {};
        \node[wt] at (4,5) {};\node[wt] at (5,4) {};
        \node[wt] at (0,0) {};\node[wt] at (1,1) {};\node[wt] at (2,2) {};\node[wt] at (3,3) {};\node[wt] at (4,4) {};\node[wt] at (5,5) {};

        \node[domWt] at (0,-1) {}; \node[domWt] at (1,0) {}; \node[domWt] at (2,0) {}; \node[domWt] at (2,1) {}; \node[domWt] at (3,1) {};\node[domWt] at (3,2) {};\node[domWt] at (4,2) {};\node[domWt] at (4,3) {};\node[domWt] at (5,4) {};\node[domWt] at (0,0) {};\node[domWt] at (1,1) {};\node[domWt] at (2,2) {};\node[domWt] at (3,3) {};\node[domWt] at (4,4) {};\node[domWt] at (5,5) {};
    \end{tikzpicture}
    
    \vspace{1.5ex}
    {\small
    \tikz[baseline=-0.6ex]\node[circle, draw=red, thick, inner sep=2.5pt]{}; admissible dominant weights \quad\tikz[baseline=-0.6ex]\node[circle, fill=red, inner sep=1.5pt]{}; weights appearing in all admissible irreps.\par}
    \vspace{1ex}
    \caption{Window of Weights}
    \label{fig:wtWindow}
\end{figure}
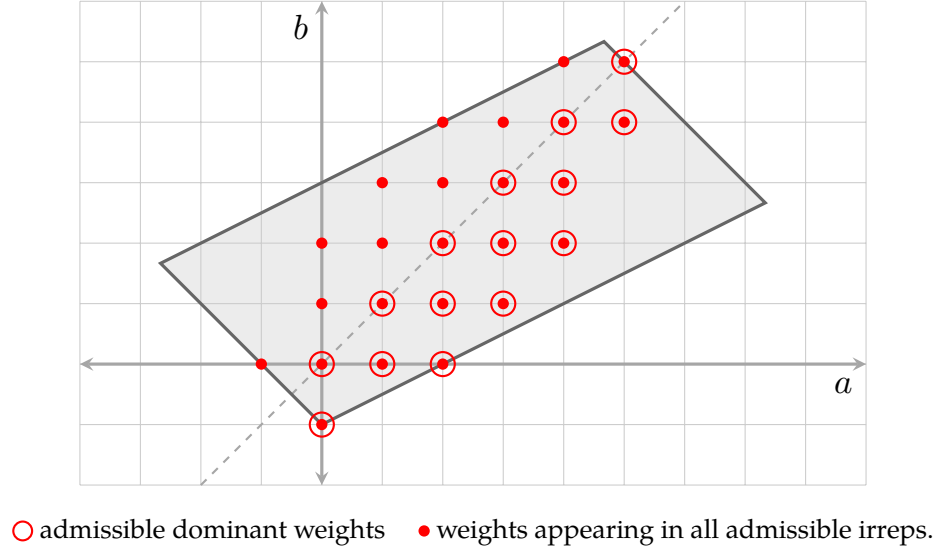

As will be clear from the later steps, it turns out that the full window is too big for our purposes. We will only consider the sub-category generated by vector bundles corresponding to $G$-irreps whose weights satisfy the grade restriction rules. This amounts to looking for dominant weights $w$ in the weight lattice such that all weights on the line segment joining $w$ and its reflection across the $a=b$ line, satisfy the grade restriction rules. Let us call them the \textit{admissible dominant weights} and call the corresponding $G$-irreps to be the \textit{admissible irreps}. In fact, the choice $(w_1, w_2)=(-10,-6)$ made above yields the maximum number of admissible dominant weights. We fix this choice and, using \Cref{fig:wtWindow}, list below all the corresponding admissible $G$-irreps thus obtained.
\begin{equation}\label{eqn:windowIrreps}
\begin{gathered}
  \C ,\  \det(S^\vee),... , \det(S^\vee)^{\otimes5},\\
  S^\vee\otimes \det(S^\vee)^{\otimes-1} , \ S^\vee,... , S^\vee\otimes\det(S^\vee)^{\otimes4},\\
\Sym^2S^\vee ,\  \Sym^2S^\vee\otimes\det(S^\vee),\ \Sym^2S^\vee\otimes\det(S^\vee)^{\otimes2}  
\end{gathered}
\end{equation}
They define the following sub-category of $D(V/G)$, which by construction is also a sub-category of $\W_+$.
\begin{equation}
    \W:=\langle \O_V\otimes E\ |\ E \text{ is an irrep in the above list}\rangle
\end{equation}
As $\W_+$ is a window, its sub-category $\W$ must be a partial window. So, we finally have:
\begin{proposition}\label[proposition]{prop:ffStackToPhase}
    The restriction functor $Li_+^*:D(V/G)\rightarrow D(V_+/G)$ is fully-faithful, when restricted to $\W\subset D(V/G)$. 
\end{proposition}
From now on we will refer to $\W$ as the ``window" and above described generators of $\W$ as ``window bundles". We may pull it back to any stack $\mathfrak{X}$ that has a morphism into $V/G$. If the morphism is obvious from context, we will refer to it's pullback simply as the ``window on $\mathfrak{X}$" and pullback of the generators as "window bundles on $\mathfrak{X}$". Following is an important observation about the window.
\begin{observation} \label[observation]{obs:preTilting}
    The generators of $\W$ have no higher Exts among them. In other words, they are ``pre-tilting". 
\end{observation}
\begin{proof}
    As $G$ is reductive, any finite dimensional $G$-representation is semi-simple. Hence, any short exact sequence of $G$-representations must be split. Then, noting that $V$ is affine, the observation follows.
\end{proof}
So, showing that an exact functor out of $\W$ is fully-faithful, amounts to showing that the images of the window bundles are again pre-tilting and have the same Homs among them as they do in $\W$. 

\begin{remark}\label[remark]{rem:windowShift}
    We could twist our window $\W$ by any tensor power of $\det(S^\vee)$ to get a new window that is just as good for our purposes. This corresponds to shifting our chosen "window of weights" (\Cref{fig:wtWindow}) by the corresponding integer amount along the $(1,1)$ direction.
\end{remark}

\subsection{Step 2: Window Functor for the Negative Side}
The partial window for the positive side constructed in the previous step turns out to be a full window for the negative side. That is, we have
\begin{proposition}\label[proposition]{prop:windowNeg}
    The pullback functor $Li_-^*:D(V/G)\rightarrow D(V_-/G)$ restricts to an equivalence on $\W\subset D(V/G)$.
\end{proposition}
Before we prove this, let us note some observations.
\begin{observation}\label[observation]{obs:windowGensNeg}
We have:
    \begin{itemize}
        \item $Li_-^*(\W)$ is generated by pullbacks of the following vector bundles on $G(2,V_6):$ 
        \begin{equation}
            \begin{array}{ccccccc}
                & \mathcal{O}    & \mathcal{O}(\xi)    & \mathcal{O}(2\xi)   & \mathcal{O}(3\xi)   & \mathcal{O}(4\xi)   & \mathcal{O}(5\xi) \\
\bunNeg(-\xi) & \bunNeg   & \bunNeg(\xi)   & \bunNeg(2\xi)  & \bunNeg(3\xi)  & \bunNeg(4\xi)  &                  \\
                & \Sym^2\bunNeg & \Sym^2\bunNeg(\xi) & \Sym^2\bunNeg(2\xi) & & & 
\end{array}
        \end{equation}
        where $\bunNeg$ is the dual to the universal bundle on $G(2,V_6)$ and $\xi$ is the hyperplane class on $G(2,6)$ induced by its Pl\"ucker embedding. 
\item The above collection of bundles generate $D(G(2,V_6))$. Moreover, it forms a full-exceptional sequence (when read from top to bottom and then left to right)
    \end{itemize}
\end{observation}
\begin{proof}
    Under the view $G(2,V_6)=\Hom(S,V_6)^o/G$, the $G$-equivariant bundle $\O_{\Hom(S,V_6)^o}\otimes S^\vee$ corresponds to $\bunNeg$ and  $\det(\bunNeg)=\O(\xi)$. So the images, under $Li^*_-$, of the generators of $\W$ (as in \eqref{eqn:windowIrreps}) are exactly the pullbacks of the bundles written above. Therefore, the first observation follows. 
    For the second observation, first recall that the Serre functor of $D(G(2,V_6))$ is $(\_)\otimes\O(-6\xi)[8]$. Then write the above collection of bundles as a sequence by reading from top to bottom and then left to right. This sequence may be obtained from Kuznetsov's full exceptional sequence on $G(2,V_6)$ (see \cite[Theorem 3.1]{KuzExcepColl}) by applying this Serre functor (ignoring the shift) on the last object $\bunNeg(5\xi)$ and then moving it to the beginning. Therefore, the second observation follows. 
\end{proof}
\begin{observation}\label[observation]{obs:U}
    $\mathcal U = \bunNeg(-\xi)$, as we have perfect pairing $\U\otimes\U\rightarrow \det(\U)=\O(-\xi)$
\end{observation}

Now we are ready to prove \Cref{prop:windowNeg}, as a consequence of the following two lemmas. 
\begin{lemma}\label[lemma]{lem:ff}
    $Li_-^*:\W\rightarrow D(V_-/G)$ is fully-faithful.
\end{lemma}
\begin{proof}
Notice that the closed complement of $V_-\subset V$ has codimension greater than 1 in $V$. So, by Hartog's lemma, for any pair of window bundles $\E, \F\in \W$, we have:
\begin{equation}
    \Hom_{D(V_-)}(\til i^*_-\E,\til i^*_-\F)= H^0(V_-,\til i^*_-\E^\vee\otimes \til i^*_-\F)\simeq H^0(V,\E^\vee\otimes \F)=\Hom_{D(V)}(\E,\F)
\end{equation}
Taking $G$-invariants on both sides, and identifying the invariant subspaces with Homs on the quotient spaces, we get induced isomorphism:
\begin{equation}
    \Hom_{D(V_-/G)}(i^*_-\E,i^*_-\F) \simeq \Hom_{D(V/G)}(\E,\F)
\end{equation}
So, what remains to be shown, for the fully-faithfulness claim, is that the Exts are also preserved. But as the window bundles are pre-tilting, this amounts to showing that the Ext spaces between any two window bundles pulled back to $V_-/G$, vanish.

Let us denote, for readability, $\GG\eq G(2,V_6)$. By \Cref{obs:windowGensNeg}, we can write $\til i^*_-\E^\vee\otimes \til i^*_-\F$ as the pull-back of some bundle $\V$ on $\GG$. Then, the Ext spaces can be re-written as follows
\begin{equation}\label{eq:ff_proven_in_appendix}
\begin{aligned}
\Ext^d_{D(V_-/G)}(\til i^*_-\E,\til i^*_-\F) &= H^d(V_-/G,\til i^*_-\E^\vee\otimes \til i^*_-\F)\\
&=H^d(\GG,r_*(\til i^*_-\E^\vee\otimes \til i^*_-\F))\\
&=H^d(\GG,\V\otimes r_*\O)\\
&=\bigoplus_{i\geq0}H^d(\GG,\V\otimes \Sym^i(\bunNeg(\xi)\oplus\O(\xi)^{\oplus 3}))\\
&=\bigoplus_{i\geq0}\bigoplus_{j=0}^{i}H^d(\GG,\V\otimes \Sym^j\bunNeg(i\xi))^{\oplus\binom{i-j+2}{2}}
\end{aligned}
\end{equation}
So, we need to show vanishings of all the following cohomologies for each bundle $\V$ built from the list given in \Cref{obs:windowGensNeg}. 
\begin{equation}
    \{H^d(\GG,\V\otimes \Sym^j\bunNeg(i\xi))\ |\ i\geq j\geq 0,\ i\geq 0,\ d>0\}
\end{equation}
Using \Cref{obs:U}, we can write any such $\V$ in terms of $\bunNeg$ only. Then, each of these bundles is one of the following.
\begin{align*}
    \O(k\xi) \quad& \text{ for } k\in  -5,..,5\\
    \bunNeg(k\xi) \quad& \text{ for } k\in -6,..,5\\
    \Sym^2\bunNeg(k\xi) \quad& \text{ for } k\in -5,..,3\\
    \bunNeg\otimes\bunNeg(k\xi) \quad& \text{ for } k\in -6,..,5\\
    \bunNeg\otimes \Sym^2\bunNeg(k\xi) \quad& \text{ for } k\in -5,..,2\\
    \Sym^2\bunNeg\otimes \Sym^2\bunNeg(k\xi) \quad& \text{ for } k\in -4,..,0
\end{align*}
For each $\V$ as above, the desired cohomology vanishings follow by decomposing $\V\otimes \Sym^j\bunNeg(i\xi)$ into multiples of Schur powers of $\bunNeg$ using Littlewood-Richardson-Rule and then showing that all higher cohomologies of all summands vanish using Borel-Weil-Bott. The computation is straightforward, but is cumbersome to do by hand. So, we handle this using a Python script, included in \Cref{app:tiltingCode}.
\end{proof}

\begin{lemma}\label[lemma]{lem:surjectivityOfWindowFunctor}
    $Li_-^*:\W\rightarrow D(V_-/G)$ is essentially surjective.
\end{lemma}
\begin{proof}
As already noted in the beginning of \Cref{sec:constructionOfFunctor} , we know that $Li^*_-:D(V/G)\rightarrow D(V_-/G)$ is essentially surjective. So, any object $\F^*\in D(V_-/G)$ extends to an object $\til\F^*\in D(V/G)$. As, $D(V/G)$ is generated by equivariant bundles associated to $G$-irreps, there exists a resolution of $\til\F^*$ using such bundles. That is, it is quasi-isomorphic to a complex made up of direct sums of such bundles. So, $\F^*$ is resolved by the pullback these bundles. In other words, $D(V_-/G)$ is generated by $G$-equivariant bundles of the form $\O_{V_-}\otimes E$. Now, any such bundle is the pullback of the corresponding $G$-equivariant bundle $\O_{\Hom(S,V_6)^o}\otimes E$ on $G(2,V_6)=\Hom(S,V_6)^o/G$. But by \Cref{obs:windowGensNeg} we know that $D(G(2,V_6))$ is generated by the window bundles (pulled back to it). So, $\O_{\Hom(S,V_6)^o}\otimes E$ is resolved by window bundles and hence so is its pullback $\O_{V_-}\otimes E$. Therefore, we have shown that any object in $D(V_-/G)$ maybe resolved using window bundles, concluding the proof.
\end{proof}

\subsection{Step 3: The Stacky Blow-up}
Here we aim to show that $f:W/\Gamma \xrightarrow{} V_+/G$ induces a fully-faithful functor $Lf^*:D(V_+/G)\rightarrow D(W/\Gam)$ of derived categories. Before the actual proof, let us note an interesting observation about this map.
\begin{observation}
$f$ is a $\P^1$-bundle over the locus $V(\ul v)$. And away from this locus, it is an isomorphism.
\end{observation}
\begin{proof}
    Consider arbitrary point $(B, 0, \ul w)\in V(\ul v)\subset V_+$. By the semi-stability condition on $V_+$, we have $\ul w\neq0$.
    We have fiber product diagram:
\begin{equation}
    \begin{tikzcd}
E \arrow[r] \arrow[d] \arrow[dr, phantom, "\lrcorner", very near start] & \{(B, 0, \ul w)\} \arrow[d] \\
G \times_{\Gam} W \arrow[r] \arrow[d] \arrow[dr, phantom, "\lrcorner", very near start] & V_+ \arrow[d] \\
W/\Gam \arrow[r] & V_+/G
\end{tikzcd}
\end{equation}
So, the fiber is: 

\begin{equation}    \begin{aligned}
        E &= \{(g,\ (B',v', \ul w'))\ | \ g\in G, \ g.(B',(v',0),\ul w')=(B,(0,0),\ul w)\}/\Gam\\
        &= \{(g,\ (Bg,\ 0,\  \det(g)^{-1}\ul w))\ | \ g\in G\}/\Gam\\
        &\simeq G/\Gam\simeq\P^1
    \end{aligned}
\end{equation}
Now consider an arbitrary point $(B, \ul v, \ul w)$ outside the locus $V(\ul v)$. Wlog, we can assume $\ul v = (v,0)$ for some $v\neq 0$ (this can be achieved by the $G$-action on $V_+$). We can compute the fiber over this point using a similar diagram. It is:
\begin{equation}
\begin{aligned}
        F &= \{(g,\ (B',v', \ul w'))\ | \ g\in G, \ g.(B',(v',0),\ul w')=(B,(v,0),\ul w)\}/\Gam\\
        &= \{(g,\ (Bg,\ v,\  \det(g)^{-1}\ul w))\ | \ g\in \Gam\}/\Gam\\
        &\simeq \Gam/\Gam\simeq \{1\}
\end{aligned}
\end{equation}
Hence we conclude our observation.
\end{proof}
This tells us that $f$ is certainly not an isomorphism and moreover suggests that it might be a ``stacky blow-up" of $V_+/G$ along $V(\ul v)/G$. In fact, there already exists a notion of blowing up stacks that naturally generalizes blowing up schemes. Given an Artin stack $\X$ and an ideal sheaf $\I $ on it, the ``blow-up of $\X$ along $\I$" is defined to be the following quotient stack.
\begin{equation}
    \Bl_{\I}\X:=\left(Spec_{\X}(\bigoplus_{d \ge 0} \I^d) \sp\setminus\sp V(\bigoplus_{d>0}\I)\right)/\C^* 
\end{equation}
In what follows, we will essentially show that our map $f$ can indeed be realized as a stacky blow-up in the above sense. Using a standard identity of quotient stacks, we can decompose $f$ as follows by realizing $W/\Gam$ as a quotient by $G$.

\begin{equation}
\begin{array}{ccccc}
    W/\Gam & \xrightarrow{\sim} & ( G\times_{\Gam}W)/G & \xrightarrow{\til f} & V_+/G \\
    p & \longmapsto & (1, p) & & \\
      &                           & (g, p) & \longmapsto & gp
\end{array}
\end{equation}
Here, the group actions involved in the middle stack are as follows:
\begin{itemize}
    \item $\Gam \curvearrowright G\times W: \gamma.(g,p):=(g\gamma^{-1}, \gamma .p)$
    \item $G\curvearrowright G\times_{\Gam}W:g'.(g,p):=(g'g,p)$
\end{itemize}
Note that the action $\Gam \curvearrowright G\times W$ is free (as $\Gam\curvearrowright G$ is free) and hence stable. So, $G\times_\Gam W$ is a scheme.

So, let us replace $f$ by $\til f$ and analyze that instead. As described above, $\til f$ is induced by a $G$-equivariant map between schemes, which we again denote by $\til f$. We have
\begin{lemma}\label[lemma]{lem:stackyBlowUp}
    $\til f: G\times_{\Gam}W\rightarrow V_+$ is isomorphic to the blow-up $\Bl_{V(\ul v)}V_+\rightarrow V_+$.
\end{lemma}
\begin{proof}
    Explicitly, $\Bl_{V(\ul v)}V_+ = \{([p,r],\coordsOfV)\in\P^1\times V_+ \st pv_2=rv_1\}$ and $\Bl_{V(\ul v)}V_+ \rightarrow V_+$ is the projection onto the coordinates $\coordsOfV$. Using this description, we can decompose $\til f$ as follows:
    \begin{equation*}
        \begin{array}{cccc}
         G\times_{\Gam}W & \xrightarrow{\mu} & \Bl_{V(\ul v)}V_+  & \longrightarrow V_+\\
         \left (g, \coordsOfU\right ) &\longmapsto &([g_{11},g_{21}],\sp  g.(B,(v,0), \ul w)) \\
         & & = ([g_{11},g_{21}], \sp (Bg^{-1}, \sp v\sp \det(g)\sp (g_{11},g_{21}), \sp \det(g)\sp\ul w) )
    \end{array}
    \end{equation*}
We claim that the first map in this decomposition is an isomorphism, with its inverse given as follows.

\begin{equation*}
    \begin{array}{ccc}
     \Bl_{V(\ul v)}V_+& \xrightarrow{\nu} & G\times_\Gam W  \\
     ([1,r], \sp(B,\sp(\lam ,\lam r),\sp\ul w)) & \longmapsto & \left ( \matTwo{1}{0}{r}{1}, \sp \left(B.\matTwo{1}{0}{r}{1}, \sp\lam, \sp\ul w\right) \right )\\
     ([p,1], \sp(B,\sp(\lam p,\lam ),\sp\ul w)) & \longmapsto & \left ( \matTwo{p}{-1}{1}{0}, \sp \left(B.\matTwo{p}{-1}{1}{0}, \sp\lam, \sp\ul w\right) \right )
\end{array}
\end{equation*}
 We have described the mapping out of the two affine charts $p\neq 0$ and $r\neq 0$, separately. Taking the $\Gam$-symmetry of the target into account, it is easy to verify that these descriptions match on the overlap of these charts and hence $\nu$ a valid morphism.
 
Let us verify the above claim. It is easy to check that $\mu \circ \nu = Id$. On the other hand, action of $\nu\circ\mu$ on a general element looks like
\begin{equation}
\begin{gathered}
    (\nu\circ \mu)\left(g, \coordsOfU\right)\\
    =\nu ([1,g_{11}^{-1}g_{21}], \sp (Bg^{-1}, \sp g_{11}\sp v\sp \det(g)\sp (1,g_{11}^{-1}g_{21}), \sp \det(g)\sp\ul w) )\\
    =\left ( \matTwo{1}{0}{g_{11}^{-1}g_{21}}{1}, \sp \left(B.g^{-1}.\matTwo{1}{0}{g_{11}^{-1}g_{21}}{1}, \sp g_{11}\sp v\sp \det(g), \sp \det(g)\ul w\right) \right )
\end{gathered}
\end{equation}
Here, we have assumed that $g_{11}\neq 0$, so that $\nu$ sends the given element to the affine chart $p\neq 0$. One can argue analogously for the other chart. Now consider the following matrix, which turns out to be an element of $\Gam$.
\begin{equation}
    \gamma \equiv g^{-1}.\matTwo{1}{0}{g_{11}^{-1}g_{21}}{1}\in\Gam
\end{equation}

It is easy to check that $\gamma^{-1}.(g,\coordsOfU)=(\nu\circ\mu)(g,\coordsOfU)$. Therefore, taking the $\Gam$-symmetry into account, we have $\nu \circ \mu = Id$.

\end{proof}

\begin{proposition}\label[proposition]{prop:tiltingOnW}
    $Lf^*:D(V_+/G)\rightarrow D(W/\Gam)$ is fully-faithful.
\end{proposition}
\begin{proof}
    In the following, denote $\til V_+\simeq \Bl_{V(\ul v)}V_+$. By previous discussion, $f:W/\Gamma \xrightarrow{} V_+/G$ is isomorphic to the morphism $\varphi:\til V_+/G\rightarrow V_+/G$, induced by the $G$-equivariant blow-up map $\til \varphi:\til V_+\rightarrow V_+$. Here, $G\curvearrowright \Bl_{V(\ul v)}V_+$ is the one induced from $G\curvearrowright G\times_{\Gam}W$, under the isomorphism described in \Cref{lem:stackyBlowUp}. 
    Let $\E, \F \in D_G(V_+)\equiv D(V_+/G)$. By fully-faithfulness of $L\til \varphi^*$ and $G$-equivariance of $\widetilde \varphi$, we have canonical isomorphism 
    \begin{equation}
        \Hom_{D(\til V_+)}(L\til \varphi^*\E,L\til \varphi^*\F)\simeq \Hom_{D(V_+)}(\E,\F)
    \end{equation}
    of $G$-modules. As $G$ is reductive, $G$-invariants of these Hom spaces are exactly equal to the Hom spaces of the corresponding objects in $D(V_+/G)$. So, taking $G$-invariants on both sides, we have induced isomorphism:
    \begin{equation}
        \Hom_{D(\til V_+/G)}(L\varphi^*\E,L \varphi^*\F)\simeq \Hom_{D(V_+/G)}(\E,\F)
    \end{equation}
    Therefore, $L\varphi^*$ is fully-faithful.
\end{proof}

\subsection{Step 4: The Open Restriction}\label{sec:openRestriction}
Here we finally arrive at $\kb$, by going through the open inclusion $i:\kb\rightarrow W/\Gam$. We will show that the induced restriction functor $Li^*:D(W/\Gam)\rightarrow D(\kb)$ is fully-faithful when restricted to the window, pulled back to $W/\Gam$. Let us begin with a crucial observation.

\begin{observation}
    The closed complement $W\sp\setminus\sp U$ is given by vanishing of coordinates $\ul b$ and hence has codimension 6 in $W$. 
\end{observation}
By Hartog's lemma, for any pair of bundles $\E, \F$ in the window (pulled back to $W/\Gam$), we have:
\begin{equation}
    \Hom_{D(U)}(\til i^*\E,\til i^*\F)\eq H^0(U,\til i^*\E^\vee\otimes \til i^*\F)\simeq H^0(W,\til i^*\E^\vee\otimes \til i^*\F)=\Hom_{D(W)}(\E,\F)
\end{equation}
Taking $\Gam$-invariants on both sides, and identifying the invariant subspaces with Homs on the quotient spaces, we get induced isomorphism:
\begin{equation}
    \Hom_{D(U/\Gam)}(i^*\E,i^*\F) \simeq \Hom_{D(W/\Gam)}(\E,\F)
\end{equation}
So, what remains to be shown, for the fully-faithfulness claim, is that the Exts are also preserved. But as the window bundles are pre-tilting, this amounts to showing that the Ext spaces between any two bundles $\E, \F$ in the window (pulled back to $\kb$) vanish.

Recall that $\Srep$ denotes the tautological $G$-representation. We will use the same notation to denote its restriction to the subgroup $\Gam$, giving us a $\Gam$-representation. Denote by $\bunPos$, the corresponding bundle on $U/\Gam$. We can realize it as an extension between familiar line bundles, and using this description we can go on to describe $\Sym^2\S^\vee$ and $\det(\S^\vee)$ more explicitly.
\begin{observation}\label[observation]{obs:ses}
    On $\kb$, we have:
    \begin{itemize}
        \item short exact sequence $0\rightarrow \O(-h-H)\rightarrow \S^\vee\rightarrow\O(h)\rightarrow 0$
        \item short exact sequence $0\rightarrow \S^\vee(-h-H)\rightarrow \Sym^2\S^\vee\rightarrow\O(2h)\rightarrow 0$
        \item $\det(\S^\vee)=\O(-H)$
    \end{itemize}
\end{observation}
\begin{proof}
Working in a basis of $\Srep^\vee$, this $\Gam$-representation is described as follows:
\begin{equation}
    \matTwo{\alpha}{\beta}{0}{\delta}.\colTwo{x}{y}=\colTwo{x}{y}\matTwo{\alpha}{\beta}{0}{\delta}^{-1}=\colTwo{\alpha^{-1}x}{\delta^{-1}y-\beta \alpha^{-1}\delta^{-1}x}
\end{equation}
    So, we have the following short exact sequence of $\Gam$-representations:
    \begin{equation}
        \begin{array}{ccccc}
         0\rightarrow \C_{\delta ^{-1}} &\longrightarrow &\Srep^\vee &\longrightarrow &\C_{\alpha^{-1}}\rightarrow 0  \\
         &\begin{bmatrix}0\\1\end{bmatrix}&&\begin{bmatrix}1&0\end{bmatrix}& 
    \end{array}
    \end{equation}

where $\C_{\alpha ^{-1}}$ and $\C_{\delta^{-1}}$ are the $\Gam$-representations corresponding to the $\Gam$-characters $\alpha^{-1}$ and $\delta^{-1}$ respectively. So, $\C_{\alpha ^{-1}}$ corresponds to the line bundle $\O(h)$ and $\C_{\delta^{-1}}=\C_{(\alpha \delta)^{-1}}\otimes\C_{\alpha}$ corresponds to the line bundle $\O(-h-H)$. So, the first short exact sequence follows. The other two statements are immediate corollaries of this sequence. 
\end{proof}
     So, the bundles in the window (pulled back to $\kb$) are as follows :
\begin{equation}
    \begin{array}{ccccccc}
                & \mathcal{O}    & \mathcal{O}(-H)    & \mathcal{O}(-2H)   & \mathcal{O}(-3H)   & \mathcal{O}(-4H)   & \mathcal{O}(-5H) \\
\bunPos(H) & \bunPos   & \bunPos(-H)   & \bunPos(-2H)  & \bunPos(-3H)  & \bunPos(-4H)  &                  \\
                & \Sym^2\bunPos & \Sym^2\bunPos(-H) & \Sym^2\bunPos(-2H) & & & 
\end{array}
\end{equation}
For any two bundles $\E,\F$ in the window, we have $\Ext^i(\E,\F)=H^i(U/\Gam,\E^\vee\otimes \F)$. So, we need to show vanishings of all these cohomologies (for $i>0$). We will note some more easy observations to help us with the cohomology computations to come.

\begin{observation}\label[observation]{obs:S} We have
\begin{itemize}
\item $\S = \S^\vee(H)$, as there exists a perfect pairing $\S \otimes \S \to \det \S = \mathcal{O}(H)$.
\item $\wt \Q  = \Lambda^4 \Q^\vee(h)$, as there exists a perfect pairing $\wt \Q  \otimes\Lambda^4 \wt \Q \rightarrow \det(\wt \Q )=\O(h)$
\end{itemize}
\end{observation}
\begin{observation}
$\omega_{\blp} = \mathcal{O}(-5h - 4H)$
\end{observation}
\begin{proof}
We can compute $\omega_{\blp}$ in two different ways.
\begin{itemize}
  \item Using $\blp = \Bl_{\P^2}\P ^8$, we have:
  \begin{equation}
      \omega_{\blp} = \pi'^*\omega_{\P ^8} \otimes \mathcal{O}(5E)
    = \mathcal{O}(-9H) \otimes \mathcal{O}(5H - 5h)
    = \mathcal{O}(-5h - 4H)
  \end{equation}
  \item Using $\blp = \P \bigl(\mathcal{O}_{\P^5}(-h) \oplus 3\mathcal{O}_{\P^5}\bigr)$, we have:
\begin{equation}
\begin{aligned}
    \omega_{\blp}
      &= \pi^*\omega_{\P^5} \otimes \pi^*\det\!\bigl(\mathcal{O}(h) \oplus 3\mathcal{O}\bigr) \otimes \mathcal{O}(-4H) \\
      &= \mathcal{O}(-6h) \otimes \mathcal{O}(h) \otimes \mathcal{O}(-4H) \\
      &= \mathcal{O}(-5h - 4H).
  \end{aligned}
\end{equation}
  
\end{itemize}
\end{proof}

\begin{lemma} \label[lemma]{lem:cohomExpression}
For $\F = \mathcal{O}_{U/\Gam}(ah + bH)$, $d \geq 2$, $a \geq -5$, we have:
\[
  H^d(\kb, \F)
  = \begin{cases}
      \displaystyle\bigoplus_{i=0}^{-b-4}
        H^{8-d}\!\Bigl(\P^5,\; \Sym^i \Lambda^4 \Q^\vee(-a-5)
          \otimes \Sym^{-i-b-4}\!\bigl(\mathcal{O}(1) \oplus \mathcal{O}^{\oplus3}\bigr)\Bigr)^\vee
      & \text{if } b \leq -4, \\[6pt]
      0 & \text{otherwise.}
    \end{cases}
\]
\end{lemma}
\begin{proof}
\begin{equation}
\begin{aligned}
  H^d(\kb, \F)
  &= H^d\bigl(\blp,\, p_*\F\bigr)\\
  &= H^d\bigl(\blp,\, \O(ah+bH)\otimes Rp_*\O\bigr)\\
  &= \bigoplus_{i \geq 0} H^d\!\Bigl(\blp,\, \Sym^i\!\bigl(\pi^*\Q^\vee(h+H)\bigr)(ah+bH)\Bigr) \\
  &= \bigoplus_{i \geq 0} H^d\!\Bigl(\blp,\, \pi^*\Sym^i\Q^\vee\bigl((i+a)h + (i+b)H\bigr)\Bigr) \\
  &= \bigoplus_{i \geq 0} H^d\!\Bigl(\P^5,\, \Sym^i \Q^\vee((i+a)h) \otimes R\pi_*\mathcal{O}\bigl((i+b)H\bigr)\Bigr).
\end{aligned}
\end{equation}
As $\O(H)$ is the relative $\O(1)$ of the $\P^3$-bundle map $\pi$ , we have 
\begin{equation}
    R\pi_*\mathcal{O}\bigl((i+b)H\bigr) = 0 \sp\text{ for } -4 < i+b < 0\sp\text{ i.e. } -4-b < i < -b
\end{equation}
So, above sum splits as $\Sigma_+ \oplus \Sigma_-$, where $\Sigma_+\eq$ the sum over $i \geq \max\{0, -b\}$ and $\Sigma_-\eq$ the sum over $0 \leq i \leq -b-4$.\\
\noindent \textit{Computing $\Sigma_+$} (using Borel-Weil-Bott, see Appendix \ref{app:cohomCode} for more explanation):
\begin{equation}
\begin{aligned}
  \Sigma_+
  &= \bigoplus_{i \geq \max\{0,-b\}}
      H^d\!\Bigl(\P^5,\, \Sym^i \Q^\vee((i+a)h) \otimes \Sym^{i+b}\!\bigl(\mathcal{O}(h) \oplus 3\mathcal{O}\bigr)\Bigr) \\
  &= \bigoplus_{i \geq \max\{0,-b\}} \bigoplus_{j=0}^{i+b}
      H^d\!\Bigl(\P^5,\, \Sym^i \Q^\vee((i+a)h) \otimes \mathcal{O}(jh)
        \otimes \mathcal{O}^{\oplus\binom{i+b-j+2}{2}}\Bigr) \\
  &= \bigoplus_{i \geq \max\{0,-b\}} \bigoplus_{j=0}^{i+b}
      H^d\!\Bigl(\P^5,\, \Sym^i \Q^\vee(i+j+a)\Bigr)^{\oplus\binom{i+b-j+2}{2}} \\
  &= 0 \quad \forall\; a \geq -5,\; b \in \mathbb{Z},\; d \geq 2.
\end{aligned}
\end{equation}

\noindent \textit{Computing $\Sigma_-$} (using Serre duality on $\blp$):
\begin{equation}
\begin{aligned}
  \Sigma_-
  &= \bigoplus_{i=0}^{-b-4}
      H^{8-d}\!\Bigl(\blp,\, \widetilde{\pi}^*\Sym^i \Q
        \bigl(-(i+a)h - (i+b)H\bigr) \otimes \mathcal{O}(-5h-4H)\Bigr)^\vee \\
  &= \bigoplus_{i=0}^{-b-4}
      H^{8-d}\!\Bigl(\blp,\, \widetilde{\pi}^*\Sym^i \Q
        \bigl(-(i+a+5)h - (i+b+4)H\bigr)\Bigr)^\vee \\
  &= \bigoplus_{i=0}^{-b-4}
      H^{8-d}\!\Bigl(\P^5,\, \Sym^i \Q(-i-a-5)
        \otimes \Sym^{-i-b-4}\!\bigl(\mathcal{O}(1) \oplus \mathcal{O}^{\oplus3}\bigr)\Bigr)^\vee \\
  &= \bigoplus_{i=0}^{-b-4}
      H^{8-d}\!\Bigl(\P^5,\, \Sym^i \Lambda^4 \Q^\vee(-a-5)
        \otimes \Sym^{-i-b-4}\!\bigl(\mathcal{O}(1) \oplus \mathcal{O}^{\oplus3}\bigr)\Bigr)^\vee.
\end{aligned}
\end{equation}
Therefore, the lemma follows.
\end{proof}

\begin{corollary}\label[corollary]{cor:cohomVanishings}
We have following cohomology vanishings for some line bundles on $\kb$.

\medskip
$H^d(\kb,\, \mathcal{O}(bH)) = 0$ for $b \in \{-5,..,5\}$ and $d \neq 0$.

\medskip
$H^d(\kb,\, \mathcal{O}(h+bH)) = 0$ for $(b \in \{-5,..,6\}$ and $d \neq 0, 3)$ or $(b \in \{-3,..,6\}$ and $d \neq 0)$.

\medskip
$H^d(\kb,\, \mathcal{O}(-h+bH)) = 0$ for $(b \in \{-6,..,5\}$ and $d \neq 0, 4)$ or $(b \in \{-4,..,5\}$ and $d \neq 0)$.

\medskip
$H^d(\kb,\, \mathcal{O}(2h+bH)) = 0$ for $(b \in \{-4,..,6\}$ and $d \neq 0, 1, 3)$ or $(b \in \{-3,..,6\}$ and $d \neq 0, 1)$.

\medskip
$H^d(\kb,\, \mathcal{O}(-2h+bH)) = 0$ for $(b \in \{-6,..,4\}$ and $d \neq 0, 1, 4)$ or $(b \in \{-5,..,4\}$ and $d \neq 0, 1)$.

\medskip
$H^d(\kb,\, \mathcal{O}(3h+bH)) = 0$ for $b \in \{-2,..,5\}$ and $d \neq 0$.

\medskip
$H^d(\kb,\, \mathcal{O}(-3h+bH)) = 0$ for $b \in \{-5,..,2\}$ and $d \neq 0, 1$.

\medskip
$H^d(\kb,\, \mathcal{O}(4h+bH)) = 0$ for $b \in \{0,..,4\}$ and $d \neq 0$.

\medskip
$H^d(\kb,\, \mathcal{O}(-4h+bH)) = 0$ for $b \in \{-4,..,0\}$ and $d \neq 0, 1$.
\end{corollary}
\begin{proof}
    This is a straightforward computation using \Cref{lem:cohomExpression} and Borel-Weil-Bott. It has been carried out in \Cref{app:cohomCode}. {In particular, the code produces the following table of degrees carrying non-zero cohomology, for $-4\le a\le4$ and $-6\le b\le6$.
    \begin{center}\footnotesize
    \begin{tabular}{r|ccccccccccccc}
    $a\backslash b$ & $-6$ & $-5$ & $-4$ & $-3$ & $-2$ & $-1$ & $0$ & $1$ & $2$ & $3$ & $4$ & $5$ & $6$\\\hline
    $-4$ & 0,1,8 & 0,1,8 & 0,1 & 0,1 & 0,1 & 0,1 & 0,1 & 0,1 & 0,1 & 0,1 & 0,1 & 0,1 & 0,1\\
    $-3$ & 0,1,8 & 0,1 & 0,1 & 0,1 & 0,1 & 0,1 & 0,1 & 0,1 & 0,1 & 0,1 & 0,1 & 0,1 & 0,1\\
    $-2$ & 0,1,4 & 0,1 & 0,1 & 0,1 & 0,1 & 0,1 & 0,1 & 0,1 & 0,1 & 0,1 & 0,1 & 0,1 & 0,1\\
    $-1$ & 0,4 & 0,4 & 0 & 0 & 0 & 0 & 0 & 0 & 0 & 0 & 0 & 0 & 0\\
    $0$ & 0,4 & 0 & 0 & 0 & 0 & 0 & 0 & 0 & 0 & 0 & 0 & 0 & 0\\
    $1$ & 0,3 & 0,3 & 0,3 & 0 & 0 & 0 & 0 & 0 & 0 & 0 & 0 & 0 & 0\\
    $2$ & 0,3 & 0,3 & 0,3 & 0 & 0 & 0 & 0 & 0 & 0 & 0 & 0 & 0 & 0\\
    $3$ & 0,3 & 0,3 & 0,3 & 0 & 0 & 0 & 0 & 0 & 0 & 0 & 0 & 0 & 0\\
    $4$ & 0,3 & 0,3 & 0,3 & 0 & 0 & 0 & 0 & 0 & 0 & 0 & 0 & 0 & 0
    \end{tabular}
    \end{center}
    Each of the nine items above is referred to in this table. So, the corollary holds as stated.}
\end{proof}

\begin{corollary}
    Let $\F$ be a bundle of the form $\E_1^\vee\otimes\E_2$, for some bundles $\E_i$ in the window. Then, $H^i(\kb, \F) = 0 \quad \forall \sp i \neq 0,1,3,4$
\end{corollary}
\begin{proof}
    Using \Cref{obs:S}, we can write all our bundles $\F$ in consideration, in terms of $\S^\vee$ only. Then, each of these bundles is one of the following:
\begin{align*}
    \O(kH) \quad& \text{ for } k\in  -5,..,5\\
    \bunPos(kH) \quad& \text{ for } k\in -5,..,6\\
    \Sym^2\bunPos(kH) \quad& \text{ for } k\in -3,..,5\\
    \bunPos\otimes\bunPos(kH) \quad& \text{ for } k\in -4,..,6\\
    \bunPos\otimes \Sym^2\bunPos(kH) \quad& \text{ for } k\in -2,..,5\\
    \Sym^2\bunPos\otimes \Sym^2\bunPos(kH) \quad& \text{ for } k\in -2,..,2
\end{align*}
    Using the short exact sequences noted in \Cref{obs:ses}, we get induced short exact sequences that describe each bundle appearing in above series as an extension between bundles appearing before it, twisted by line bundles of the form $\O(ah+bH)$, where $a\in -4,..,4$. Using the induced long exact sequences and the cohomology vanishings proved in \Cref{cor:cohomVanishings} we can show vanishing of $H^i(\kb, \F)$ for each bundle $\F$ in the above series and for each $i\geq0$, except for $i=0,1,3,4$. 
\end{proof}
From the cohomology vanishings proved so far, we can conclude vanishings of all higher Exts between any pair of bundles in our window, except for $\Ext^1, \Ext^3$ and $\Ext^4$ of some pairs of bundles. To show their vanishings we will use local cohomology theory.
\begin{lemma}
    Let $\F$ be a bundle of the form $\E_1^\vee\otimes\E_2$, for some bundles $\E_i$ in the window (pulled back to $\kb$). Then, $H^i(\kb, \F) = 0 \quad \forall \sp i \in 1,..,4$
\end{lemma}
\begin{proof}
       We have an open sub-stack inclusion $\kb \xrightarrow{} W/\Gamma$ induced by the $\Gamma$-equivariant open inclusion $U\xrightarrow{} W$. Denote by $Z$, the closed complement $W\ \setminus \ U$. It is given by the vanishing of $\ul b$. As each coordinate $b_i$ is a $\Gam$-semi-invariant function on $W$, we have that $Z$ is a local complete intersection of codimension 6 in $W$. So, by theorem 3.8 in \cite{HartshorneLocalCohom} and the spectral sequence $E_2^{p,q}\eq H^p(W, \mathcal{H}^q_Z(E))\Rightarrow H^{p+q}_Z(E)$, we have $H^j_Z(E)=0$ for any $j<6$ and any locally free sheaf $E$ on $W$. 
    
    Now Let us focus on our given sheaf $\F=\E_1^\vee\otimes\E_2$. Seen as a $\Gam$-equivariant sheaf on $W$, it corresponds to $F=\O_{W} \otimes E_1^\vee\otimes E_2$, where $E_i$ is the $\Gam$-representation corresponding to $\E_i$. Let us write down the long exact sequence in ``cohomologies with supports" (see \cite[Proposition 1.9]{HartshorneLocalCohom}) for the sheaf $F$. 
\begin{equation}
\begin{aligned}
0 &\to H_Z^0(F) \to H^0( W, F) \to H^0( U, F) \\
  &\to H_Z^1(F) \to H^1( W, F) \to H^1( U, F) \\
  &\to H_Z^2(F) \to \dots
\end{aligned}
\end{equation}
Using this sequence and the local cohomology vanishings noted above, we get isomorphisms
\begin{equation}
    H^i(U,F)\simeq H^i(W,F) \quad \forall i\leq4
\end{equation}

Note that the cohomologies of $\F$ on $U/\Gam$ are, by definition, given as $H^i(U/\Gam,\F)=H^i_\Gam(U,F)$. Here, the spaces on the right hand side are the "$\Gam$-equivariant cohomology spaces", which are defined as the indexed right derived functors of the composition of "taking global sections" functor followed by "taking $\Gam$-invariants" functor. Consequently, we have the following spectral sequence, that computes these equivariant cohomologies in terms of group cohomologies of $\Gam$ with coefficients in the usual sheaf cohomology spaces of $F$ on $U$.
\begin{equation}
    E_2^{p,q}\eq H^p(\Gam, H^q(U,F))\Rightarrow H^{p+q}_\Gam(U,F)
\end{equation}

An analogous spectral sequence exists for the $\Gam$-equivariant cohomologies of $F$ on $W$. Using these spectral sequences and the isomorphisms derived above, we conclude

\begin{equation}
    H^i(U/\Gam,\F)\simeq H^i(W/\Gam,\F) \quad \forall i\leq4
\end{equation}
Now by \Cref{prop:ffStackToPhase}, \Cref{obs:preTilting} and \Cref{prop:tiltingOnW} we have $H^i(W/\Gam,\F)=0 \ \forall i>0$. By these vanishings and above isomorphisms, we conclude our desired cohomology vanishings on $U/\Gam$.
    
\end{proof}
So, we conclude vanishing of all higher Exts between any pair of bundles in our window, hence proving our fully-faithfulness claim.
\begin{proposition}
$Li^*:D(W/\Gam)\rightarrow D(\kb)$ is fully-faithful when restricted to the window, pulled back to $W/\Gam$.
\end{proposition}

\subsection{Step 5: Turning on the Superpotential}
The results of the previous steps, amount to the construction of a fully-faithful functor $D(V_-/G)\rightarrow D(\kb)$. It is given by the inverting the first arrow in the following and then composing it with the rest. 
\begin{equation}
    D(V_-/G)\xleftarrow[Li_-^{*}]{\sim}\W \xrightarrow[Li_+^*]{} D(V_+/G)\rightarrow D(W/\Gam)\rightarrow D(U/\Gam)
\end{equation}
Now we will ``turn on the superpotential" for this functor. That is, we will lift it to a fully-faithful functor between the corresponding categories of B-branes. Let us begin by turning on the superpotential for the window subcategory $\W\subset D(V/G)$. It turns out that for our situation, the correct notion for this is simply the full subcategory of $D(V/G,\chk s)$ generated by B-branes whose underlying vector bundle is a direct sum of bundles coming from the window. Denote this subcategory by $(\W,\chk s)\subset D(V/G,\chk s)$. Then, we have induced maps:
\begin{equation}\label{eqn:functorsWithSuperpot}
    D(V_-/G,\potneg)\xleftarrow[Li_-^{*}]{}(\W ,\chk s)\xrightarrow[Li_+^*]{} D(V_+/G,\chk s_+)\rightarrow D(W/\Gam,\omega)\rightarrow D(U/\Gam, \omega)
\end{equation}
\begin{lemma} \label[lemma]{lem:ffb}
    $Li_-^{*}:(\W ,\chk s) \xrightarrow[]{} D(V_-/G,\potneg)$ is fully-faithful.
\end{lemma}
\begin{proof}
Recall the spectral sequence \eqref{eqn:morsOfBranes} that describes morphisms in a B-brane category. Using this description, fully-faithfulness follows if we show that the following complexes of vector spaces are isomorphic, for any $i$ and any pair $(\E,d_\E),(\F,d_\F)\in (\W ,\chk s)$.
\begin{gather*}
\left( H^i(V/G,\H om(\E,\F)),d_{\E,\F} \right)\\
\left( H^i(V_-/G,\H om(i_-^*\E,i_-^*\F)),d_{i_-^*\E,i_-^*\F} \right)
\end{gather*}
Now, both these complexes vanish for $i\neq 0$, as the window bundles have no higher Exts on both $V/G$ and $V_-/G$. And for $i=0$, the complexes are isomorphic due to Hartog's lemma.
\end{proof}

\begin{lemma}
    $Li_-^{*}:(\W ,\chk s) \xrightarrow[]{} D(V_-/G,\potneg)$ is essentially surjective.
\end{lemma}
\begin{proof}
This proof is analogous to the proof of \cite[Lemma 3.6]{Seg2011}. This has also been formalized in general terms in \cite[Lemma 4.10]{Seg2014}. We will summarize this technique here, adopting the conventions of the above cited papers.  By the proof of surjectivity at the level of derived categories \Cref{lem:surjectivityOfWindowFunctor}, we know that any vector bundle on $V_-/G$  has a $\C^*_R$-equivariant resolution by direct sum of vector bundles coming from the window and their shifts. Let us fix such a resolution for the underlying vector bundle of an arbitrary brane $(\E,d_\E)\in D(V_-/G,\potneg)$.
\begin{equation}
    0\rightarrow E_{-m}\xrightarrow{\partial_{-m}}\dots\xrightarrow{\partial_{-2}}E_{-1}\xrightarrow{\partial_{-1}}E_0\xrightarrow{\partial_0} \E\rightarrow 0
\end{equation}
Denoting $\til\E\equiv \bigoplus\limits_{i}E_{-i}(i)$, we get a complex $(\til\E,\partial)$ whose differential has R-charge 1. It turns out, using the fact that the bundles appearing in the resolution have no higher Exts between them, one can perturb this differential by inductively adding backward arrows, so that it eventually becomes a curved differential $d_{\til\E}=\partial+d$ having R-charge 1 and satisfying ${d_{\til\E}}^2=\chk s_-$. Moreover, it can be done in such a way that we have an isomorphism of branes $(\E,d_\E)\simeq(\til\E,d_{\til\E})$.

As $\til \E$ is made up of bundles coming from the window, it extends uniquely to a bundle $\hat\E\in \W$. By Hartog's lemma, the morphism $d_{\til\E}\in H^0(V_-/G,\H om(\til\E,\til\E))$ also lifts uniquely to a morphism $d_{\hat\E}\in H^0(V/G,\H om(\hat\E,\hat\E))$. So, finally we get a brane $(\hat\E,d_{\hat\E})\in (\W,\chk s)$ that restricts to $(\til\E,d_{\til\E})$.
\end{proof}

\begin{lemma}
    $(\W ,\chk s)\rightarrow D(U/\Gam, \omega)$, got by composing arrows in \eqref{eqn:functorsWithSuperpot}, is fully-faithful.
\end{lemma}
\begin{proof}
    This follows by exactly the same argument as in the proof of \Cref{lem:ffb} 
\end{proof}

Therefore, by inverting the first arrow in \eqref{eqn:functorsWithSuperpot} and then composing it with the rest, we conclude:
\begin{proposition}\label[proposition]{prop:ffFunctor}
We have a fully-faithful functor $D(V_-/G,\chk s_-) \rightarrow D\LGpos$, such that the underlying vector bundle of any B-brane in the image, admits a representative as a direct sum of bundles coming from the window $\W$.
\end{proposition}

\section{Proof of Equivalence}\label{sec:proofOfEquiv}
By \Cref{critNeg} and \Cref{lem:critPos} the fully-faithful functor between the B-brane categories, constructed in previous section (\Cref{prop:ffFunctor}), may be seen as a fully-faithful functor $\Phi : D(X)\rightarrow D(\bly)$ between the derived categories of the critical loci. Now we will show that this functor induces our desired equivalence $D(X)\simeq D(Y)$. We will focus here entirely on the (modified) positive side of the GLSM, the side where $\bly$ appears. Diagram \eqref{diag:positiveSide} sets up notations for the maps that we will use in this section. There, 
\begin{itemize}
    \item $\bar p:\overline{U/\Gam}\rightarrow\bly$ is the vector bundle projection $U/\Gam\xrightarrow{p}\blp$ restricted to $\bly\subset \blp$
    \item $i_0$ and $j_0$ are the zero section embeddings corresponding $p$ and $\bar p$ respectively
    \item $E$ and $\wtil{\mathcal{C}}$ are the exceptional divisors of the blow-ups $\blp$ and $\bly$ respectively
\end{itemize}
\begin{equation} \label{diag:positiveSide}
\begin{tikzcd}[row sep=large, column sep=large]
    & U/\Gam \arrow[d, shift right=0.6ex, "p"']
    & \overline{U/\Gam} \arrow[l, hook', "\bar{i}"'] \arrow[d, shift right=0.6ex, "\bar{p}"']
    & \\
    E \arrow[r, hook, "\mu"] \arrow[d]
    & \widetilde{\P }^8 \arrow[u, shift right=0.6ex, hook', "i_0"'] \arrow[d, "\widetilde{\pi}"']
    & \widetilde{Y} \arrow[l, hook', "j"'] \arrow[u, shift right=0.6ex, hook', "j_0"'] \arrow[d, "q"']
    & \widetilde{\mathcal{C}} \arrow[l, hook', "\nu"'] \arrow[d, "\bar{q}"'] \\
    \P ^2 \arrow[r, hook, "\bar{\mu}"]
    & \P ^8
    & Y \arrow[l, hook', "\bar{j}"']
    & \mathcal{C} \arrow[l, hook', "\bar{\nu}"']
\end{tikzcd}
\end{equation}
Let us set up a few more notations before proceeding. $\G$ be the image of the fully-faithful functor $D(V_-/G,\chk s_-) \rightarrow D\LGpos$ and $\kappa$ be the functor $\bar i_*\circ \bar p^*:D(\bly)\rightarrow D\LGpos$, which we will sometimes call the \textit{Kn\"orrer functor}. Now, the identification in \Cref{lem:critPos} is enabled exactly by $\kappa$. So, $\operatorname{im}(\Phi)=\kappa^{-1}(\G)$. 

We know that $D(Y)$ embeds in $D(\bly)$ via pullback along the blow-up and under this embedding, its left orthogonal is a copy of the derived category of the blow-up center $C$. More precisely, we have a semi-orthogonal decomposition $D(\bly)=\langle Lq^*D(Y), \ \nu_*\bar q^* D(C)\rangle$. In fact, we will consider the more general decomposition, where both the categories have been twisted by some multiple of the exceptional divisor. Here we identify the class of $\til C$ in $\til Y$ as that of $E$, since $\til C\subset \bly$ is just the divisor $E\cap  \til Y$.
\begin{equation}\label{eqn:SOD}
    D(\bly)=\langle Lq^*D(Y)(kE), \ \nu_*\bar q^* D(C)(kE)\rangle
\end{equation}
So, to show that the image of $D(X)$ (under $\Phi$) lands in some copy of $D(Y)$, it suffices to show that the image is right orthogonal to the corresponding copy of $D(C)$. Following gives a sufficient condition to achieve this orthogonality.
\begin{lemma}\label[lemma]{lem:suffCondForOrth}
    $\Phi(D(X)) \perp \nu_*\bar q^* D(C)(kE)$ i.e. $RHom(\nu_*\bar q^* D(C)(kE),\ \Phi(D(X)))=0$ holds, if the following cohomologies vanish
    \begin{equation}
        \{H^*(E, \F|_E(-lH-(k-1)E))\sp|\sp \F \in \text{the window}, l\in \mathbb{Z}\}
    \end{equation}
\end{lemma}
To prove this, we will need the following observations.
\begin{observation}\label[observation]{obs:cartesianSquare}
    $j_* \circ \kappa^{-1} = Li_0^*$
\end{observation}
\begin{proof}
    Even though the maps $i_0$ and $j_0$ are not flat, we still have the identity $Li_0^*\circ \bar i_* = j_*\circ Lj_0^*$, due to \cite[Corollary 2.27]{KuzHyperplaneSections}. Using this and the definition of $\kappa$, we have
\begin{equation}
\begin{gathered}
    Li_0^*\circ \kappa = Li_0^*\circ \bar i_*\circ \bar p^* = j_*\circ Lj_0^*\circ \bar p^* = j_* \\
    \Rightarrow j_* \circ \kappa^{-1} = Li_0^*
\end{gathered}
\end{equation}
\end{proof}
\begin{observation}\label[observation]{obs:quasiIsos}
    We have quasi-isomorphisms:
    \begin{itemize}
        \item $\mu_*\O_E\simeq(\O_{\blp}(-E) \rightarrow\O_{\blp})$
        \item $\nu_*\O_{\til C}\simeq(\O_{\bly}(-E) \rightarrow\O_{\bly})$
    \end{itemize}
\end{observation}
\begin{proof}
    The first quasi-isomorphism follows from the Koszul resolution of $E\subset\blp$
    \begin{equation}
        0\rightarrow \O_{\blp}(-E) \rightarrow\O_{\blp}\rightarrow {\mu_*\O_{E}}\rightarrow 0
    \end{equation}
    For the second one, notice that $\til C\subset \bly$ is just the divisor $E\cap  \til Y$. So, we have a similar looking Koszul resolution for $\til C\subset \bly$, which yields the second quasi-isomorphism.
\end{proof}

\begin{proof}[Proof of \Cref{lem:suffCondForOrth}]
As noted earlier, we can swap $\Phi(D(X))$ for $\kappa^{-1}\G$. As $\O_ C(H)$ is an ample line bundle on $ C$, its tensor powers $\O_ C(lH)$ form a spanning class of $D( C)$. Hence, $\nu_*\bar q^* \O_ C(lH)(kE)=\nu_*\O_{\til C}(lH+kE)$ form a spanning class of $\nu_*\bar q^* D( C)(kE)$. That is, the desired orthogonality is equivalent to showing that 
\begin{equation}
    RHom(\nu_*\O_{\til C}(lH+kE), \kappa^{-1}\G)=0 \forallsp l\in \Z
\end{equation}
For any $\E\in \G$, we use \Cref{obs:cartesianSquare} and \Cref{obs:quasiIsos} to write
\begin{equation}
\begin{aligned}
    &RHom_{\bly} (\nu_*\O_{\til C}(lH+kE), \kappa^{-1}\E))\\
        =&R\Gam_{\bly} R\H om(\O(lH+(k-1)E) \rightarrow\O(lH+kE), \kappa^{-1}\E))\\    
        =&R\Gam_{\bly} ((\O(lH+(k-1)E) \rightarrow\O(lH+kE))^\vee\derTensor \kappa^{-1}\E)) \\
        =&R\Gam_{\bly} ((\O(-lH-kE) \rightarrow \O(-lH-(k-1)E))[-1]\derTensor \kappa^{-1}\E)) \\
        =&R\Gam_{\blp} ((\O(-lH-kE) \rightarrow \O(-lH-(k-1)E)[-1]\derTensor j_*\circ\kappa^{-1}\E))\\
        =&R\Gam_{\blp} (\mu_*\O_E(-lH-(k-1)E)[-1]\derTensor \E|_{\blp}))  \\
        =&R\Gam_E (\E|_E(-lH-(k-1)E))[-1]
\end{aligned}
\end{equation}
Here, the restriction $\E|_{\blp}$ is under the zero section embedding $\blp\xrightarrow{i_0} \kb$ and $\E|_E$ is its further restriction to the exceptional divisor. On $\blp$ (seen in $\kb$), the superpotential vanishes and $R$-charge acts trivially. The same happens on $E\subset \blp$. Hence, $\E|_E$ is just a complex of vector bundles, each of which is a direct sum of bundles coming from the window. Now, if each sheaf in a complex is acyclic, then the complex itself is acyclic. So, for $R\Gam_E (\E|_E(-lH-(k-1)E))$ to vanish for an arbitrary B-brane $\E$, it suffices to show that $R\Gam_E (\F|_E(-lH-(k-1)E))$ vanishes for each vector bundle $\F$ coming from the window. Note that apriori, the latter statement is a stronger one, as there we have swapped B-branes for their vector bundle building blocks.
\end{proof}
It turns out that all these cohomologies indeed do vanish for a unique value of the twist $k$ appearing in the decomposed presentation of $D(\bly)$. We have:
\begin{lemma}\label[lemma]{lem:cohomForOrth}
    $H^*(E, \F|_E(-lH+ 3E))=0 \quad \forallsp\F \in \text{the window and } l\in \mathbb{Z}$
\end{lemma}
\begin{proof}
    Under the identification $\blp\simeq \P(\O_{\P^5}(-h)\oplus \O_{\P^5}^{\oplus3})$, the exceptional divisor $E$ embeds in $\blp$ as $\P(\O_{\P^5}^{\oplus3})=\P^5\times \P^2$. Now as $H$ is the relative $\O(1)$ of the projective bundle, we have $E\sim H-h$. Also, in this view of $E$, we have $\O_E(h)=\O_E(1,0)$ and $\O_E(H)=\O_E(0,1)$. So, $\F|_E(-lH+aE))=\F|_E(-a,a-l)$. It suffices to just focus on the cases where $\F=\O, \bunPos$ or $\Sym^2\bunPos$, as we need to show vanishings for arbitrary twists by $H$ anyway. \\\\
    \textit{For $\F=\O$}, we have $H^*(E, \O(-a, a-l))=0$ for $a\in1,2,3,4,5$ and $l\in \Z$.\\\\
    \textit{For $\F=\bunPos$}, recall the first short exact sequence in \Cref{obs:ses}. Pulling it back to $E$, we get short exact sequence 
    \begin{equation}
        0\rightarrow \O(-1,-1)\rightarrow \bunPos|_E\rightarrow \O(1,0)\rightarrow 0
    \end{equation}
    This is split, as $RHom_E(\O(1,0),\O(-1,-1))=0$. So, $\bunPos|_E \simeq \O(-1,-1)\oplus\O(1,0)$. So,
    \begin{equation}
    \begin{aligned}
        H^*(E, \bunPos|_E(-a, a-l))
        &\simeq H^*(E, \O(-a-1, a-l-1))\oplus H^*(E, \O(-a+1, a-l)) \\
        &=0  \quad\text{for } a \in 2,3,4 \text{ and } l \in \Z 
    \end{aligned}
    \end{equation}
    \textit{For $\F=\Sym^2\bunPos$}, we have $\Sym^2\bunPos|_E \simeq O(-2,-2)\oplus\O(0,-1) \oplus \O(2,0)$. So, 
 \begin{equation}
 \begin{aligned}
        &H^*(E, \Sym^2\bunPos|_E(-a, a-l))\\
        &\simeq
        H^*(E, \O(-a-2, a-l-2))\oplus H^*(E, \O(-a, a-l-1))
        \oplus H^*(E, \O(-a+2, a-l)) \\
        &=0  \quad\text{for } a = 3 \text{ and } l \in \Z 
    \end{aligned}
 \end{equation}
    Therefore, we have our desired cohomology vanishings if and only if $a=3$.
\end{proof}
By \Cref{lem:suffCondForOrth} and \Cref{lem:cohomForOrth}, we conclude that $\nu_*\bar q^* D( C)(-2E) \perp \Phi(D(X)))$,
as sub-categories of $D(\bly)$. This semi-orthogonality and the semi-orthogonal decomposition \eqref{eqn:SOD} together imply that the fully-faithful functor $\Phi$ lands in $Lq^*D(Y)(-2E)\simeq D(Y)$. Hence, $\Phi$ induces a fully-faithful functor $D(X)\rightarrow D(Y)$. As it is a fully-faithful functor between derived categories of smooth projective varieties, it is automatically admissible. So, in summary, we have constructed a fully-faithful, admissible functor $D(X)\rightarrow D(Y)$. Such a functor must be an equivalence, as derived categories of Calabi-Yau varieties are indecomposable. Therefore, we conclude our proof of equivalence, and with it, our long tale.

\begin{theorem}\label[theorem]{thm:derivedEquivalence}
    There exists an equivalence $D(X)\simeq D(Y)$ of derived categories of Calabi-Yau threefolds $X$ and $Y$. 
\end{theorem}

\begin{remark}\label[remark]{rem:windowShiftAutoEquivalences}
    In fact, there exists a $\Z$-indexed family of such equivalences, as we have $\Z$-many choices of windows $\{\W_k\}_k$, due to \Cref{rem:windowShift}. Using them, one can build the so-called \textit{window shift auto-equivalences} for both $D(X)$ and $D(Y)$.
    \begin{equation}
    \begin{aligned}
        \Phi_X^{k,l}:D(X)\xrightarrow{\sim}(\W_k,\chk s)\xrightarrow{\sim} D(Y) \xrightarrow{\sim} (\W_l,\chk s)\xrightarrow{\sim} D(X) \\
        \Phi_Y^{k,l}:D(Y)\xrightarrow{\sim}(\W_k,\chk s)\xrightarrow{\sim} D(X) \xrightarrow{\sim} (\W_l,\chk s)\xrightarrow{\sim} D(Y)
    \end{aligned}
    \end{equation}
If we fix the identifications $D(X)\simeq(\W_0,\chk s)$ and $D(Y)\simeq(\W_0,\chk s)$, then $D(X)\simeq(\W_k,\chk s)\simeq D(X)$ and $D(Y)\simeq(\W_l,\chk s)\simeq D(Y)$ are the auto-equivalences $(\_ )\otimes \O(k\xi)\in Aut(D(X))$ and $(\_ )\otimes \O(-lH)\in Aut(D(Y))$ respectively. By mirror symmetry (see \cite{monodromyCalc}), they are expected to correspond to taking $k$ and $-l$ many loops around the MUM points $0$ and $\infty$ of the SKMS (figure \ref{fig:SKMS}), respectively. Hence, a window shift auto-equivalence $\Phi_X^{k,l}$ is expected to correspond to the homotopy class of a path where we start at $\chk X$, take $k$ many loops around $0$, move to $\chk Y$ along some path $\gamma$, take $-l$ many loops around $\infty$, then finally move back to $\chk X$ along $-\gamma$.
\end{remark}

\begin{remark}\label[remark]{rem:4Windows}
Disregarding loops around the singularities, there are four homotopy classes of paths from $\chk X$ to $\chk Y$ in the conjectured common SKMS of $X$ and $Y$ (figure \ref{fig:SKMS}). So, by mirror symmetry, we expect there to exist four "fundamental" windows, giving in total four "fundamental" equivalences $D(X)\simeq D(Y)$, including the one constructed in \Cref{thm:derivedEquivalence}. We will leave the discovery of the three other windows and the uncovering of their precise relationship to the SKMS, to future work.
\end{remark}

\appendix
\section{Computer Algebra}
Several computations throughout the paper have been carried out with the aid of computer algebra systems. In particular, Macaulay2 \cite{macaulay2} has been extensively used to compute resolutions, ideals and handle explicit toric geometry and related combinatorics. On the other hand, most of the calculations involving tensor products and cohomology of complexes of homogeneous vector bundles (e.g. Hodge numbers of zero loci and tilting properties), have been carried out using the Python package \href{https://github.com/marcorampazzo/homogeneous-varieties}{\texttt{homogeneous-varieties}} produced by the authors, and also publicly available.

All the scripts included in this appendix, are available in the repository \href{https://github.com/marcorampazzo/the-last-CY-pair}{\texttt{the-last-CY-pair}}, together with detailed instructions on how to execute the code and reproduce the results.

\subsection{Hodge Number Computations in \Cref{sec:geometry}}\label[appendix]{app:hodgeCode}
The repository \href{https://github.com/marcorampazzo/homogeneous-varieties}{\texttt{homogeneous-varieties}} contains a method for computing Hodge numbers of zero loci of general sections of homogeneous vector bundles over rational homogeneous varieties, give that they are completely determined by the relevant (Koszul) long exact sequences. The method permits to compute the Hodge diamond of $X$ quoted in
\Cref{sec:geometry}. However, the Hodge diamond of $\wt Y$ claimed in the proof of \Cref{prop:hodge_numbers} requires an additional observation, since the exact sequences alone are not enough to ensure that $h^{1,1} = 2$. In fact, the constraint $h^{1,1}\geq 2$ follows by the fact that $\wt Y$ is a blowup: indeed, $H^{1,1}$ contains the span of the pullback of the $(1,1)$-class and the exceptional divisor. This additional condition resolves the ambiguity. For this reason, we introduce a modified version of the function, which can be found in \href{https://github.com/marcorampazzo/the-last-CY-pair}{\texttt{the-last-CY-pair}} under the name \verb|blowup_hodge_numbers|, in the file \verb|hodge_numbers.py|. Executing the file computes the Hodge numbers of both $X$ and $\wt Y$.

\subsection{Toric Geometry Computations in \Cref{sec:toricComputationsForY}} \label[appendix]{app:toricCode}
The scripts appearing in this appendix are in the Macaulay2 programming language and are available in the repository \href{https://github.com/marcorampazzo/the-last-CY-pair}{\texttt{the-last-CY-pair}} as the file \verb|mirror_principal_period.m2|. First we load the following packages which we will use extensively throughout the code. 
\begin{lstlisting}[style=macaulay2]
needsPackage "NormalToricVarieties"
needsPackage "Polyhedra"
\end{lstlisting}

Next we define some methods. Here, \verb|toricIdealFromTorusCharacters| takes a collection of torus characters (encoded in the columns of a matrix L) and returns the ideal (in a polynomial ring S) of the closure of the image of the parameterizing map into an affine space, given by monomials corresponding the given torus characters. This ideal will be homogeneous if and only if the given torus characters live on an affine hyperplane in the ambient real space. In that case, the ideal obtained may be interpreted as the homogenous ideal of the projectivization of the closure of the image of the parameterizing map.
\begin{lstlisting}[style=macaulay2]
dot = (a,b) -> sum apply(#flatten entries a, i -> (flatten entries a)_i*(flatten entries b)_i)

encodeBinomial = f-> (transpose matrix exponents f)_0 - (transpose matrix exponents f)_1

createMonomial = (l, S) -> product apply(#l, i -> ((gens S)_i)^(l_i)) 

toricIdealFromTorusCharacters = (L, S)->(
    T := (coefficientRing S)[t_1..t_(rank target L), u_1..u_(rank target L)];
    T = T/ideal(toList(1..(rank target L))/(i -> t_i*u_i - 1)); -- coordinate ring of a torus
    parametrizingMonomials := apply(rank source L, j -> product apply(rank target L, i -> if L_(i,j) >= 0 then (t_(i+1))^(L_(i,j)) else (u_(i+1))^(-L_(i,j))));
    return ideal mingens ker map(T, S, parametrizingMonomials)
)
\end{lstlisting}
Next, we construct the ideals of the embeddings of the smooth Fano variety $\til W$ and the toric variety $W$ in $\P^{11}$ and verify that $W$ is a flat degeneration of $\til W$, by checking that their Hilbert polynomials match. We then find a collection of torus characters parameterizing this embedding. We verify that this is the correct collection, in the last line below.
\begin{lstlisting}[style=macaulay2]
K = ZZ/101
P11 = K[x_0..x_11]

M1 = random(P11^6,P11^{6:-1}); M1 = M1 - (transpose M1)
v1 = random(P11^6,P11^{1:-1})
idealW1 = ideal(M1*v1, pfaffians(6,M1))

M = matrix {
    {   0,  x_0,    0,    0,    0, -x_5},
    {-x_0,    0,  x_1,    0,    0,    0},
    {   0, -x_1,    0,  x_2,    0,    0},
    {   0,    0, -x_2,    0,  x_3,    0},
    {   0,    0,    0, -x_3,    0,  x_4},
    { x_5,    0,    0,    0, -x_4,    0}
}
v = transpose matrix{{x_6..x_11}}
idealW = ideal(M*v, pfaffians(6,M))

hilbertPolynomial idealW == hilbertPolynomial idealW1 -- true

integralRelns = sub(matrix transpose(((flatten entries mingens idealW)/(f -> encodeBinomial f))/(c -> flatten entries c)), QQ)
torusCharacters = transpose mingens ker transpose integralRelns 
toricIdealFromTorusCharacters(sub(torusCharacters, ZZ), P11) == idealW -- true 
\end{lstlisting}

Next, we obtain the fan $\Sigma$ that describes $W$ as an abstract toric Fano variety. Then using this description, we verify that $W$ is indeed a terminal Gorenstein toric variety having Picard rank 1. 

\begin{lstlisting}[style=macaulay2]
P = toSublattice convexHull torusCharacters 
sigma = normalFan P
rank source rays sigma -- 18 = no. of rays of sigma

W = normalToricVariety sigma
isFano W -- true. So, W is Gorenstein Fano
facesAsCones(0, sigma)/(c -> rank source rays c == #(hilbertBasis c)) -- this is all true. So, W has at worst terminal singularities
rank picardGroup W -- 1
\end{lstlisting}

Next, we construct the three divisors $H_1,H_2,H_3$ in the hyperplane class of $W$, that form a partition of $-K_W \equiv -\text{toricDivisor}(W) = \til W_0+...+\til W_{17}$. $H_1$ is in the hyperplane class, as it is an effective divisor that generates Pic($W$). Then $H_2$ and $H_3$ are also in the hyperplane class as they differ from $H_1$ by principal divisors.
\begin{lstlisting}[style=macaulay2]
H1 = - toricDivisor(flatten entries ((fromCDivToWDiv W)*(transpose fromCDivToPic W)), W)
H2 = H1 + toricDivisor(flatten entries((rays sigma)^{0}), W)
H3 = H1 + toricDivisor(flatten entries((rays sigma)^{1}), W)
H1 + H2 + H3 == - toricDivisor W -- true
\end{lstlisting}

Next, we construct the monoid $L(\Sigma)$ and find its minimal generating set, which in this case is its Hilbert basis.
\begin{lstlisting}[style=macaulay2]
relnsAmongRays = mingens ker rays sigma
relnsAmongRaysAsCone = coneFromVData(relnsAmongRays|(-relnsAmongRays))
LSigma = intersection(relnsAmongRaysAsCone, posOrthant ambDim relnsAmongRaysAsCone)
dim LSigma -- 12
gensLSigma = hilbertBasis LSigma
\end{lstlisting}

Next, we compute the ideal of $A(\Sigma)\subset \C^{18}$. Then using it, we compute the restrictions of the monomials corresponding to the generators of $L(\Sigma)$, to $A(\Sigma)$. We observe that for any generator $f_i$, we have $\ul t^{f_i}|_{A(\Sigma)}=(t_0t_9t_{17})^{d_i}$, where $d_i$ turns out to be equal to $\langle f_i,H_1\rangle$. So, we take our variable $z$ to be $t_0t_9t_{17}$.
\begin{lstlisting}[style=macaulay2]
R = K[t_0..t_(rank source rays sigma-1)]
conesSigma = join apply(1..dim sigma, d -> cones(d, sigma))
relns = flatten (conesSigma/(c->(
        temp := mingens ker((rays sigma)_c);
        apply(rank source temp, j -> product(apply(rank target temp, i -> if temp_(i,j) > 0 then (R_(c_i))^(temp_(i,j)) else 1)) - product(apply(rank target temp, i -> if temp_(i,j) < 0 then (R_(c_i))^(-temp_(i,j)) else 1)))
    )
))
idealASigma = saturate(ideal relns, product gens R)

generatingMonomials = gensLSigma/(g -> createMonomial(flatten entries g, R/idealASigma))
d = generatingMonomials / (m -> (degree m)_0); d = d/gcd(d)
gensLSigma/(f -> dot(f,H1)) == d -- true
\end{lstlisting}
Next, we find the 12 intrinsic coordinates on $L(\Sigma)$, among the 18 ambient space coordinates. Those coordinates turn out to be $\{l_4,l_6,l_7,l_9,l_{10},l_{11},l_{12},l_{13},l_{14},l_{15},l_{16},l_{17}\}$. Then we find the inequalities defining $L(\Sigma)$ and compute $\langle\ul l, H\rangle$, expressing everything in the above intrinsic coordinates. Finally, re-naming the intrinsic coordinates as ${m_0,...,m_{11}}$ yields the final expression \eqref{eqn:mirrorPrincipalPeriodForY} for the power series $\Phi_0(z)$. 
\begin{lstlisting}[style=macaulay2]
L = ZZ[l_0..l_17]
relnsAmongL = sub(reducedRowEchelonForm sub(hyperplanes LSigma, K), ZZ)
temp = mutableMatrix relnsAmongL;
toSub = {}
for i in 0..((rank target relnsAmongL)-1) do (
    k = -1;
    for j in 0..((rank source relnsAmongL)-1) do if temp_(i,j) != 0 then (k = j; break);
    if k >= 0 then(
        temp_(i,k) = 0;
        toSub = append(toSub, L_k => - dot(temp^{i}, vars L));
    );
)
inequalitiesDefiningLSigma = sub(halfspaces(LSigma)*(transpose vars L), toSub)
lDotH = sub(dot(vars L, H1), toSub) -- the intersection condition for nth coefficient is "lDotH" = n.
\end{lstlisting}

\subsection{Proof of \Cref{prop:PicardFuchs}} \label[appendix]{app:proofOfPicardFuchs}
The scripts appearing in this appendix are in the Macaulay2 programming language and are available in the repository \href{https://github.com/marcorampazzo/the-last-CY-pair}{\texttt{the-last-CY-pair}} as the file \verb|mirror_PF_operator.m2|.

Denote by $a_n$, the coefficient of $z^n$ in the power series $\Phi_0(z)$ (given in \Cref{eqn:mirrorPrincipalPeriodForY}). By inspecting the formulas for the linear forms $f_i(\ul m)$ appearing there, we notice that we can write $a_n$ as a sum over a product of multinomials, as follows.
\begin{equation}
a_n = \sum_{\substack{\ul m \in \Z^{12}_{\geq0}\\f_i(\ul m)\geq 0 \ \forall i\\ m_4+...+m_{11} =n}}
    \binom{n}{m_0, f_2(\ul{m}), f_3(\ul{m}), f_5(\ul{m}), f_6(\ul{m})}\binom{n}{m_1, m_2, m_3, f_1(\ul{m}), f_4(\ul{m})}\binom{n}{m_4,..., m_{11}}
\end{equation}
To work with this series further, we will need the following lemmas.
\begin{lemma}\label[lemma]{lem:combIdentity}
Consider a collection of $d$ many integral linear forms
$\{r_\mu(\ul m)\equiv\sum_{i} \sum_{j=1}^{\ell_i} {r}_\mu^{i,j} m_{ij} \}_{\mu\in1,...,d}$ in a collection of integral variables $\ul m\equiv (m_{ij})_{ij}$. Then, the following holds for all $n\geq 0$.

\begin{equation}
  \sum_{\substack{m_{ij} \ge 0 \\ \sum_{j} m_{ij} = n \\ {r}_\mu(\ul{m})=0}}
\ \prod_{i=1}^{k} \binom{n}{m_{i1} ,..., m_{i\ell_i}}=CT(F_1^n...F_k^n)
\end{equation}
where CT stands for "constant term" and $F_i$ are the following Laurent polynomials
\begin{equation}
    F_i = \sum_{j=1}^{\ell_i} \ul{u}^{\ul{r}^{i,j}} \in \mathbb{C}\big[u_1^{\pm 1}, \dots, u_d^{\pm 1}\big]
\end{equation}

\end{lemma}
\begin{proof}
Using multinomial theorem, we have
\begin{align*}
\sum_{\substack{m_{ij} \ge 0 \\ \sum_{j} m_{ij} = n \\ {r}_\mu(\ul{m})=0}}
\ \prod_{i=1}^{k} \binom{n}{m_{i1} ,..., m_{i\ell_i}} 
&= CT \sum_{\substack{m_{ij} \ge 0 \\ \sum_j m_{ij} = n}}
\ \prod_{i=1}^{k} \binom{n}{m_{i1}, ,..., ,m_{i\ell_i}} \,
\ul{u}^{\sum_i \sum_{j}\ul{r}^{i,j} m_{ij}}\\
&= CT \sum_{\substack{m_{ij} \ge 0 \\ \sum_j m_{ij} = n}}
\ \prod_{i=1}^{k} \left( \binom{n}{m_{i1} ,..., m_{i\ell_i}} \,
\ul{u}^{\sum_{j} \ul{r}^{i,j} m_{ij}} \right) \\
&= CT \prod_{i=1}^{k}\left( \sum_{\substack{m_{ij} \ge 0 \\ \sum_{j} m_{ij} = n}}
\binom{n}{m_{i1} ,..., m_{i\ell_i}} \prod_{j=1}^{\ell_i} \left( \ul{u}^{\ul{r}^{i,j}} \right)^{m_{ij}}\right) \\
&= CT \prod_{i=1}^{k} \left( \sum_{j=1}^{\ell_i} \ul{u}^{\ul{r}^{i,j}} \right)^{n}
\end{align*}
\end{proof}

\begin{lemma}\label[lemma]{lem:coefficientsOfSoln}
Consider operator $D = z^m P_0(\Theta) + z^{m-1} P_1(\Theta) + \dots + P_m(\Theta)\in \C[z,\Theta]$, where $P_0, \dots, P_m$ are some polynomials. If $\varphi = \sum_{n= 0}^\infty a_n z^n$ satisfies $D\varphi=0$, then $(a_n)_n$ satisfy the following reccurence relation.
\begin{equation}
P_0(n) a_n + P_1(n+1) a_{n+1} + \dots + P_m(n+m) a_{n+m} = 0 \quad \forall n\geq0
\end{equation}
\end{lemma}
\begin{proof}
    This is easily derived by noting that $z.\sum_{n=0}^\infty a_n z^n=\sum_{n= 1}^\infty a_{n-1} z^n$ and $\Theta.\sum_{n \ge 0} a_n z^n=\sum_{n \ge 0} na_n z^n$.
\end{proof}
Using \Cref{lem:combIdentity} we can cook up three Laurent polynomials  $F_1,F_2,F_3\in\C[x_1^{\pm1},..,x_6^{\pm1}]$ such that $a_n=CT((F_1F_2F_3)^n)$. They are
\begin{equation}
\begin{aligned}
F_1 &= \frac{x_1 x_2 x_3}{x_5 x_6} + \frac{x_1 x_2}{x_6} + \frac{x_2 x_3 x_4}{x_5 x_6} + x_1 + x_4, \\
F_2 &= \frac{x_5 x_6}{x_2} + x_2 + x_3 + x_5 + x_6, \\
F_3 &= \frac{1}{x_1 x_2} + \frac{x_6}{x_1 x_2 x_3} + \frac{x_5 x_6}{x_1 x_2^2 x_3} + \frac{1}{x_3 x_4} + \frac{x_5}{x_2 x_3 x_4} + \frac{1}{x_2 x_4} + \frac{x_6}{x_2 x_3 x_4} + \frac{x_5 x_6}{x_2^2 x_3 x_4}.
\end{aligned}
\end{equation}
Using the following Macaulay2 code, we find that any exponent $(a_1,...,a_6)$ appearing in $F_1F_2F_3$ satisfies the relations $a_1+a_4=a_2+a_3+a_5+a_6=0$.

\begin{lstlisting}[style = macaulay2]
needsPackage "NormalToricVarieties"

T = QQ[x_1..x_6, y_1..y_6]/ideal(apply(6, i->x_(i+1)*y_(i+1)-1)) -- coordinate ring of a torus
F1 = x_1*x_2*x_3*y_5*y_6 + x_1*x_2*y_6 + x_2*x_3*x_4*y_5*y_6 + x_1 + x_4
F2 = y_2*x_5*x_6 + x_2 + x_3 + x_5 + x_6
F3 = y_1*y_2 + y_1*y_2*y_3*x_6 + y_1*y_2^2*y_3*x_5*x_6 + y_3*y_4 + y_2*y_3*y_4*x_5 + y_2*y_4 + y_2*y_3*y_4*x_6 + y_2^2*y_3*y_4*x_5*x_6

ee = exponents(F1*F2*F3)
vv = ee/(e->e_{0,1,2,3,4,5}-e_{6,7,8,9,10,11}) -- all exponents appearing in F1*F2*F3
ker matrix vv -- we get two relations
\end{lstlisting}

So, $F_1F_2F_3$ depends only on the variables
\begin{equation}
    u_1\equiv \frac{x_1}{x_4},u_2\equiv\frac{x_2}{x_6},u_3\equiv\frac{x_3}{x_6},u_4\equiv\frac{x_5}{x_6}
\end{equation}
Now, $F_1F_2F_3$, regarded as a function in $u_1,...,u_4$, is just $G_1G_2G_3$, where $G_i$ is the Laurent polynomial in $\C[u_1^{\pm1},...,u_4^{\pm1}]$ given by putting $x_4=x_6=1$ in $F_i$, and then renaming the remaining variables. Explicitly, they are
\begin{equation}
\begin{aligned}
G_1 &= \frac{u_1 u_2 u_3}{u_4} + u_1 u_2 + \frac{u_2 u_3 }{u_4 } + u_1 + 1, \\
G_2 &= \frac{u_4 }{u_2} + u_2 + u_3 + u_4 + 1, \\
G_3 &= \frac{1}{u_1 u_2} + \frac{1}{u_1 u_2 u_3} + \frac{u_4 }{u_1 u_2^2 u_3} + \frac{1}{u_3} + \frac{u_4}{u_2 u_3} + \frac{1}{u_2} + \frac{1}{u_2 u_3} + \frac{u_4 }{u_2^2 u_3}.
\end{aligned}
\end{equation}
Clearly, we again have $a_n=CT((G_1G_2G_3)^n)$. Next, we verify that the Newton polytope of $G_1G_2G_3$ is a reflexive 4-dimensional polytope in $\mathbb{R}^4$ and hence, in particular, contains the origin in its interior. 
\begin{lstlisting}[style=macaulay2]
T = QQ[u_1..u_4, v_1..v_4]/ideal(apply(4, i->u_(i+1)*v_(i+1)-1)) -- coordinate ring of a torus

G1 = u_1*u_2*u_3*v_4 + u_1*u_2 + u_2*u_3*v_4 + u_1 + 1
G2 = u_4*v_2 + u_2 + u_3 + u_4 + 1
G3 = v_1*v_2 + v_1*v_2*v_3 + v_1*v_2^2*v_3*u_4 + v_3 + v_2*v_3*u_4 + v_2 + v_2*v_3 + v_2^2*v_3*u_4
E = exponents(G1*G2*G3)/(e -> e_{0,1,2,3} - e_{4,5,6,7})
NewtonPolytope = convexHull transpose matrix E
dim NewtonPolytope
isReflexive NewtonPolytope
\end{lstlisting}

So, by \cite[Proposition 7]{LGFano}, we have that $\Phi_0(z)$ is the principal period of the pencil $\{V(G_1G_2G_3-t)\subset(\C^*)^4)\}_{t\in \C}$ of threefolds. In other words, the Picard-Fuchs operator $D\in\C[z,\Theta]$ of this pencil annihilates $\Phi_0(z)$. It is well known that the degree of the (finite) discriminant locus of any semi-stable compactification of this pencil should be an upper bound for the $z$-degree of $D$. We consider the case where the fibers are compactified to subvarieties of $\P^4$ in the obvious way and the base is compactified to $\P^1$ in the obvious way. We implement this compactification in the following code and compute the degree of its discriminant locus, which we find to be equal to 5.
\begin{lstlisting}[style=macaulay2]
S = QQ[u_0..u_4, t_0, t_1] -- homogenous coordinate ring of the ambient space of the pencil.

Q = u_1*u_2^3*u_3*u_4
P = sub(G1*G2*G3*sub(Q,T), S)
pencil = homogenize(t_0*P - t_1*Q, u_0)

J = saturate((ideal(apply(5, i -> diff(u_i, pencil))) + ideal(pencil)), ideal(sub(pencil, {t_0=>0, t_1=>1}), sub(pencil, {t_0=>1, t_1=>0})));
discriminantLocus = ker map(S/J, QQ[t_0,t_1], matrix{{(t_0)_S,(t_1)_S}})
degree discriminantLocus -- 5
\end{lstlisting}
So, $D$ must have $z$-degree at most 5. Moreover, for $D$ to be the mirror Picard-Fuchs operator, it must satisfy the additional constraints of having $\Theta$-degree equal to 4 and having $0$ as the unique indicial root. Taking all these constraints into account, one finds that $D$ can be written in the form 
\begin{equation}D=z^5P_0(\Theta)+z^{4}P_1(\Theta)+z^3P_2(\Theta)+z^2P_3(\Theta)+zP_4(\Theta)+49\Theta^4
\end{equation}
where $P_k$ are some polynomials of degree at most 4. So, we see that $D$ is determined by $25$ many indeterminates, constituting the coefficients of the unknown polynomials $P_k$. By \Cref{lem:coefficientsOfSoln}, these indeterminates in turn are determined by the first $30$ many coefficients of its solution, $\Phi_0(z)$. Using the code below, we compute these coefficients and then solve the resulting linear equations in the indeterminates to compute $D$. We find that $D$ is indeed equal to the given operator $D_Y$, thus concluding the proof.
\begin{lstlisting}[style = macaulay2]
A = (n) -> (sub(last terms((G1*G2*G3)^n), ZZ))
fullCoefficientList = (f) -> (flatten entries basis(first degree f, ring f))/(m -> coefficient(m, f))

C = QQ[z, theta][c_(0,0)..c_(4,4)]
P = apply(5, i -> ((n) -> c_(i,0) + c_(i,1)*n + c_(i,2)*n^2 + c_(i,3)*n^3 + c_(i,4)*n^4))
a = apply(30, n -> A(n))

eqns = apply(25, n -> {a_n*P_0(n) + a_(n+1)*P_1(n+1) + a_(n+2)*P_2(n+2) + a_(n+3)*P_3(n+3) + a_(n+4)*P_4(n+4), -49*a_(n+5)*(n+5)^4}); -- affine relations among the indeterminates
eqnsAsMat = matrix(eqns/(e -> fullCoefficientList e_0))
rank eqnsAsMat -- 25
targetAsMat = transpose matrix{eqns/(e -> e_1)}
solns = inverse(eqnsAsMat)*targetAsMat

D = sub(z^5*P_0(theta) + z^4*P_1(theta) + z^3*P_2(theta) + z^2*P_3(theta) + z*P_4(theta) + 49*theta^4, transpose solns)
\end{lstlisting}

\subsection{The proof of \Cref{lem:ff}}
\label[appendix]{app:tiltingCode}
The proof of fully-faithfulness of $Li^*_-$ requires the vanishing of the expression:
\begin{equation}
    H^d\bigl(G(2,V_6),\ \E^\vee\otimes\F\otimes\Sym^j\bunNeg(i)\bigr)=0
\end{equation}
for all $d>0$ and $i\geq j\geq0$, over all $225$ ordered pairs $(\E, \F)$ of window bundles (see Equation \ref{eq:ff_proven_in_appendix}). This is handled by the script \verb|pretilting.py|, which can be found in \href{https://github.com/marcorampazzo/the-last-CY-pair}{\texttt{the-last-CY-pair}}.

However, one should observe that proving the claim for $i\geq j \geq 0$ would in principle require to compute an infinite number of vanishings. This task can be reduced to finitely many checks as follows.
\begin{itemize}
    \item If an irreducible summand of $\E^\vee\otimes\F$ has the shape $\Sym^r\bunNeg(c)$, then its tensor product with $\Sym^j\bunNeg(i)$ decomposes into $\Sym^{r+j-2q}\bunNeg(c+i+q)$, where $0\leq q\leq\min(r,j)$.
    \item For $i\geq\max(0,\max(-c))$, with the maximum over the summands of the fixed pair, all these weights are dominant, so there is no higher cohomology.
    \item The code checks the remaining cases by Borel--Weil--Bott. It reports $654$ finite cases, all of them displaying no higher cohomology.
\end{itemize} The finite computation can be reproduced by running \verb|pretilting.py|.

\subsection{Proof of \Cref{cor:cohomVanishings}}
\label[appendix]{app:cohomCode}

The script \verb|open_cohomology.py|, which also is available in
\href{https://github.com/marcorampazzo/the-last-CY-pair}{\texttt{the-last-CY-pair}},
computes the degrees carrying non-zero cohomology in
\begin{equation}
    H^\bullet\bigl(\kb,\O (ah+bH)\bigr).
\end{equation}
It produces the table appearing in the proof of
\Cref{cor:cohomVanishings} and checks all the vanishings stated there.
Let us explain how this computation follows from
\Cref{lem:cohomExpression}, and how the infinite sums appearing in
its proof reduce to finitely many checks.

The statement of \Cref{lem:cohomExpression} concerns degrees
$d\geq2$, under the assumption $a\geq-5$. To compute degrees $0$
and $1$ as well, let us recall the identity:
\begin{equation}
H^d\bigl(\kb,\O (ah+bH)\bigr)
=
\bigoplus_{i\geq0}
H^d\!\Bigl(
\blp,\,
\pi^*\Sym^i\Q^\vee
\otimes\O \bigl((i+a)h+(i+b)H\bigr)
\Bigr)
\end{equation}

For the relative twist $t=i+b$, the projective-bundle
pushforward formula shows that the range $-4<i+b<0$ does not contribute, and the remaining cases are precisely the two ranges denoted by $\Sigma_+$
and $\Sigma_-$ in the proof of \Cref{lem:cohomExpression}:
\[
i\geq\max\{0,-b\},
\qquad\text{and}\qquad
0\leq i\leq-b-4.
\]
The second range is empty when $b>-4$.

Let us call $d$ the cohomology degree, and write
\begin{equation}
H^d\bigl(\kb,\O (ah+bH)\bigr)
=\Sigma_+^d\oplus\Sigma_-^d,
\end{equation}
where
\begin{equation}
\begin{aligned}
\Sigma_+^d
&=
\bigoplus_{i\geq\max\{0,-b\}}
\ \bigoplus_{j=0}^{i+b}
H^d\!\Bigl(\P^5,\Sym^i\Q^\vee(i+a+j)\Bigr)
^{\oplus\binom{i+b-j+2}{2}},\\
\Sigma_-^d
&=
\bigoplus_{i=0}^{-b-4}
\ \bigoplus_{j=0}^{-i-b-4}
H^{d-3}\!\Bigl(\P^5,\Sym^i\Q^\vee(i+a-j-1)\Bigr)
^{\oplus\binom{-i-b-j-2}{2}}.
\end{aligned}
\end{equation}

Note that the expression for $\Sigma_-^d$ we use here, is slightly different from the one appearing in
\Cref{lem:cohomExpression}, which can easily be obtained by the former via Serre duality on $\P^5$. In particular, observe that the sum $\Sigma_-^d$ is finite: we can leave its evaluation to the script.\\
\\
To handle $\Sigma_+^d$,
write $m=i+a+j$. The Borel--Weil--Bott weight for
$\Sym^i\Q^\vee(m)$, after adding
$\rho=(5,4,3,2,1,0)$, is
\[
(m+5,\ i+4,\ 3,\ 2,\ 1,\ 0)
\]
(see \cite[Appendix A]{borisovcaldararuperry} for an explanation of the Borel--Weil--Bott algorithm for Grassmannians). If $-5\leq m\leq-2$, this weight has repeated entries,
so all cohomology vanishes. If $m\geq-1$, the last four
entries are smaller than the first two, and
Borel--Weil--Bott gives
\begin{equation}
\left\{
q\ :\ H^q\bigl(\P^5,\Sym^i\Q^\vee(m)\bigr)\neq0
\right\}
=
\begin{cases}
\{0\},&m\geq i,\\
\varnothing,&m=i-1,\\
\{1\},&-1\leq m\leq i-2.
\end{cases}
\end{equation}
Thus, in the range
of \Cref{cor:cohomVanishings}, all contributions in degrees
at least $2$ already come from the finite sum $\Sigma_-^d$.

It remains to account for degrees $0$ and $1$ in the infinite
sum $\Sigma_+^d$. For this purpose, the script sets
\[
N=\max\{0,-b,-a-1,-a-b\}.
\]
For every $i\geq N$, we have
\[
i\geq0,\qquad
i+b\geq0,\qquad
i+a\geq-1,\qquad
i+b\geq-a.
\]
The first two inequalities place $i$ in the range of
$\Sigma_+^d$. The third ensures that $m=i+a+j\geq-1$
for every $0\leq j\leq i+b$, so the preceding
Borel--Weil--Bott computation applies. The last inequality,
together with $i+b\geq0$, ensures that
$j=\max\{0,-a\}$ is an admissible choice.

This choice gives $m\geq i$, hence a non-zero contribution
in degree $0$. If $a\leq-2$, the admissible choice $j=0$
also gives $-1\leq m=i+a\leq i-2$, hence a non-zero
contribution in degree $1$. If $a\geq-1$, every admissible
$j$ gives $m\geq i-1$, so degree $1$ cannot occur.
Consequently, for every $i\geq N$, the combined contribution
of the summands indexed by $0\leq j\leq i+b$ has non-zero
cohomology precisely in degrees
\[
\begin{cases}
\{0\},&a\geq-1,\\
\{0,1\},&a\leq-2.
\end{cases}
\]

The script therefore computes the finitely many contributions
with $0\leq i<N$ using the two pushforward formulas above,
and adds the degrees contributed by all $i\geq N$ using
this explicit description. No arbitrary truncation of the
infinite sum is involved. Repeating the computation for
$-4\leq a\leq4$ and $-6\leq b\leq6$ gives the table
in the proof of \Cref{cor:cohomVanishings}. The script then
checks the thirteen ranges in its nine items, verifying
all the stated vanishings.

\bibliographystyle{alpha}
\bibliography{references}

@article{Seg2014,
  author  = {Addington, Nicolas and Donovan, Will and Segal, Ed},
  title   = {The {P}faffian-{G}rassmannian equivalence revisited},
  journal = {Algebraic Geometry},
  volume  = {2},
  number  = {3},
  pages   = {332--364},
  year    = {2015},
  doi     = {10.14231/AG-2015-015}
}

@article{Ship2010,
  author  = {Shipman, Ian},
  title   = {A geometric approach to Orlov's theorem},
  journal = {Compositio Mathematica},
  year    = {2012},
  volume  = {148},
  number  = {5},
  pages   = {1365--1389},
  doi     = {10.1112/S0010437X12000255},
  url     = {https://doi.org/10.1112/S0010437X12000255}
}

@article{Seg2011,
  author  = {Segal, Ed},
  title   = {Equivalences Between {GIT} Quotients of {L}andau-{G}inzburg {B}-Models},
  journal = {Communications in Mathematical Physics},
  volume  = {304},
  number  = {2},
  pages   = {411--432},
  year    = {2011},
  issn    = {1432-0916},
  doi     = {10.1007/s00220-011-1232-y},
  url     = {https://doi.org/10.1007/s00220-011-1232-y}
}

@article{HL,
  author    = {Halpern-Leistner, Daniel},
  title     = {The derived category of a {GIT} quotient},
  journal   = {Journal of the American Mathematical Society},
  volume    = {28},
  number    = {3},
  pages     = {871--912},
  year      = {2015},
  publisher = {American Mathematical Society},
  doi       = {10.1090/S0894-0347-2014-00815-8},
  url       = {https://doi.org/10.1090/S0894-0347-2014-00815-8}
}

@article {Hoskins,
    AUTHOR = {Hoskins, Victoria},
     TITLE = {Stratifications associated to reductive group actions on
              affine spaces},
   JOURNAL = {Q. J. Math.},
  FJOURNAL = {The Quarterly Journal of Mathematics},
    VOLUME = {65},
      YEAR = {2014},
    NUMBER = {3},
     PAGES = {1011--1047},
      ISSN = {0033-5606,1464-3847},
   MRCLASS = {14L24 (14R20 53D20)},
  MRNUMBER = {3261979},
MRREVIEWER = {Immanuel\ van Santen},
       DOI = {10.1093/qmath/hat046},
       URL = {https://doi.org/10.1093/qmath/hat046},
}

@book{HartshorneLocalCohom,
    author = {Hartshorne, Robin},
    title = {Local Cohomology},
    publisher = {Springer Berlin, Heidelberg},
    year = {1967},
    isbn = {978-3-540-03912-9},
    doi = {https://doi.org/10.1007/BFb0073971},
}

@article{KuzHyperplaneSections,
author = {Kuznetsov, A},
year = {2007},
month = {10},
pages = {447},
title = {Hyperplane sections and derived categories},
volume = {70},
journal = {Izvestiya: Mathematics},
doi = {10.1070/IM2006v070n03ABEH002318}
}

@article{ionue_ito_miura_CY3s,
author = {Inoue, Daisuke and Ito, Atsushi and Miura, Makoto},
year = {2019},
month = {06},
pages = {},
title = {Complete intersection Calabi--Yau manifolds with respect to homogeneous vector bundles on Grassmannians},
volume = {292},
journal = {Mathematische Zeitschrift},
doi = {10.1007/s00209-018-2163-5}
}

@article{KuzExcepColl,
author = {Kuznetsov, Alexander},
year = {2006},
month = {01},
pages = {},
title = {Exceptional collections for Grassmannians of isotropic lines},
volume = {97},
journal = {Proceedings of the London Mathematical Society},
doi = {10.1112/plms/pdm056}
}

@article{KapRam17,
  author  = {Kapustka, Micha{\l} and Rampazzo, Marco},
  title   = {Torelli problem for Calabi--Yau threefolds with {GLSM} description},
  journal = {Communications in Number Theory and Physics},
  year    = {2019},
  volume  = {13},
  number  = {4},
  pages   = {725--761},
  doi     = {10.4310/CNTP.2019.v13.n4.a2}
}

@book{lazarsfeld_positivity_II,
  author    = {Lazarsfeld, Robert},
  title     = {Positivity in Algebraic Geometry~{II}:
               Positivity for Vector Bundles, and Multiplier Ideals},
  series    = {Ergebnisse der Mathematik und ihrer Grenzgebiete. 3. Folge},
  volume    = {49},
  publisher = {Springer-Verlag},
  address   = {Berlin},
  year      = {2004},
  isbn      = {978-3-540-22531-7},
}

@phdthesis{tanturri_thesis,
  author  = {Tanturri, Fabio},
  title   = {On Degeneracy Loci of Morphisms between Vector Bundles},
  school  = {SISSA -- International School for Advanced Studies},
  address = {Trieste},
  year    = {2013},
}

@article{bondalorlovreconstruction,
    author = "Bondal, Alexei and Orlov, Dmitri",
    title = "{Reconstruction of a Variety from the Derived Category and Groups of Autoequivalences}",
    doi = "10.1023/a:1002470302976",
    journal = "Compos. Math.",
    volume = "125",
    number = "3",
    pages = "327--344",
    year = "2001"
}

@article {ottemrennemo,
    AUTHOR = {Ottem, John Christian and Rennemo, J\o rgen Vold},
     TITLE = {A counterexample to the birational {T}orelli problem for
              {C}alabi-{Y}au threefolds},
   JOURNAL = {J. Lond. Math. Soc. (2)},
  FJOURNAL = {Journal of the London Mathematical Society. Second Series},
    VOLUME = {97},
      YEAR = {2018},
    NUMBER = {3},
     PAGES = {427--440},
      ISSN = {0024-6107,1469-7750},
   MRCLASS = {14C34 (14F05 14J32)},
  MRNUMBER = {3816394},
MRREVIEWER = {Zhi\ Jiang},
       DOI = {10.1112/jlms.12111},
       URL = {https://doi.org/10.1112/jlms.12111},
}

@article{borisovcaldararu,
  author  = {Borisov, Lev and C{\u{a}}ld{\u{a}}raru, Andrei},
  title   = {The Pfaffian--Grassmannian Derived Equivalence},
  journal = {Journal of Algebraic Geometry},
  volume  = {18},
  number  = {2},
  pages   = {201--222},
  year    = {2009},
  doi     = {10.1090/S1056-3911-08-00496-7},
  url     = {https://doi.org/10.1090/S1056-3911-08-00496-7}
}

@article{borisovcaldararuperry,
author = {Borisov, Lev and Căldăraru, Andrei and Perry, Alexander},
year = {2018},
month = {05},
pages = {133-162},
title = {Intersections of two Grassmannians in $\mathbb{P}^9$},
volume = {2020},
journal = {Journal für die reine und angewandte Mathematik (Crelles Journal)},
doi = {10.1515/crelle-2018-0014}
}

@article{bridgeland_flops,
  author    = {Bridgeland, Tom},
  title     = {Flops and derived categories},
  journal   = {Inventiones Mathematicae},
  volume    = {147},
  number    = {3},
  pages     = {613--632},
  year      = {2002},
  doi       = {10.1007/s002220100185},
  eprint    = {math/0009053},
  archivePrefix = {arXiv}
}

@article{hosono-takagi,
  author    = {Hosono, Shinobu and Takagi, Hiromichi},
  title     = {Double quintic symmetroids, {R}eye congruences, and their derived equivalence},
  journal   = {Journal of Differential Geometry},
  volume    = {104},
  number    = {3},
  pages     = {443--497},
  year      = {2016},
  doi       = {10.4310/jdg/1478138549},
  eprint    = {1302.5883},
  archivePrefix = {arXiv}
}

@article{imou,
  author    = {Ito, Atsushi and Miura, Makoto and Okawa, Shinnosuke and Ueda, Kazushi},
  title     = {The class of the affine line is a zero divisor in the {G}rothendieck ring: via $G_2$-{G}rassmannians},
  journal   = {Journal of Algebraic Geometry},
  volume    = {28},
  number    = {2},
  pages     = {245--250},
  year      = {2019},
  doi       = {10.1090/jag/731},
  eprint    = {1606.04210},
  archivePrefix = {arXiv}
}

@article{kuznetsov_g2_cy3,
  author    = {Kuznetsov, Alexander},
  title     = {Derived equivalence of {I}to--{M}iura--{O}kawa--{U}eda {C}alabi--{Y}au 3-folds},
  journal   = {Journal of the Mathematical Society of Japan},
  volume    = {70},
  number    = {3},
  pages     = {1007--1013},
  year      = {2018},
  doi       = {10.2969/jmsj/76827682},
  eprint    = {1611.08386},
  archivePrefix = {arXiv}
}

@article{rennemo_hpd_sym2V,
  author    = {Rennemo, J{\o}rgen Vold},
  title     = {The homological projective dual of $\operatorname{Sym}^2\,\mathbb{P}(V)$},
  journal   = {Compositio Mathematica},
  volume    = {156},
  number    = {3},
  pages     = {476--525},
  year      = {2020},
  doi       = {10.1112/s0010437x19007772},
  eprint    = {1509.04107},
  archivePrefix = {arXiv}
}

@article{KapRam2025,
  author  = {Donovan, Will and Hara, Wahei and Kapustka, Micha{\l} and Rampazzo, Marco},
  title   = {Window categories for a simple 9-fold flop of {G}rassmannian type},
  journal = {arXiv preprint arXiv:2510.06184},
  year    = {2025},
  url     = {https://arxiv.org/abs/2510.06184}, 
}

@article{batyrevToricDegen,
      title={Toric Degenerations of Fano Varieties and Constructing Mirror Manifolds}, 
      author={Victor V. Batyrev},
      year={1997},
      journal = {arXiv preprint arXiv:alg-geom/9712034},
      url = {https://arxiv.org/abs/alg-geom/9712034}, 
}

@article{batyrevGKZ,
  author  = {Batyrev, Victor V. and van Straten, Duco},
  title   = {Generalized Hypergeometric Functions and Rational Curves on Calabi--Yau Complete Intersections in Toric Varieties},
  journal = {Communications in Mathematical Physics},
  year    = {1995},
  volume  = {168},
  number  = {3},
  pages   = {493--533},
  doi     = {10.1007/BF02101841},
  eprint  = {alg-geom/9307010},
  archivePrefix = {arXiv}
}

@article{miuraMirrorSymmetry,
author = {Miura, Makoto},
year = {2017},
month = {08},
pages = {},
title = {Minuscule Schubert Varieties and Mirror Symmetry},
journal = {Symmetry, Integrability and Geometry: Methods and Applications},
doi = {10.3842/SIGMA.2017.067}
}

@article{skmsDonovanWemyss,
    author = "Donovan, Will and Wemyss, Michael",
    title = {{Stringy K{\"a}hler moduli, mutation and monodromy}},
    eprint = "1907.10891",
    archivePrefix = "arXiv",
    primaryClass = "math.AG",
    doi = "10.4310/jdg/1736262124",
    journal = "J. Diff. Geom.",
    volume = "129",
    number = "1",
    pages = "115--164",
    year = "2025"
}

@article{skmsHLeistnerSam,
  author  = {Halpern-Leistner, Daniel and Sam, Steven V.},
  title   = {Combinatorial Constructions of Derived Equivalences},
  journal = {Journal of the American Mathematical Society},
  volume  = {33},
  number  = {3},
  pages   = {735--773},
  year    = {2020},
  doi     = {10.1090/jams/940},
  url     = {https://doi.org/10.1090/jams/940}
}

@book{calabiYauBook,
  author    = {Gross, Mark and Joyce, Dominic and Huybrechts, Daniel},
  editor    = {Ellingsrud, Geir and Ranestad, Kristian and Olson, Loren and Str{\o}mme, Stein A.},
  title     = {Calabi--Yau Manifolds and Related Geometries},
  subtitle  = {Lectures at a Summer School in Nordfjordeid, Norway, June 2001},
  publisher = {Springer-Verlag Berlin Heidelberg},
  year      = {2003},
  series    = {Universitext},
  issn      = {0172-5939},
  isbn      = {978-3-540-44059-8},
  doi       = {10.1007/978-3-642-19004-9}
}

@book{voisin,
  author    = {Voisin, Claire},
  title     = {Hodge Theory and Complex Algebraic Geometry I},
  series    = {Cambridge Studies in Advanced Mathematics},
  volume    = {76},
  publisher = {Cambridge University Press},
  year      = {2002},
  translator= {Schneps, Leila},
  isbn      = {0521802601}
}

@article{rodland,
  author  = {R{\o}dland, Einar Andreas},
  title   = {The Pfaffian Calabi--Yau, its Mirror, and their Link to the Grassmannian $G(2,7)$},
  journal = {Compositio Mathematica},
  year    = {2000},
  volume  = {122},
  number  = {2},
  pages   = {135--149},
  doi     = {10.1023/A:1001847914402}
}

@article{hori,
    author = "Herbst, Manfred and Hori, Kentaro and Page, David",
    title = "{Phases Of N=2 Theories In 1+1 Dimensions With Boundary}",
    eprint = "0803.2045",
    archivePrefix = "arXiv",
    primaryClass = "hep-th",
    reportNumber = "DESY-07-154, CERN-PH-TH-2008-048",
    month = "3",
    year = "2008"
}

@article{witten_glsm,
  author  = {Witten, Edward},
  title   = {Phases of $N=2$ theories in two dimensions},
  journal = {Nuclear Physics B},
  volume  = {403},
  number  = {1--2},
  pages   = {159--222},
  year    = {1993},
  doi     = {10.1016/0550-3213(93)90033-L},
  eprint  = {hep-th/9301042},
  archivePrefix = {arXiv}
}

@article{kuznetsov_grassmannians_of_lines,
    author = "Kuznetsov, Alexander",
    title = "{Homological projective duality for Grassmannians of lines}",
    eprint = "math/0610957",
    archivePrefix = "arXiv",
    primaryClass = "math.AG",
    year = "2006"
}

@article{LGFano,
  author  = {Przyjalkowski, Victor},
  title   = {Weak Landau--Ginzburg models for smooth Fano threefolds},
  journal = {Izvestiya: Mathematics},
  year    = {2013},
  volume  = {77},
  number  = {4},
  pages   = {772--794},
  doi     = {10.1070/IM2013v077n04ABEH002660},
  eprint  = {0902.4668},
  archivePrefix = {arXiv}
}

@article{MS_Kap,
author = {Kapustka, Michal},
year = {2013},
month = {10},
pages = {},
title = {Mirror symmetry for Pfaffian Calabi-Yau 3-folds via conifold transitions}
}

@article{fujiki-nakano-contraction-criterion,
  author  = {Fujiki, Akira and Nakano, Shigeo},
  title   = {Supplement to ``On the Inverse of Monoidal Transformation''},
  journal = {Publications of the Research Institute for Mathematical Sciences},
  year    = {1971},
  volume  = {7},
  number  = {3},
  pages   = {637--644}
}

@article{categorical_plucker_formula,
  author  = {Jiang, Qingyuan and Leung, Naichung Conan and Xie, Ying},
  title   = {Categorical Pl{\"u}cker Formula and Homological Projective Duality},
  journal = {Journal of the European Mathematical Society},
  year    = {2021},
  volume  = {23},
  number  = {6},
  pages   = {1859--1898},
  doi     = {10.4171/JEMS/1045},
  eprint  = {1704.01050},
  archivePrefix = {arXiv},
  primaryClass = {math.AG}
}

@article{SSY,
  author  = {Schenck, Hal and Stillman, Mike and Yuan, Beihui},
  title   = {Calabi--Yau threefolds in $\mathbb{P}^n$ and Gorenstein rings},
  journal = {Advances in Theoretical and Mathematical Physics},
  volume  = {26},
  number  = {3},
  pages   = {764--792},
  year    = {2022},
  doi     = {10.4310/ATMP.2022.v26.n3.a7},
  eprint  = {2011.10871},
  archivePrefix = {arXiv},
  primaryClass = {math.AG}
}

@article{Gerhardus_Jockers,
  author  = {Gerhardus, Andreas and Jockers, Hans},
  title   = {Dual Pairs of Gauged Linear Sigma Models and Derived Equivalences of Calabi--Yau Threefolds},
  journal = {Journal of Geometry and Physics},
  year    = {2017},
  volume  = {114},
  pages   = {223--259},
  doi     = {10.1016/j.geomphys.2016.12.005},
  eprint  = {1505.00099},
  archivePrefix = {arXiv},
  primaryClass = {hep-th}
}

@misc{galkin_talk,
    author = "Galkin, Sergey",
    title = "{An explicit construction of Miura's varieties}",
    note = "Talk presented at Tokyo University, Graduate School for Mathematical Sciences, Komaba Campus, February 20, 2014",
    year = "2014"
}

@article{laterveer,
  author  = {Laterveer, Robert},
  title   = {Motives and the {Pfaffian--Grassmannian} equivalence},
  journal = {Journal of the London Mathematical Society},
  volume  = {104},
  number  = {4},
  pages   = {1738--1764},
  year    = {2021},
  doi     = {10.1112/jlms.12473},
  url     = {https://doi.org/10.1112/jlms.12473}
}

@article{martin_grothendieck_ring,
  author  = {Martin, Nicolas},
  title   = {The class of the affine line is a zero divisor in the {Grothendieck} ring: An improvement},
  journal = {Comptes Rendus Math{\'e}matique},
  volume  = {354},
  number  = {9},
  pages   = {936--939},
  year    = {2016},
  doi     = {10.1016/j.crma.2016.05.016},
  url     = {https://doi.org/10.1016/j.crma.2016.05.016}
}

@misc{database,
  author       = {van Straten, Duco and Metelitsyn, Pavel},
  title        = {Calabi--Yau Differential Operator Database},
  howpublished = {\url{https://cydb.mathematik.uni-mainz.de/}},
  note         = {Version 3, Johannes Gutenberg University Mainz,
                  accessed 17 September 2026}
}

@book{kollar_mori,
  author    = {Koll{\'a}r, J{\'a}nos and Mori, Shigefumi},
  title     = {Birational Geometry of Algebraic Varieties},
  series    = {Cambridge Tracts in Mathematics},
  volume    = {134},
  publisher = {Cambridge University Press},
  year      = {1998},
  doi       = {10.1017/CBO9780511662560}
}

@misc{macaulay2,
  author       = {Grayson, Daniel R. and Stillman, Michael E.},
  title        = {Macaulay2, a software system for research in algebraic geometry},
  howpublished = {Available at \url{https://macaulay2.com/}}
}

@article{beilinson,
  author  = {Beilinson, Alexander A.},
  title   = {Coherent Sheaves on $\mathbb{P}^n$ and Problems of Linear Algebra},
  journal = {Functional Analysis and Its Applications},
  volume  = {12},
  number  = {3},
  pages   = {214--216},
  year    = {1978},
  doi     = {10.1007/BF01681436},
  url     = {https://doi.org/10.1007/BF01681436}
}

@unpublished{our_hpd_project,
  author = {Rampazzo, Marco and Mongardi, Giovanni and Martinelli, Luigi
            and Kapustka, Micha{\l} and Samal, Prajwal},
  title  = {Homological projective duality for the Degree-33 {S}chubert Variety
            in the {C}ayley Plane},
  note   = {Work in progress},
  year   = {2026}
}

@article{orlovblowup,
  author  = {Orlov, Dmitri},
  title   = {Projective bundles, monoidal transformations, and derived categories of coherent sheaves},
  journal = {Izvestiya: Mathematics},
  year    = {1993},
  volume  = {41},
  number  = {1},
  pages   = {133--141},
  doi     = {10.1070/IM1993v041n01ABEH002182},
  url     = {https://doi.org/10.1070/IM1993v041n01ABEH002182}
}

@incollection{monodromyCalc,
    AUTHOR = {van Enckevort, Christian and van Straten, Duco},
     TITLE = {Monodromy calculations of fourth order equations of
              {C}alabi-{Y}au type},
 BOOKTITLE = {Mirror symmetry. {V}},
    SERIES = {AMS/IP Stud. Adv. Math.},
    VOLUME = {38},
     PAGES = {539--559},
 PUBLISHER = {Amer. Math. Soc., Providence, RI},
      YEAR = {2006},
      ISBN = {978-0-8218-4251-5; 0-8218-4251-X},
   MRCLASS = {14J32 (14D05 32G20 32S40)},
  MRNUMBER = {2282974},
MRREVIEWER = {Yukiko\ Konishi},
       DOI = {10.1090/amsip/038/23},
       URL = {https://doi.org/10.1090/amsip/038/23},
}

\end{document}